\documentclass[reqno]{amsart}
\usepackage{amssymb}
\usepackage[usenames, dvipsnames]{color}
\usepackage{hyperref}

\usepackage{verbatim}

\numberwithin{equation}{section}

\newtheorem{theorem}{Theorem}[section]
\newtheorem{corollary}[theorem]{Corollary}
\newtheorem{lemma}[theorem]{Lemma}
\newtheorem{proposition}[theorem]{Proposition}

\theoremstyle{definition}
\newtheorem{remark}[theorem]{Remark}

\theoremstyle{definition}

\theoremstyle{definition}
\newtheorem{assumption}[theorem]{Assumption}

\makeatletter
\def\dashint{\operatorname%
{\,\,\text{\bf--}\kern-.98em\DOTSI\intop\ilimits@\!\!}}
\makeatother

\def\bF{\mathbb{F}}
\def\bH{\mathbb{H}}

\def\bR{\mathbb{R}}

\def\cA{\mathcal{A}}
\def\cB{\mathcal{B}}
\def\cC{\mathcal{C}}
\def\cD{\mathcal{D}}

\def\cH{\mathcal{H}}

\def\cM{\mathcal{M}}

\def\cS{\mathcal{S}}

\begin{document}
\title[time fractional parabolic equations]{$L_p$-estimates for time fractional parabolic equations in divergence form with measurable coefficients}

\author[H. Dong]{Hongjie Dong}
\address[H. Dong]{Division of Applied Mathematics, Brown University, 182 George Street, Providence, RI 02912, USA}

\email{Hongjie\_Dong@brown.edu}

\thanks{H. Dong was partially supported by the NSF under agreement DMS-1600593.}

\author[D. Kim]{Doyoon Kim}
\address[D. Kim]{Department of Mathematics, Korea University, 145 Anam-ro, Seongbuk-gu, Seoul, 02841, Republic of Korea}

\email{doyoon\_kim@korea.ac.kr}

\thanks{D. Kim was supported by Basic Science Research Program through the National Research Foundation of Korea (NRF) funded by the Ministry of Education (2016R1D1A1B03934369).}



\begin{abstract}
In this paper, we establish $L_p$ estimates and solvability for time fractional divergence form parabolic equations in the whole space when leading coefficients are merely measurable in one spatial variable and locally have small mean oscillations with respect to the other variables. The corresponding results for equations on a half space are also derived.
\end{abstract}

\maketitle

\section{Introduction}

This paper is concerned with divergence form parabolic equations with a non-local type time derivative term of the form
\begin{equation}
							\label{eq0525_01}
- \partial_t^\alpha u + D_i(a^{ij} D_{j} u+a^i u) + b^i D_i u + c u = D_i g_i + f
\end{equation}
in the whole space $(0,T) \times \bR^d$ or on the half space $(0,T) \times \bR^d_+$, where
$$
\bR^d_+=\{x=(x_1,x')=(x_1,\ldots,x_d)\in \bR^d:x_1>0\},
$$
with the zero initial condition $u(0,\cdot)=0$. Here $\partial_t^\alpha u$ is the Caputo fractional derivative of order $\alpha \in (0,1)$:
$$
\partial_t^\alpha u(t,x) = \frac{1}{\Gamma(1-\alpha)} \frac{d}{dt} \int_0^t (t-s)^{-\alpha} \left[ u(s,x) - u(0,x) \right] \, ds.
$$
We refer the reader to Section \ref{sec2} for a precise definition of $\partial_t^\alpha u$.
In the half space case, we impose the zero Dirichlet boundary condition $u(t,x)=0$ when $t\in (0,T)$, $x_1=0$,  and $x'\in \bR^{d-1}$.
Equations of this type have been used, for example, to model fractional diffusion in plasma turbulence. Recently, there are many interesting work about parabolic equations with non-local time derivatives. For instance, De Giorgi--Nash--Moser type H\"older estimates for time fractional parabolic equations were obtained by Zacher \cite{MR3038123}, and for parabolic equations with fractional operators in both $t$ and $x$ by Allen et al. \cite{MR3488533}. Sobolev type estimates for non-divergence form parabolic equations with non-local time derivatives were investigated in \cite{MR3581300,arXiv:1806.02635}.

The current paper can be viewed as a continuation of \cite{arXiv:1806.02635}, where the corresponding non-divergence form equations in the whole space are studied, when leading coefficients are merely measurable in the time variable and locally have small mean oscillations in the spatial variables. In this paper, we assume that the leading coefficients $a^{ij}$ are uniformly elliptic, not necessarily symmetric, merely measurable in one spatial variable, and locally have small mean oscillations with respect to the other variables. Our main result Theorem \ref{main_thm} reads that, under these conditions, for any given
$$
f,g=(g_1,\ldots,g_d) \in L_p\left((0,T) \times \bR^d \right),
$$
there exists a unique solution $u\in \cH^{\alpha,1}_{p,0}$ to the equation \eqref{eq0525_01} in $(0,T) \times \bR^d$ with the estimate
\begin{align*}
&\|Du\|_{L_p\left((0,T) \times \bR^d \right)}+\|u\|_{L_p\left((0,T) \times \bR^d \right)}
+ \|\partial_t^\alpha u\|_{\bH^{-1}_p\left((0,T) \times \bR^d \right)}\\
&\le N \|f\|_{L_p\left((0,T) \times \bR^d \right)}
+N \|g\|_{L_p\left((0,T) \times \bR^d \right)},
\end{align*}
where the constant $N$ is independent of $f$, $g$, and $u$.
See Section \ref{sec2} for the definitions of various function spaces.
The condition that the coefficients are allowed to be measurable in $x_1$ enables us to deduce the corresponding result in the half space case simply by using the argument of odd/even extensions.
Note that this extension argument cannot be applied if the coefficients are continuous or even have small mean oscillations with respect to all the spatial variables.

The Sobolev theory for parabolic equations with the usual time derivative $u_t$ has been studied extensively in the literature. In view of the well-known counterexamples, see \cite{MR3488249} for the parabolic case, in general
there does not exist a solvability theory for equations with bounded and measurable coefficients. Therefore, many efforts have been made to treat particular types of discontinuous coefficients. An important class of discontinuous coefficients is the class of vanishing mean oscillations (VMO), with which the solvability results in Sobolev spaces for second order linear equations were established in
early 1990s by Italian schools. The main technical tool is the theory of singular
integrals, in particular, the Calder\'on--Zygmund theorem and the Coifman--Rochberg--Weiss commutator theorem. This approach, however, usually does not allow measurable coefficients because one needs smoothness of the
corresponding fundamental solutions.

Among others, there are two alternative approaches which do not involve singular integrals.  In \cite{MR2304157}, Krylov gave a unified approach for both divergence and non-divergence linear elliptic and parabolic equations in the whole space, with coefficients which are in the class of VMO with respect to the space variables and are allowed to be merely measurable in the time variable. See \cite{MR2352490} for mixed-norm $L_p(L_q)$ estimates. His proof relies on mean oscillation estimates, and uses the Hardy--Littlewood maximal function theorem and the Fefferman--Stein sharp function theorem. Another approach was given earlier by Caffarelli and Peral \cite{MR1486629}, which is based on a level set argument together with a ``crawling of ink spots'' lemma originally due to Safonov and Krylov \cite{MR563790,MR579490}. With these approaches, VMO coefficients are treated in a straightforward manner by a perturbation argument. Besides that they are singular integral free, another advantage of these approaches is their flexibility: they can be applied to both divergence and non-divergence or even nonlocal equations with coefficients which are very irregular in some of the independent variables. See, for instance, \cite{MR2069724,MR2338417,MR2300337,MR2800569,MR2764911,MR2863859,arXiv:1806.02635} and the references therein. The class of coefficients considered in this paper was first treated by the second named author and Krylov in \cite{MR2338417,MR2300337} for elliptic and parabolic equations in non-divergence form, and later also in \cite{MR2601069} (without symmetric condition on $a^{ij}$), \cite{MR2680179} (with symmetric $a^{ij}$), and \cite{MR2800569,MR2764911,MR2835999, MR2968240} for elliptic and parabolic equations and systems in divergence form.

In this paper, we adapted the level set argument in \cite{MR1486629} to equations of the form \eqref{eq0525_01} with a nonlocal time derivative term, following the scheme in \cite{arXiv:1806.02635}.
The main difficulty arises in the key step where one needs to estimate local $L_\infty$ estimates of the gradient of solutions to locally homogeneous equations.
Starting from the $L_2$-estimate, which can be derived by integration by parts, and applying the Sobolev type embedding results proved in Appendix, we are only able to show that the gradient are in $L_{p_1}$ for some $p_1>2$ instead of $L_\infty$  (cf. Lemma \ref{lem0823_2}).
Nevertheless, this allows us to obtain the $L_p$ estimate and solvability for any $p\in [2,p_1)$ and $a^{ij}=a^{ij}(x_1)$ by using a modified level set type argument. Then we repeat this procedure and iteratively increase the exponent $p$ for any $p\in [2,\infty)$. In the case when $p\in (1,2)$, we apply a duality argument.
For equations with leading coefficients being measurable in $x_1$ and locally having small mean oscillations in $(t,x')$, we use a perturbation argument by incorporating the small mean oscillations of the coefficients into local mean oscillation estimates of solutions with compact support. After that, the standard partition of unity argument completes the proof. It is worth noting that to apply the Sobolev type embedding results, we need to estimate the $\bH^{\alpha,1}_{2}$ norm of the spatial derivatives of solutions.
Compared to \cite{arXiv:1806.02635}, here the main obstacle is that we cannot estimate the whole gradient $D_{x} u$ in the $\bH^{\alpha,1}_2$ space due to the lack of regularity of $a^{ij}$ in the $x_1$ direction. To this end, we also exploit an idea in \cite{MR2800569} by considering $D_{x'} u$ and a certain linear combination of the first spatial derivatives, instead of $D_{x} u$. See \eqref{eq0823_06} in the proof of Lemma \ref{lem0823_2}.

The remaining part of the paper is organized as follows.
In the next section, we introduce some notation and state the main results of the paper. In Section \ref{Sec3}, we define function spaces for fractional time derivatives and state some of their properties. In Section \ref{sec4}, we obtain the $L_2$ estimate and solvability for equations with coefficients depending only on $x_1$, and then from them we derive certain local estimates, which are used later in the iteration argument. Section \ref{sec5} is devoted to the estimates of level sets of gradient of solutions. We present the proofs of the main theorems in Section \ref{sec6}. In Appendix, we establish several Sobolev type embedding theorems involving time fractional derivatives adapted to our setting, which are of independent interest.

\section{Notation and main results}
							\label{sec2}

\subsection{Notation}
We first introduce some notation used throughout the paper.
For $\alpha \in (0,1)$ and $S \in \bR$, we denote
$$
I^\alpha_S \varphi(t) = \frac{1}{\Gamma(\alpha)} \int_S^t (t-s)^{\alpha - 1} \varphi(s) \, ds
$$
for $\varphi \in L_1(S,\infty)$,
where
$$
\Gamma(\alpha) = \int_0^\infty t^{\alpha - 1} e^{-t} \, dt.
$$
In \cite{MR1544927} $I^\alpha \varphi$ is called the $\alpha$-th integral of $\varphi$ with origin $S$.
In this paper we often write $I^\alpha$ instead of $I_0^\alpha$ for the $\alpha$-th integral with origin $0$. For a sufficiently smooth function $\varphi(t)$, we set
$$
D_t^\alpha \varphi(t) = \frac{d}{dt} I^{1-\alpha}_S \varphi(t) = \frac{1}{\Gamma(1-\alpha)} \frac{d}{dt} \int_{S}^t (t-s)^{-\alpha} \varphi(s) \, ds,
$$
and
\begin{align*}
\partial_t^\alpha \varphi(t) &= \frac{1}{\Gamma(1-\alpha)} \int_{S}^t (t-s)^{-\alpha} \varphi'(s) \, ds\\
&= \frac{1}{\Gamma(1-\alpha)} \frac{d}{dt} \int_{S}^t (t-s)^{-\alpha} \left[ \varphi(s) - \varphi(S) \right] \, ds.
\end{align*}
Note that if $\varphi(S) = 0$, then
$$
D_t^\alpha \varphi =  \partial_t(I_S^{1-\alpha} \varphi) = \partial_t^\alpha \varphi.
$$
Since there is no information about the origin $S$ in the notation $D_t^\alpha$ and $\partial_t^\alpha$, we sometimes write $\partial_t I_S^{1-\alpha}$ in place of $D_t^\alpha$ (or $\partial_t^\alpha$ whenever appropriate) to indicate the origin.

Let $\cD$ be a subset (not necessarily open) of $\bR^k$, $k \in \{1,2, \ldots\}$.
By $\varphi \in C_0^\infty(\cD)$,
we mean that $\varphi$ is infinitely differentiable in $\cD$ and is supported in the intersection of $\cD$ and a bounded open subset in $\bR^d$.
In particular, $\varphi$ may not be zero on the boundary of $\cD$, unless $\cD$ is an open subset of $\bR^k$.
For $\alpha \in (0,1)$, we denote
$$
Q_{R_1,R_2}(t,x) = (t-R_1^{2/\alpha}, t) \times B_{R_2}(x) \quad \text{and} \quad Q_R(t,x)=Q_{R,R}(t,x).
$$
We often write $B_R$ and $Q_R$ instead of $B_R(0)$ and $Q_R(0,0)$, respectively.

In this paper, we assume that there exists $\delta \in (0,1)$ such that
$$
a^{ij}(t,x)\xi_j \xi_j \geq \delta |\xi|^2,\quad |a^{ij}| \leq \delta^{-1}
$$
for any $\xi \in \bR^d$ and $(t,x) \in \bR \times \bR^d$.
For a domain $\Omega$ in $\bR^d$ and $T \in (S, \infty)$, we say that $u \in \cH_{p,0}^{\alpha,1}\left((S,T) \times \Omega \right)$ satisfies the divergence form equation
$$
- \partial_t^\alpha u + D_i \left( a^{ij} D_j u +a^i u\right) + b^i D_i u + c u = D_i g_i + f
$$
in $(S,T) \times \Omega$,
where $g_i, f \in L_p\left((S,T) \times \Omega \right)$ if
\begin{multline}
							\label{eq0103_01}
\int_S^T \int_{\Omega} I_S^{1-\alpha}u \, \varphi_t \, dx \, dt + \int_S^T \int_{\Omega} \left(- a^{ij}D_j u D_i \varphi - a^i u D_i\varphi + b^i D_i u \varphi + c u \varphi \right) \, dx \, dt
\\
= \int_S^T \int_{\Omega} \left(f \varphi - g_i D_i \varphi \right) \, dx \, dt
\end{multline}
for all $\varphi \in C_0^\infty\left([S,T) \times \Omega\right)$. Note that we require $\varphi(T,x) = 0$.
See Section \ref{Sec3} for the definition and some properties of $\cH_{p,0}^{\alpha,1}\left((S,T) \times \Omega\right)$. In particular, the zero initial condition at $t = S$ is implicitly imposed because $u$ belongs to $\cH_{p,0}^{\alpha,1}\left((S,T) \times \Omega\right)$.
For the Dirichlet problem in the half space case,  we impose the zero lateral boundary condition on $u$ when $\Omega = \bR^d_+$.
See Theorem \ref{main_thm}.
Since $I_S^{1-\alpha}u, Du, u \in L_p\left((S,T) \times \Omega\right)$ for $u \in \cH_{p,0}^{\alpha,1}\left((S,T) \times \Omega\right)$, the test function $\varphi$ can be chosen from $W_q^{1,1}\left((S,T) \times \Omega\right)$, $1/p+1/q=1$, satisfying $\varphi(t,x)|_{(S,T) \times \partial\Omega} = 0$ and $\varphi(T,x) = 0$.
The equality \eqref{eq0103_01} also holds with $t_0$ in place of $T$ for any $t_0 \in (S,T)$ if $\varphi(t_0,x) = 0$.

\subsection{Main results}
Our first main result is for equations with coefficients $a^{ij}$ depending only on $x_1$ without any regularity assumptions.

\begin{theorem}
							\label{thm0412_1}
Let $\lambda \geq 0$, $\alpha \in (0,1)$, $T \in (0,\infty)$, $a^{ij} = a^{ij}(x_1)$, and $p \in (1,\infty)$.
There exists $N=N(d,\delta,\alpha,p)$ such that, for any
$u \in \cH_{p,0}^{\alpha,1}(\bR^d_T)$ satisfying
\begin{equation}
							\label{eq0411_03}
-\partial_t^\alpha u + D_i \left( a^{ij}(x_1) D_j u \right) - \lambda u
= D_i g_i + f
\end{equation}
in $\bR^d_T := (0,T) \times \bR^d$, where $g_i \in L_p(\bR^d_T)$, $i=1,\ldots,d$, and $f \in L_p(\bR^d_T)$, we have
\begin{equation}
							\label{eq0411_04}
\|D u\|_{L_p(\bR^d_T)} + \sqrt{\lambda} \|u\|_{L_p(\bR^d_T)} \leq N \|g_i\|_{L_p(\bR^d_T)} + \frac{N}{\sqrt{\lambda}} \|f\|_{L_p(\bR^d_T)},
\end{equation}
provided that $\lambda > 0$.
If $\lambda = 0$, we have
\begin{equation}
							\label{eq0904_06}
\|D u\|_{L_p(\bR^d_T)} + \|u\|_{L_p(\bR^d_T)} \leq N \|g_i\|_{L_p(\bR^d_T)} + N \|f\|_{L_p(\bR^d_T)},
\end{equation}
where $N = N(d,\delta,\alpha,p,T)$.
If $\lambda = 0$ and $f = 0$, we have
\begin{equation}
							\label{eq0905_04}
\|D u\|_{L_p(\bR^d_T)} \leq N \|g_i\|_{L_p(\bR^d_T)},
\end{equation}
where $N=N(d,\delta,\alpha,p)$, but independent of $T$.

Moreover, for any $g_i \in L_p(\bR^d_T)$, $i=1,\ldots,d$, and $f \in L_p(\bR^d_T)$, there exists a unique $u \in \cH_{p,0}^{\alpha,1}(\bR^d_T)$ satisfying \eqref{eq0411_03} and the estimate \eqref{eq0411_04} or \eqref{eq0904_06} depending on whether $\lambda > 0$ or $\lambda =0$.
\end{theorem}

\begin{remark}
                    \label{rem2.2}
From \eqref{eq0411_04} for $\lambda > 0$ and \eqref{eq0904_06} for $\lambda=0$, we also get \eqref{eq0904_06} with $N$ independent of $\lambda$ for $u$ satisfying \eqref{eq0411_03} with any $\lambda\ge 0$. Indeed, for sufficiently small $\lambda\le \varepsilon:=1/(2N)$ with $N$ being the constant in \eqref{eq0904_06} for $\lambda=0$, we can move the term $\lambda u$ to the right-hand side of the equation, and then apply \eqref{eq0904_06} for $\lambda=0$ and absorb the term $N\lambda \|u\|$ on the right-hand side of the inequality. For $\lambda> \varepsilon$, we apply \eqref{eq0411_04} to get
$$
\|D u\|_{L_p(\bR^d_T)} + \sqrt{\varepsilon} \|u\|_{L_p(\bR^d_T)} \leq N \|g_i\|_{L_p(\bR^d_T)} + \frac{N}{\sqrt{\varepsilon}} \|f\|_{L_p(\bR^d_T)}.
$$
\end{remark}

We also consider more general operators with lower-order terms and with coefficients depending on both $t$ and $x$. In this case, we impose the following partially small BMO condition on the leading coefficients.
\begin{assumption}[$\gamma_0$]
                        \label{assump2.2}
There is a constant $R_0\in (0,1]$ such that for each parabolic cylinder $Q_r(t,x)$ with $r\le R_0$ and $(t,x)\in \bR^{d+1}$, we have
$$
\sup_{i,j} \dashint_{\!x_1 - r}^{\,\,\,x_1+r} \dashint_{Q'_r(t,x')} \left|a^{ij}(s,y_1,y') - \bar{a}^{ij}(y_1)\right| \, dy' ds d y_1 \leq \gamma_0,
$$
where
$$
Q'_r(t,x') = (t-r^{2/\alpha},t) \times B_r'(x') = (t-r^{2/\alpha},t) \times \{y' \in \bR^{d-1}: |x'-y'| < r \},
$$
$$
\bar{a}^{ij}(y_1) = \dashint_{Q_r'(t,x)} a^{ij}(\tau,y_1,z') \, dz' \, d\tau.
$$
\end{assumption}

We also assume that there exists a positive constant $K$ such that the lower-order coefficients $a^i$, $b^i$, and $c$ satisfy
\begin{equation}
							\label{eq1221_01}
|a^i| \leq K, \quad |b^i|\le K,\quad |c|\le K.\end{equation}

\begin{theorem}
							\label{main_thm}
Let $\alpha \in (0,1)$, $T \in (0,\infty)$, $p \in (1,\infty)$, and $\Omega = \bR^d$ or $\Omega = \bR^d_+$.
There exists $\gamma_0\in (0,1)$ depending only on $d$, $\delta$, $\alpha$, and $p$, such that, under Assumption \ref{assump2.2} ($\gamma_0$), the following hold. Suppose that $u \in \cH_{p,0}^{\alpha,1}(\Omega_T)$, $\Omega_T = (0,T) \times \Omega$, satisfies
\begin{equation}
							\label{eq0411_03c}
\left\{
\begin{aligned}
-\partial_t^\alpha u + D_i\left( a^{ij} D_j u + a^i u \right) + b^i D_i u + cu &= D_i g_i + f \quad \text{in} \,\,\, \Omega_T
\\
u(t,0,x') =& 0 \quad \text{on} \,\,\, (0,T) \times\partial\Omega \quad \text{if} \,\,\, \Omega = \bR^d_+,
\end{aligned}
\right.
\end{equation}
where $g_i \in L_p(\Omega_T)$, $i=1,\ldots,d$, and $f \in L_p(\Omega_T)$.
Then
\begin{equation}
							\label{eq0411_04c}
\|u\|_{\cH_p^{\alpha,1}(\Omega_T)} \leq N \|g_i\|_{L_{p}(\Omega_T)} + N \|f\|_{L_{p}(\Omega_T)},
\end{equation}
where $N = N(d,\delta,\alpha,p,K, R_0,T)$.
Moreover, for any $g_i \in L_p(\Omega_T)$, $i=1,\ldots,d$, and $f \in L_{p}(\Omega_T)$, there exists a unique $u \in \cH_{p,0}^{\alpha,1}(\Omega_T)$ satisfying \eqref{eq0411_03c} and \eqref{eq0411_04c}.
\end{theorem}

\begin{remark}
When $\Omega = \bR^d_+$ in the above theorem, one can also deal with the conormal derivative boundary value problem.
That is, the Dirichlet boundary condition can be replaced with the conormal derivative boundary condition
$$
a^{1j}D_j u + a^1 u = g_1 \quad \text{on} \quad (0,T) \times \partial\Omega.
$$
In this case, we require \eqref{eq0103_01} to be satisfied for any $\varphi \in C_0^\infty\left([S,T) \times \bar\Omega\right)$.
The proof is almost the same as for the Dirichlet boundary condition case using odd/even extensions.
\end{remark}

\section{Function spaces}
                        \label{Sec3}

Let $\Omega$ be a domain (open and connected, but not necessarily bounded) in $\bR^d$ and $S, T \in (-\infty,\infty)$ such that $S < T$.
In this paper, we consider the parabolic domain $(S,T) \times \Omega$, which is a subset of $\bR^{d+1}$.
If $S=0$, we sometimes denote $(0,T) \times \Omega$ by $\Omega_T$.
In particular, for $\Omega = \bR^d$, we write $\bR^d_T = (0,T) \times \bR^d$.

We first recall the definitions of $\tilde{\bH}_p^{\alpha,k}$, $\bH_p^{\alpha,k}$,
and $\bH_{p,0}^{\alpha,k}$ from \cite{arXiv:1806.02635}.
For $1 \le p \le \infty$, $\alpha \in (0,1)$, and $k \in \{1,2,\ldots\}$, we set
$$
\widetilde{\bH}_p^{\alpha,k}\left((S,T) \times \Omega\right) = \left\{ u \in L_p: \partial_t I_S^{1-\alpha} u, \, D^\beta_x u \in L_p, \, 0 \leq |\beta| \leq k
\right\}
$$
with the norm
\begin{equation*}
\|u\|_{\widetilde{\bH}_p^{\alpha,k}\left((S,T) \times \Omega\right)} = \|\partial_t I_S^{1-\alpha}u\|_{L_p\left((S,T) \times \Omega\right)} + \sum_{0 \leq |\beta| \leq k}\|D_x^\beta u\|_{L_p\left((S,T) \times \Omega\right)},
\end{equation*}
where by $\partial_t I_S^{1-\alpha} u$ we mean that there exists $g \in L_p\left((S,T) \times \Omega\right)$ such that
\begin{equation}
							\label{eq0122_01}
\int_S^T\int_\Omega g(t,x) \varphi(t,x) \, dx \, dt = - \int_S^T\int_\Omega I_S^{1-\alpha}u(t,x) \partial_t \varphi(t,x) \, dx \, dt
\end{equation}
for all $\varphi \in C_0^\infty\left((S,T) \times \Omega\right)$.
Recall that $\partial_t I_S^{1-\alpha} u$ can be denoted by $D_t^\alpha u$.
Next, $\bH_p^{\alpha,k}\left((S,T) \times \Omega\right)$ is defined by
\begin{align*}
&\bH_p^{\alpha,k}\left((S,T) \times \Omega\right)\\
&= \left\{ u \in \widetilde{\bH}_p^{\alpha,k}\left((S,T) \times \Omega\right): \text{\eqref{eq0122_01} is satisfied for all}\,\, \varphi \in C_0^\infty\left([S,T) \times \Omega\right)\right\}
\end{align*}
with the same norm as for $\widetilde{\bH}_p^{\alpha,k}\left((S,T) \times \Omega\right)$.
Note that test functions for the space $\bH_p^{\alpha,k}\left((S,T) \times \Omega\right)$ are not necessarily zero at $t = S$.
We then define
$\bH_{p,0}^{\alpha,k}((S,T)\times\Omega)$
to be functions in $\bH_p^{\alpha,k}((S,T)\times\Omega)$
each of which is
approximated by a sequence $\{u_n(t,x)\} \subset C^\infty\left([S,T]\times \Omega\right)$ such that $u_n$ vanishes for large $|x|$ and $u_n(S,x) = 0$.

By $w \in \bH_p^{-1}\left((S,T) \times \Omega \right)$ we mean that there exist $f, g_i \in L_p\left((S,T) \times \Omega\right)$, $i =1, \ldots, d$, such that
$$
w = D_i g_i + f
$$
in $(S,T) \times \Omega$ in the distribution sense
and
\begin{align*}
&\|w\|_{\bH_p^{-1}\left((S,T) \times \Omega \right)}\\
&= \inf \left\{ \sum_{i=1}^d\|g_i\|_{L_p\left((S,T) \times \Omega\right)} + \|f\|_{L_p\left((S,T) \times \Omega\right)}: w = D_i g_i + f\right\} < \infty.
\end{align*}
We also write
$$
w = \operatorname{div}g + f,
$$
where $g = (g_1,\ldots,g_d)$.

For $u \in L_p\left((S,T) \times \Omega \right)$, we say $D_t^\alpha u \in \bH_p^{-1}\left((S,T) \times \Omega \right)$
if there exist $f, g_i \in L_p\left((S,T) \times \Omega \right)$, $i=1,\ldots, d$, such that, for any $\varphi \in C_0^\infty\left((S,T) \times \Omega\right)$,
\begin{equation}
							\label{eq0720_01}
\int_S^T \int_\Omega I^{1-\alpha}_S u \, \partial_t \varphi \, dx \, dt = \int_S^T \int_\Omega g_i D_i \varphi \, dx \, dt - \int_S^T \int_\Omega f \varphi \, dx \, dt.
\end{equation}

Let $\tilde{\cH}_p^{\alpha, 1}\left((S,T) \times \Omega \right)$ be the collection of functions $u \in L_p\left((S,T) \times \Omega \right)$ such that $D_x u \in L_p\left((S,T) \times \Omega \right)$ and $D_t^\alpha u \in \bH_p^{-1}\left((S,T) \times \Omega \right)$.
For $u \in \tilde{\cH}_p^{\alpha, 1}\left((S,T) \times \Omega \right)$, if \eqref{eq0720_01} holds for any $\varphi \in C_0^\infty\left([S,T) \times \Omega\right)$ whenever \eqref{eq0720_01} holds for any $\varphi \in C_0^\infty\left((S,T) \times \Omega\right)$, we say that $u \in \cH_p^{\alpha,1}\left((S,T) \times \Omega\right)$ with the norm
$$
\|u\|_{\cH_p^{\alpha,1}\left((S,T)\times\Omega\right)} = \|u\|_{L_p\left((S,T) \times \Omega\right)} + \|D_x u\|_{L_p\left((S,T) \times \Omega\right)} + \|D_t^\alpha u\|_{\bH_p^{-1}\left((S,T) \times \Omega\right)}.
$$
We then define $\cH_{p,0}^{\alpha,1}\left((S,T) \times \Omega\right)$ to be the collection of $u \in \cH_p^{\alpha,1}\left((S,T) \times \Omega\right)$ satisfying the following. There exists a sequence of $\{u_n\} \subset C^\infty\left([S,T] \times \Omega\right)$ such that $u_n$ vanishes for large $|x|$,  $u_n(S,x) = 0$, and
\begin{equation}
							\label{eq0723_01}
\|u_n - u\|_{\cH_p^{\alpha,1}\left((S,T) \times \Omega\right)} \to 0
\end{equation}
as $n \to \infty$.

\begin{remark}
                \label{rem3.1}
Similarly as in Lemma 3.1 and Remark 3.4 of \cite{arXiv:1806.02635}, if $\alpha > 1-1/p$, we have
$$
\cH_p^{\alpha, 1}\left((S,T) \times \Omega \right) \subsetneq \tilde{\cH}_p^{\alpha, 1}\left((S,T) \times \Omega \right).
$$
Otherwise, these two spaces coincide.
\end{remark}

\begin{remark}
Without introducing $\tilde{\cH}_p^{\alpha, 1}\left((S,T) \times \Omega \right)$ and $\cH_p^{\alpha, 1}\left((S,T) \times \Omega \right)$, we may define $\cH_{p,0}^{\alpha,1}\left((S,T) \times \Omega\right)$ to be the collection of $u \in L_p\left((S,T) \times \Omega\right)$ satisfying that
\begin{equation}
							\label{eq0906_01}
Du \in L_p\left((S,T) \times \Omega\right), \quad D_t^\alpha u \in \bH_p^{-1}\left((S,T) \times \Omega\right),
\end{equation}
and there exists a sequence of $\{u_n\} \subset C^\infty\left([S,T] \times \Omega\right)$ such that $u_n$ vanishes for large $|x|$,  $u_n(S,x) = 0$, and \eqref{eq0723_01} holds as $n \to \infty$.
If $\cH_{p,0}^{\alpha,1}\left((S,T) \times \Omega\right)$ is defined in this way, we have
$$
\cH_{p,0}^{\alpha,1}\left((S,T) \times \Omega\right) \subset \cH_p^{\alpha,1}\left((S,T) \times \Omega\right).
$$
That is, if $u$ belongs to $\cH_{p,0}^{\alpha,1}\left((S,T) \times \Omega\right)$ which is defined as above, then \eqref{eq0720_01} holds for any $\varphi \in C_0^\infty\left([S,T) \times \Omega\right)$ whenever \eqref{eq0720_01} holds for any $\varphi \in C_0^\infty\left((S,T) \times \Omega\right)$.
To see this, let $u \in L_p\left((S,T) \times \Omega\right)$ satisfy \eqref{eq0906_01}, and let $\{u_n\}$ be a sequence such that $u_n \in C^\infty\left([S,T]\times \Omega\right)$ with $u_n(S,x)=0$, which vanishes for large $|x|$, and \eqref{eq0723_01} holds.
Since $D_t^\alpha u \in \bH_p^{-1}\left((S,T)\times\Omega\right)$, we write
$$
D_t^\alpha u = \operatorname{div} g + f, \quad f, \, g = (g_1, \ldots, g_d) \in L_p\left((S,T)\times\Omega\right).
$$
That is, \eqref{eq0720_01} holds for $\varphi \in C_0^\infty\left((S,T)\times\Omega\right)$.
Since
$$
\|D_t^\alpha u_n- D_t^\alpha u\|_{\bH_p^{-1}\left((S,T) \times \Omega\right)}
\to 0\quad \text{as}\,\,n\to \infty,
$$
there exist $g_n, f_n \in L_p\left((S,T)\times\Omega\right)$ such that
$$
\int_S^T \int_\Omega I^{1-\alpha}_S(u_n-u) \, \partial_t \varphi \, dx \, dt = \int_S^T \int_\Omega g_n \cdot \nabla \varphi \, dx \, dt - \int_S^T \int_\Omega f_n \varphi \, dx \, dt
$$
and
$$
\|g_n\|_{L_p\left((S,T)\times\Omega\right)} + \|f_n\|_{L_p\left((S,T)\times\Omega\right)} \to 0
\quad\text{as}\,\,n \to \infty,
$$
where $g_n = ({g_n}_1,\ldots,{g_n}_d)$.
Set
$$
G_n := g_n + g, \quad F_n := f_n + f.
$$
Then one can write
$$
D_t^\alpha u_n = \operatorname{div} G_n + F_n.
$$
That is,
\begin{equation}
							\label{eq0723_02}
\int_S^T \int_\Omega I^{1-\alpha}_S u_n \partial_t \varphi \, dx \, dt = \int_S^T \int_\Omega G_n \cdot \nabla \varphi \, dx \, dt - \int_S^T \int_\Omega F_n \varphi \, dx \, dt
\end{equation}
for $\varphi \in C_0^\infty\left((S,T)\times\Omega\right)$.
Let $\psi \in C_0^\infty\left([S,T) \times \Omega\right)$.
Since
$$
I_S^{1-\alpha}u_n(S,x) = 0,
$$
we have
\begin{multline}
							\label{eq0723_04}
\int_S^T \int_\Omega I_S^{1-\alpha} u_n \, \partial_t \psi \, dx \, dt = - \int_S^T \int_\Omega \partial_t (I_S^{1-\alpha} u_n) \psi \, dx \, dt
\\
= - \lim_{k \to \infty}\int_S^T \int_\Omega \partial_t \left(I_S^{1-\alpha} u_n\right) \psi \eta_k \, dx \, dt,
\end{multline}
where
$$
\eta \in C^\infty(\bR), \quad
\eta(t) =
\left\{
\begin{aligned}
0 \quad t \leq 0,
\\
1 \quad t \geq 1,
\end{aligned}
\right.
\quad
0 \leq \eta(t) \leq 1,
$$
and
$$
\eta_k(t) = \eta(kt).
$$
From the fact that $\psi \eta_k \in C_0^\infty\left((S,T) \times \Omega\right)$ and \eqref{eq0723_02}, it follows that
$$
- \int_S^T\int_\Omega \partial_t \left(I_S^{1-\alpha} u_n\right) \psi \eta_k \, dx \, dt = \int_S^T \int_\Omega I_S^{1-\alpha} u_n \partial_t \left(\psi \eta_k\right) \, dx \, dt
$$
$$
= \int_S^T \int_\Omega G_n \cdot \nabla(\psi \eta_k) \, dx \, dt - \int_S^T \int_\Omega F_n \psi \eta_k \, dx \, dt
$$
$$
\to \int_S^T \int_\Omega G_n \cdot \nabla \psi \, dx \, dt - \int_S^T \int_\Omega F_n \psi \, dx \, dt
\quad\text{as}\,\,k\to\infty.
$$
This combined with \eqref{eq0723_04} shows that
\begin{equation}
							\label{eq0723_03}
\int_S^T \int_\Omega I_S^{1-\alpha} u_n \, \partial_t \psi \, dx \, dt = \int_S^T \int_\Omega G_n \cdot \nabla\psi \, dx \, dt - \int_S^T \int_\Omega F_n \psi \, dx \, dt
\end{equation}
for $\psi \in C_0^\infty\left([S,T) \times \Omega\right)$.
By the properties of $I_S^{1-\alpha}$ and the fact that
$u_n \to u$ in $L_p\left((S,T) \times \bR^d\right)$, we have
$$
I_S^{1-\alpha}u_n \to I_S^{1-\alpha}u
$$
in $L_p\left((S,T) \times \Omega\right)$.
By letting $n \to \infty$ in \eqref{eq0723_03} and using the fact that $G_n \to g$ and $F_n \to f$ in $L_p\left((S,T) \times \Omega\right)$, we see that
$$
\int_S^T \int_\Omega I_S^{1-\alpha} u  \partial_t \psi \, dx \, dt = \int_S^T \int_\Omega g \cdot \nabla \psi \, dx \, dt - \int_S^T \int_\Omega f \psi \, dx \, dt.
$$
Therefore, \eqref{eq0720_01} holds for $\psi \in C_0^\infty\left([S,T) \times \Omega\right)$.
\end{remark}

\begin{lemma}
							\label{lem0206_1}
Let $p \in [1,\infty)$, $\alpha \in (0,1)$, $-\infty < S < t_0 < T < \infty$, and $u \in \cH_{p,0}^{\alpha,1}\left( (t_0,T) \times \Omega \right)$.
If $u$ is extended to be zero for $t \leq t_0$, denoted by $\bar{u}$, then $\bar{u} \in \cH_{p,0}^{\alpha,1}\left((S,T) \times \Omega\right)$.
\end{lemma}

\begin{proof}
Without loss of generality, we assume $t_0 = 0$ so that
$$
- \infty < S < 0 < T < \infty.
$$
For $u \in \cH_{p,0}^{\alpha,1}(\Omega_T)$,
let $\bar{u}$ be the zero extension of $u$ for $t \leq 0$.
We first check that
$\bar{u} \in \cH_p^{\alpha,1}\left((S,T)\times\Omega\right)$.
It is readily seen that
$\bar{u} \in L_p\left((S,T)\times\Omega\right)$ and
$$
D_x \bar{u} = \left\{
\begin{aligned}
D_x u, \quad 0 \leq t \leq T,
\\
0, \quad S \leq t < 0,
\end{aligned}
\right.
\quad
D_x \bar{u} \in L_p\left((S,T) \times \Omega \right).
$$
Since $D_t^\alpha u \in \bH_p^{-1}(\Omega_T)$, there exists $f, \, g = (g_1,\ldots,g_d) \in L_p(\Omega_T)$ such that, for any $\psi \in C_0^\infty\left([0,T) \times \Omega\right)$,
\begin{equation}
							\label{eq0723_05}
\int_{\Omega_T} I_0^{1-\alpha} u \, \partial_t \psi \, dx \, dt = \int_{\Omega_T} g \cdot \nabla \psi \, dx \, dt - \int_{\Omega_T} f \psi \, dx \, dt,
\end{equation}
where the equality indeed holds for $\psi \in C_0^\infty\left([0,T) \times \Omega\right)$ because $u \in \cH_{p,0}^{\alpha,1}(\Omega_T) \subset \cH_p^{\alpha,1}(\Omega_T)$.
See the definition of $\cH_p^{\alpha,1}(\Omega_T)$ before Remark \ref{rem3.1}.
Set $\bar{g}$ and $\bar{f}$ to be the zero extensions of $g$ and $f$ for $t\leq 0$.
Then, since
$$
I_S^{1-\alpha} \bar{u} = \left\{
\begin{aligned}
I_0^{1-\alpha} u, &\quad t \in [0,T],
\\
0, &\quad t \in [S,0),
\end{aligned}
\right.
$$
it follows from \eqref{eq0723_05} that
$$
\int_S^T\int_\Omega I_S^{1-\alpha} \bar{u} \, \partial_t \varphi \, dx \, dt = \int_{\Omega_T} I_0^{1-\alpha} u \, \partial_t \varphi \, dx \, dt
$$
$$
= \int_{\Omega_T} g \cdot \nabla \varphi \, dx \, dt - \int_{\Omega_T} f \varphi \, dx \, dt = \int_S^T\int_\Omega \bar{g} \cdot \nabla \varphi \, dx \, dt - \int_S^T\int_\Omega \bar{f} \varphi \, dx \, dt
$$
for any $\varphi \in C_0^\infty\left([S,T) \times\Omega\right)$.
This shows that
\begin{equation}
							\label{eq0724_01}
D_t^\alpha \bar{u} = \partial_t I_S^{1-\alpha} \bar{u} = \operatorname{div} \bar{g} + \bar{f}
\end{equation}
in $(S,T) \times \Omega$  and $\bar{u} \in \cH_p^{\alpha,1}\left((S,T) \times \Omega\right)$.

Let $\{u_n\}$ be an approximating sequence of $u$ such that $u_n \in C^\infty\left([0,T] \times \Omega\right)$, $u_n$ vanishes for large $|x|$, $u_n(0,x) = 0$, and
$$
\|u_n - u\|_{\cH_p^{\alpha,1}\left((0,T) \times \Omega\right)} \to 0
\quad\text{as}\,\,n \to \infty.
$$
Extend $u_n$ to be zero for $t \leq 0$, denoted by $\bar{u}_n$.
As is shown above, we have $\bar{u}_n \in \cH_p^{\alpha, 1}\left((S,T)\times\Omega\right)$.
Clearly,
\begin{equation}
							\label{eq0723_06}
\|\bar{u}_n - \bar{u}\|_{L_p\left((S,T)\times\Omega\right)} + \|D_x\bar{u}_n - D_x\bar{u}\|_{L_p\left((S,T)\times\Omega\right)} \to 0
\quad\text{as}\,\,n \to \infty.
\end{equation}
We now check
\begin{equation}
							\label{eq0907_01}
\|D_t^\alpha \bar{u}_n - D_t^\alpha \bar{u}\|_{\bH_p^{-1}\left((S,T)\times\Omega\right)} \to 0
\quad\text{as}\,\,n \to \infty.
\end{equation}
Since
$$
\|u_n - u\|_{\cH_{p,0}^{\alpha,1}(\Omega_T)} \to 0 \quad\text{as}\,\,n \to \infty,
$$
there exist $g_n, f_n \in L_p(\Omega_T)$ such that
$$
D_t^\alpha (u_n-u) = \partial_t I_0^{1-\alpha}(u_n-u) = \operatorname{div} g_n + f_n
$$
in $\Omega_T$
and
$$
\|g_n\|_{L_p(\Omega_T)} + \|f_n\|_{L_p(\Omega_T)} \to 0
\quad\text{as}\,\,n \to \infty.
$$
Then by the reasoning above, we have
\begin{equation}
							\label{eq0724_02}
D_t^\alpha(\bar{u}_n - \bar{u}) = \partial_t I_S^{1-\alpha}(\bar{u}_n - \bar{u}) = \operatorname{div}\bar{g}_n + \bar{f}_n
\end{equation}
in $(S,T) \times \Omega$, where $\bar{g}_n, \bar{f}_n$ are the zero extensions of $g_n, f_n$ for $t \leq 0$.
Clearly,
$$
\|\bar{g}_n\|_{L_p\left((S,T)\times\Omega\right)} + \|\bar{f}_n\|_{L_p\left((S,T)\times\Omega\right)}
 \to 0
 \quad\text{as}\,\,n \to \infty.
$$
This shows \eqref{eq0907_01}, which along with \eqref{eq0723_06} proves that
\begin{equation}
							\label{eq0724_05}
\| \bar{u}_n - \bar{u} \|_{\cH_p^{\alpha,1}\left((S,T)\times\Omega\right)} \to 0
\quad\text{as}\,\,n \to \infty.
\end{equation}
Set
$$
G_n = \bar{g}_n + \bar{g}, \quad F_n = \bar{f}_n + \bar{f}, \quad G_n, F_n \in L_p\left((S,T) \times \Omega\right).
$$
From \eqref{eq0724_01} and \eqref{eq0724_02} we have
\begin{equation}
							\label{eq0724_04}
D_t^\alpha \bar{u}_n = \partial_t I_S^{1-\alpha} \bar{u}_n = \operatorname{div} G_n + F_n
\end{equation}
in $(S,T) \times \Omega$.

We now mollify $\bar{u}_n$ to obtain an approximating sequence in $C^\infty\left([S,T] \times \Omega\right)$ for $\bar{u}$ so that  we finally check $\bar{u} \in \cH_{p,0}^{1, \alpha}\left((S,T)\times \Omega\right)$.
Let $\eta(t)$ be an infinitely differentiable function defined in $\bR$ satisfying $\eta \ge 0$, $\eta(t) = 0$ outside $(0,1)$, and	
$$
\int_\bR \eta(t) \, dt = 1.
$$
Set
$$
\bar{u}_n^{(\varepsilon)}(t,x) = \int_{-\infty}^T \eta_{\varepsilon}(t-s) \bar{u}_n(s,x) \, ds, \quad \eta_\varepsilon(t) = \frac{1}{\varepsilon^{2/\alpha}} \eta(t/\varepsilon^{2/\alpha}).
$$
Then it follows easily that $\bar{u}_n^{(\varepsilon)}(t,x) \in C^\infty\left([S,T] \times \Omega\right)$, $\bar{u}_n^{(\varepsilon)}(S,x) = 0$, $\bar{u}_n^{(\varepsilon)}(t,x)$ vanishes for large $|x|$, and
\begin{equation}
							\label{eq0907_03}
D_x \bar{u}_n^{(\varepsilon)}(t,x) = \int_{-\infty}^T \eta_{\varepsilon}(t-s) D_x \bar{u}_n(s,x)\, ds
\end{equation}
for $(t,x) \in (S,T) \times \Omega$.
Moreover,
\begin{equation}
							\label{eq0724_03}
D_t^\alpha \bar{u}_n^{(\varepsilon)}(t,x) = \operatorname{div} G_n^{(\varepsilon)} + F_n^{(\varepsilon)},
\end{equation}
where
$$
G_n^{(\varepsilon)}(t,x) = \int_{-\infty}^T \eta_\varepsilon(t-s) G_n(s,x) \, ds = \int_S^T \eta_\varepsilon(t-s) G_n(s,x) \, ds, $$
$$
F_n^{(\varepsilon)}(t,x) = \int_{-\infty}^T \eta_\varepsilon(t-s) F_n(s,x) \, ds = \int_S^T \eta_\varepsilon(t-s) F_n(s,x) \, ds.
$$
To verify \eqref{eq0724_03}, we first check that
$$
I_S^{1-\alpha} \bar{u}_n^{(\varepsilon)} = \left( I_S^{1-\alpha}\bar{u}_n \right)^{(\varepsilon)}
$$
in $(S,T) \times \Omega$.
Indeed, using the fact that $\eta(r) = 0$ if $r \leq 0$, for $(t,x) \in (S,T) \times \Omega$, we have
$$
\Gamma(1-\alpha) I_S^{1-\alpha} \bar{u}_n^{(\varepsilon)}(t,x) = \int_S^t (t-s)^{-\alpha} \bar{u}_n^{(\varepsilon)}(s,x) \, ds
$$
$$
= \int_S^t (t-s)^{-\alpha} \int_{-\infty}^T \eta_\varepsilon(s-r) \bar{u}_n(r,x) \, dr \, ds
$$
$$
= \int_S^t (t-s)^{-\alpha} \int_S^s \eta_\varepsilon(s-r) \bar{u}_n(r,x) \, dr \, ds
$$
$$
= \int_S^t \int_r^t (t-s)^{-\alpha} \eta_\varepsilon(s-r) \bar{u}_n(r,x) \, ds \, dr
$$
$$
= \int_S^t \int_r^t (s-r)^{-\alpha} \eta_\varepsilon(t-s) \bar{u}_n(r,x) \, ds \, dr
$$
$$
= \int_S^t \int_S^s (s-r)^{-\alpha} \eta_\varepsilon(t-s) \bar{u}_n(r,x) \, dr \, ds
$$
$$
= \int_S^t \eta_\varepsilon(t-s) \int_S^s (s-r)^{-\alpha} \bar{u}_n(r,x)\, dr \, ds
$$
$$
= \Gamma(1-\alpha) \int_S^t \eta_\varepsilon(t-s) I_S^{1-\alpha} \bar{u}_n(s,x) \, ds = \Gamma(1-\alpha) \left(I_S^{1-\alpha}\bar{u}_n\right)^{(\varepsilon)}(t,x).
$$
Then, for $\varphi \in C_0^\infty\left((S,T) \times \Omega\right)$,
\begin{multline}
							\label{eq0907_02}
\int_S^T \int_\Omega I_S^{1-\alpha} \bar{u}_n^{(\varepsilon)} \, \partial_t \varphi \, dx \, dt = \int_S^T \int_\Omega \left( I_S^{1-\alpha} \bar{u}_n\right)^{(\varepsilon)} \partial_t \varphi \, dx \, dt
\\
= \int_S^T \int_\Omega \int_{-\infty}^T \eta_\varepsilon(t-s) \left(I_S^{1-\alpha}\bar{u}_n\right)(s,x) \, ds \, \partial_t \varphi(t,x) \, dx \, dt
\\
= \int_S^T \int_\Omega \int_S^T \eta_\varepsilon(t-s) \left(I_S^{1-\alpha}\bar{u}_n\right)(s,x) \, ds \, \partial_t \varphi(t,x) \, dx \, dt
\\
= \int_S^T \int_\Omega \left(I_S^{1-\alpha}\bar{u}_n\right)(s,x) \int_S^T \eta_\varepsilon(t-s) \partial_t \varphi(t,x) \, dt \, dx \, ds
\\
= \int_S^T \int_\Omega \left(I_S^{1-\alpha}\bar{u}_n\right)(s,x) \, \partial_s \Phi(s,x) \, dx \, ds,
\end{multline}
where
$$
\Phi(s,x): = \int_S^T \eta_\varepsilon(t-s) \varphi(t,x) \, dt \in C_0^\infty\left([S,T) \times \Omega \right).
$$
In particular, $\Phi(T,x) = 0$ because $\eta_\varepsilon(t-T) = 0$ for $t \in [S,T]$.
Using \eqref{eq0907_02} and \eqref{eq0724_04}, we see that
$$
\int_S^T \int_\Omega I_S^{1-\alpha} \bar{u}_n^{(\varepsilon)} \, \partial_t \varphi \, dx \, dt = \int_S^T \int_\Omega G_n \cdot \nabla \Phi \, d x \, ds - \int_S^T \int_\Omega F_n \Phi \, dx \, ds
$$
$$
= \int_S^T \int_\Omega G_n^{(\varepsilon)} \cdot \nabla \varphi \, dx \, dt  - \int_S^T \int_\Omega F_n^{(\varepsilon)} \varphi \, dx \, dt.
$$
This justifies \eqref{eq0724_03}.
By \eqref{eq0907_03}, \eqref{eq0724_03}, the properties of mollifications, and \eqref{eq0724_05}, we have
$$
\|\bar{u}_n^{(\varepsilon)} - \bar{u}\|_{\cH_p^{\alpha,1}\left((S,T)\times\Omega\right)} \to 0
$$
as $n \to \infty$ and $\varepsilon \to 0$.
Therefore, $\bar{u} \in \cH_{p,0}^{\alpha,1}\left((S,T) \times \Omega\right)$.
\end{proof}

\begin{lemma}
							\label{lem0207_1}
Let $p \in [1,\infty)$, $\alpha \in (0,1)$, $k \in \{1,2,\ldots\}$, $-\infty < S \leq t_0 < T < \infty$, and $v \in \cH_{p,0}^{\alpha,1}\left((S,T) \times \Omega \right)$.
Then, for any infinitely differentiable function $\eta$ defined on $\bR$ such that $\eta(t)=0$ for $t \leq t_0$ and
$$
|\eta'(t)| \le M, \quad t \in \bR,
$$
the function $\eta v$ belongs to $\cH_{p,0}^{\alpha,1}\left( (t_0,T) \times \Omega \right)$ and in the sense of distribution
\begin{equation}
							\label{eq0727_01}
\partial_t^\alpha (\eta v)(t,x) = \partial_t I_{t_0}^{1-\alpha} (\eta
v)(t,x) = \eta(t) \partial_t I_S^{1-\alpha} v (t,x) - F(t,x)
\end{equation}
in $(t_0,T) \times \Omega$,
where
\begin{equation}
							\label{eq0207_04}
F(t,x): = \frac{\alpha}{\Gamma(1-\alpha)} \int_S^t (t-s)^{-\alpha-1} \left(\eta(s) - \eta(t)\right) v(s,x) \, ds
\end{equation}
satisfies
\begin{equation}
							\label{eq0207_01}
\|F\|_{L_p\left((t_0,T) \times \Omega\right)} \le N(\alpha, p, M, T, S) \|v\|_{L_p\left( (S,T) \times \Omega\right)}.
\end{equation}
\end{lemma}

\begin{remark}
							\label{rem0910_1}
In Lemma \ref{lem0207_1} above if $\eta v$ belongs to $\bH_{p,0}^{\alpha,2}\left((t_0,T) \times \Omega\right)$, by \eqref{eq0727_01}
$$
\eta(t) \partial_t I_S^{1-\alpha} v \in L_p\left((t_0,T) \times \Omega\right).
$$
In this case \eqref{eq0727_01} holds a.e. in $(t_0,T) \times \Omega$.
\end{remark}

\begin{proof}[Proof of Lemma \ref{lem0207_1}]
As in the proof of Lemma \ref{lem0206_1}, we assume that $t_0 = 0$.
First we check \eqref{eq0207_01}.
Note that since $|\eta'(t)| \leq M$, we have
\begin{align*}
&\left| \int_S^t (t-s)^{-\alpha-1} \left(\eta(t) - \eta(s) \right) v(s,x) \, ds \right|\\
&\leq M \int_S^t (t-s)^{-\alpha}|v(s,x)| \, ds = M \Gamma(1-\alpha) I^{1-\alpha}_S |v(t,x)|
\end{align*}
for $(t,x) \in \Omega_T$.
Hence, the inequality \eqref{eq0207_01} follows from  Lemma A.2 in \cite{arXiv:1806.02635} with $1-\alpha$ in place of $\alpha$ (also see Remark A.3 in \cite{arXiv:1806.02635}).

By $\eta(t) \partial_t I_S^{1-\alpha} v (t,x)$ we mean that, if
$$
D_t^\alpha v = \partial_t I_S^{1-\alpha} v = \operatorname{div} g + f
$$
in $(S,T) \times \Omega$, where $g, f \in L_p\left((S,T) \times \Omega \right)$, then
\begin{equation}
							\label{eq0907_05}
\eta(t) \partial_t I_S^{1-\alpha} v = \operatorname{div} (\eta g) + \eta f
\end{equation}
in $(S,T) \times \Omega$ so that, for any $\psi \in C_0^\infty\left((S,T) \times \Omega\right)$, we have
$$
\int_S^T \int_\Omega I_S^{1-\alpha} v \, \partial_t (\eta(t) \psi) \, dx \, dt = \int_S^T \int_\Omega (\eta g) \cdot \nabla \psi \, dx \, dt - \int_S^T \int_\Omega \eta f \psi \, dx \, dt.
$$

We now prove that $\eta v \in \cH_{p,0}^{\alpha,1}(\Omega_T)$ and \eqref{eq0727_01}.
Since $v \in \cH_{p,0}^{\alpha,1}\left((S,T) \times \Omega\right)$, there is a sequence $\{v_n\}$ such that $v_n \in C^\infty\left([S,T] \times \Omega\right)$, $v_n(S,x) = 0$, $v_n(t,x)$ vanishes for large $|x|$, and
$$
\|v_n - v\|_{\cH_{p,0}^{\alpha,1}\left((S,T)\times\Omega\right)} \to 0
\quad\text{as}\,\,n \to \infty
$$
Also let $g, f, g_n, f_n \in L_p\left((S,T)\times\Omega\right)$ be functions such that
$$
\partial_t^\alpha v = D_t^\alpha v = \partial_t I_S^{1-\alpha} v = \operatorname{div} g + f,
$$
$$
\partial_t^\alpha v_n = D_t^\alpha v_n = \partial_t I_S^{1-\alpha} v_n = \operatorname{div} g_n + f_n,
$$
$$
\|g_n - g\|_{L_p\left((S,T)\times\Omega\right)} + \|f_n - f\|_{L_p\left((S,T)\times\Omega\right)} \to 0
\quad\text{as}\,\,n \to \infty
$$
Observe that, for $(t,x) \in \Omega_T$,
$$
I_0^{1-\alpha} (\eta v_n) = \eta(t) I_S^{1-\alpha} v_n + I_0^{1-\alpha} (\eta v_n) - \eta(t) I_S^{1-\alpha} v_n
$$
$$
= \eta(t) I_S^{1-\alpha} v_n + \bF_n(t,x),
$$
where
$$
\bF_n(t,x) := I_0^{1-\alpha} (\eta v_n) - \eta(t) I_S^{1-\alpha} v_n.
$$
Using the fact that $\eta(t) = 0$ for $t \leq 0$, we note that
$$
\bF_n(t,x) = \frac{1}{\Gamma(1-\alpha)}\int_S^t (t-s)^{-\alpha} \left(\eta(s) - \eta(t)\right) v_n(s,x) \, ds,
$$
from which we see that
\begin{multline}
							\label{eq0907_04}
\partial_t \bF_n(t,x) = -\frac{\alpha}{\Gamma(1-\alpha)} \int_S^t (t-s)^{-\alpha-1} \left(\eta(s) - \eta(t)\right) v_n(s,x) \, dx
\\
- \frac{\eta'(t)}{\Gamma(1-\alpha)} \int_S^t (t-s)^{-\alpha} v_n(s,x)\, ds = - F_n(t,x) - \eta'(t) I_S^{1-\alpha} v_n,
\end{multline}
where $F_n$ is defined as in \eqref{eq0207_04} with $v_n$ in place of $v$.
Then, for $\varphi \in C_0^\infty\left([0,T) \times \Omega\right)$,
$$
\int_{\Omega_T} I_0^{1-\alpha} (\eta v_n) \, \partial_t \varphi \, dx \, dt = \int_{\Omega_T} \left(\eta(t) I_S^{1-\alpha} v_n + \bF_n \right) \partial_t \varphi \, dx \, dt
$$
$$
= \int_S^T I_S^{1-\alpha} v_n \, \partial_t (\eta \varphi) \, dx \, dt + \int_{\Omega_T} \left( \bF_n \partial_t \varphi - I_S^{1-\alpha} v_n \, \eta'(t) \varphi \right) \, dx \, dt =: J_1 + J_2,
$$
where
$$
J_1 = \int_S^T \left(g_n \cdot \nabla (\eta \varphi) - f_n \eta \varphi\right) \, dx \, dt = \int_{\Omega_T} \left( (\eta g_n) \cdot \nabla \varphi - \eta f_n \varphi \right) \, dx \, dt
$$
and, by \eqref{eq0907_04},
$$
J_2 = - \int_{\Omega_T} \left( \partial_t \bF_n + I_S^{1-\alpha} v_n \eta'(t) \right) \varphi \, dx \, dt
= \int_{\Omega_T} F_n \varphi \, dx \, dt.
$$
This shows that
$$
\partial_t I_0^{1-\alpha} (\eta v_n) = \operatorname{div} (\eta g_n) + \eta f_n - F_n.
$$
We then see that
$$
\|\eta u_n - \eta u\|_{L_p(\Omega_T)} + \|D_x(\eta u_n) - D_x(\eta u)\|_{L_p(\Omega_T)}
$$
$$
+ \|\eta g_n - \eta g\|_{L_p(\Omega_T)} + \|\eta f_n - \eta f\|_{L_p(\Omega_T)} + \|F_n - F\|_{L_p(\Omega_T)} \to 0
\quad\text{as}\,\,n \to \infty
$$
and
$$
D_t^\alpha (\eta v) = \partial_t^\alpha (\eta v) = \partial_t I_0^{1-\alpha} (\eta v) = \operatorname{div} (\eta g) + \eta f - F.
$$
Therefore, $\eta v \in \cH_{p,0}^{\alpha,1}(\Omega_T)$ and \eqref{eq0727_01} follows upon noting \eqref{eq0907_05}.
The lemma is proved.
\end{proof}

\section{Auxiliary results}
                                    \label{sec4}

In Lemma \ref{lem0904_1} below, $a^{ij} = a^{ij}(t,x)$ satisfy only the ellipticity condition without any regularity assumptions.

\begin{lemma}[$L_p$ energy estimate]
							\label{lem0904_1}
Let $p \in [2,\infty)$, $T \in (0,\infty)$, and $u \in \cH_{p,0}^{\alpha,1}(\bR^d_T)$ satisfy
\begin{equation}
							\label{eq0904_02}
-\partial_t I_0^{1-\alpha} u + D_i \left(a^{ij} D_j u \right)
= D_i g_i + f
\end{equation}
in $\bR^d_T$, where $g_i, f \in L_p(\bR^d_T)$.
Then for any $\tau \in (0,T]$,
we have
\begin{equation}
							\label{eq1209_01}
\sup_{0 < t < \tau}\| I_0^{1-\alpha}|u|^p(t,\cdot) \|_{L_1(\bR^d)} \leq N \tau^{\alpha(p-2)/2} \|g_i\|^p_{L_p(\bR^d_\tau)} + N \tau^{\alpha(p-1)} \|f\|^p_{L_p(\bR^d_\tau)},
\end{equation}
where $N=N(d,\delta,\alpha,p)$.
We also have
\begin{equation}
							\label{eq0903_01}
\|u\|_{L_p(\bR^d_T)} \leq N T^{\alpha/2} \|g_i\|_{L_p(\bR^d_T)}
+ N T^\alpha \|f\|_{L_p(\bR^d_T)},
\end{equation}
where $N = N(d,\delta,\alpha, p)$.
\end{lemma}

\begin{proof}
We first prove that, for $u \in C_0^\infty\left([0,T] \times \bR^d\right)$ such that $u(0,x) = 0$,
\begin{equation}
							\label{eq0904_01}
J:= u(t,x) |u(t,x)|^{p-2} \cdot \partial_t I_0^{1-\alpha} u(t,x)
- \frac{1}{p} \partial_t I_0^{1-\alpha} |u|^p(t,x) \ge 0.
\end{equation}
To see this, for fixed $t\in (0,T)$ and $x\in \bR^d$, let
$$
F(s)=\frac 1 p (|u(s,x)|^p-|u(t,x)|^p)-(u(s,x)-u(t,x))u(t,x)|u(t,x)|^{p-2}
$$
and
$$
F_1(s)=\frac 1 p |u(s,x)|^p,\quad F_2(s)=u(s,x)u(t,x)|u(t,x)|^{p-2}.
$$
By the convexity of the function $|y|^p$, more precisely, by the property $
h(s)\ge h(t)+(s-t)h'(t)$ for a convex function $h$, we see that $F(s)\ge 0$ on $[0,T]$ with the equality at $s=t$. This and integration by parts clearly yield that
$$
\int_0^t (t-s)^{-\alpha}(F_1'(s)-F_2'(s))\,ds
=\int_0^t (t-s)^{-\alpha}F'(s)\,ds\le 0,
$$
which implies \eqref{eq0904_01} because $F_1(0)=F_2(0)=0$.

For $u \in \cH_{p,0}^{\alpha,1}(\bR^d)$ satisfying \eqref{eq0904_02}, we find $u_n \in C_0^\infty\left([0,T] \times \bR^d\right)$ such that $u_n(0,x) = 0$ and $u_n \to u$ in $\cH_{p,0}^{\alpha,1}(\bR^d_T)$.
Then there exist $g_n, f_n \in L_p(\bR^d_T)$ such that
\begin{equation}
							\label{eq0904_03}
-\partial_t I_0^{1-\alpha} u_n + D_i \left( a^{ij} D_j u_n \right)
= \operatorname{div} g_n + f_n
\end{equation}
in $\bR^d_T$ and
$$
\|g - g_n\|_{L_p(\bR^d_T)} + \|f_n - f\|_{L_p(\bR^d_T)} \to 0
\quad\text{as}\,\,n \to \infty.
$$
By multiplying \eqref{eq0904_03} by $u_n|u_n|^{p-2}$, integrating over $(0,\tau) \times \bR^d$, $\tau \in (0,T]$,
and using \eqref{eq0904_01}, we obtain that
\begin{multline}
							\label{eq0905_01}
\frac{1}{p}\int_{\bR^d_\tau} \partial_t I_0^{1-\alpha} |u_n|^p (t,x) \, dx \, dt + \int_{\bR^d_\tau} a^{ij} D_j u_n D_i \left( u_n |u_n|^{p-2} \right) \, dx \, dt
\\
\leq \int_{\bR^d_\tau} g_n \cdot \nabla \left(u_n |u_n|^{p-2}\right) \, dx \, dt - \int_{\bR^d_\tau} f_n u_n |u_n|^{p-2} \, dx \, dt.
\end{multline}
Since
\begin{equation*}
\int_{\bR^d_\tau} \partial_t I_0^{1-\alpha} |u_n|^p(t,x) \, dx \, dt = \int_{\bR^d}\left(I_0^{1-\alpha}|u_n|^p\right)(\tau,x) \, dx,
\end{equation*}
using the ellipticity condition and Young's inequality, from \eqref{eq0905_01} we have that, for any $\varepsilon_1, \varepsilon_2, \varepsilon_3 > 0$,
$$
\|I_0^{1-\alpha}|u_n|^p(\tau,\cdot)\|_{L_1(\bR^d)} +
\int_{\bR^d_\tau} |Du_n|^2 |u_n|^{p-2} \, dx \, dt
$$
$$
\leq N \varepsilon_1^{-1} \int_{\bR^d_\tau} |g_n|^2 |u_n|^{p-2} \, dx \, dt + \varepsilon_1 \int_{\bR^d_\tau} |Du_n|^2 |u_n|^{p-2} \, dx \, dt
$$
$$
+ N \varepsilon_2^{1-p} \int_{\bR^d_\tau} |f_n|^p \, dx \, dt + \varepsilon_2 \int_{\bR^d_\tau} |u_n|^p \, dx \, dt,
$$
where
$$
\int_{\bR^d_\tau} |g_n|^2 |u_n|^{p-2} \, dx \, dt \leq N\varepsilon_3^{(2-p)/2} \int_{\bR^d_\tau} |g_n|^p \, dx \, dt + \varepsilon_3 \int_{\bR^d_\tau} |u_n|^p \, dx \, dt
$$
and $N = N(d,\delta,p)$.
Then choose $\varepsilon_1 = 1/2$ so that we have
\begin{multline}
							\label{eq0904_04}
\|I_0^{1-\alpha}|u_n|^p(\tau,\cdot)\|_{L_1(\bR^d)} \leq N \varepsilon_3^{(2-p)/2} \int_{\bR^d_\tau} |g_n|^p \, dx \, dt
\\
+ N \varepsilon_2^{1-p} \int_{\bR^d_\tau} |f_n|^p \, dx \, dt + (\varepsilon_2 + N \varepsilon_3) \int_{\bR^d_\tau} |u_n|^p \, dx \, dt
\end{multline}
for any $\varepsilon_2, \varepsilon_3 > 0$, where $N = N(d,\delta,p)$.
We then note that
\begin{multline}
							\label{eq0904_05}
\int_{\bR^d_\tau} |u_n(s,x)|^p \, ds \, dx \leq \tau^\alpha \int_{\bR^d} \int_0^\tau (\tau-s)^{-\alpha} |u_n(s,x)|^p \, ds \, dx
\\
= \Gamma(1-\alpha)\tau^\alpha \int_{\bR^d} \left(I_0^{1-\alpha} |u_n|^p\right)(\tau,x) \, dx.
\end{multline}
Using this and \eqref{eq0904_04} with suitable $\varepsilon_2, \varepsilon_3 > 0$,
we have
\begin{equation}
							\label{eq1209_02}
\|I_0^{1-\alpha}|u_n|^p(\tau,\cdot)\|_{L_1(\bR^d)} \leq N \tau^{\alpha(p-2)/2} \|g_n\|_{L_p(\bR^d_\tau)}^p + N \tau^{\alpha(p-1)} \|f_n\|_{L_p(\bR^d_\tau)}^p.
\end{equation}
This implies that, for each $\tau \in (0,T]$, the sequence $\{u_n\}$ is Cauchy in the norm
$$
\|v\|_{L_{p,(\tau-\cdot)^{-\alpha}}\left((0,\tau); L_p(\bR^d)\right)} := \left(\int_{\bR^d} \int_0^\tau |v(s,x)|^p (\tau-s)^{-\alpha} \, ds \, dx \right)^{1/p}.
$$
Since
$$
\|v\|_{L_p\left((0,\tau) \times \bR^d\right)} \leq \tau^\alpha \|v\|_{L_{p,(\tau-\cdot)^{-\alpha}}\left((0,\tau); L_p(\bR^d)\right)}
$$
and $u_n \to u$ in $L_p\left((0,\tau) \times \bR^d\right)$, we conclude that, for each $\tau \in (0,T]$,
$$
\|I_0^{1-\alpha}|u|^p(\tau,\cdot)\|_{L_1(\bR^d)} \leq N \tau^{\alpha(p-2)/2} \|g\|_{L_p(\bR^d_\tau)}^p + N \tau^{\alpha(p-1)} \|f\|_{L_p(\bR^d_\tau)}^p.
$$
This proves \eqref{eq1209_01}.

From \eqref{eq1209_02} and \eqref{eq0904_05} with $\tau = T$, we also have
$$
\int_{\bR^d} \int_0^T |u_n(s,x)|^p \, ds \,dx \leq N T^{\alpha p/2} \|g_n\|_{L_p(\bR^d_T)}^p + N T^{\alpha p} \|f_n\|_{L_p(\bR^d_T)}^p,
$$
where $N = N(d,\delta,\alpha,p)$.
By letting $n \to \infty$, we obtain \eqref{eq0903_01}.
\end{proof}

To prove Theorem \ref{thm0412_1}, we first prove the theorem for $p = 2$ in the proposition below.
In fact, when $p=2$ Theorem \ref{thm0412_1} holds for $a^{ij} = a^{ij}(t,x)$ with no regularity assumptions as in Lemma \ref{lem0904_1}.

\begin{proposition}
							\label{prop0720_1}
Theorem \ref{thm0412_1} holds when $p=2$.
\end{proposition}

\begin{proof}
A version of this result without $\lambda u$ term can be found in \cite{MR2538276}.
For the reader's convenience, we present here a detailed proof.

We prove the a priori estimates \eqref{eq0411_04}, \eqref{eq0904_06}, and \eqref{eq0905_04}.
Once these estimates are available, one can use the method of continuity and the solvability of a simple equation such as
$$
-\partial_t^\alpha u + \Delta u = D_i g_i + f
$$
in $\bR^d_T$.
Indeed, to solve this equation in $\cH_{2,0}^{\alpha,1}(\bR^d_T)$, one can use the results for non-divergence form equations in \cite{arXiv:1806.02635} and the a priori estimates \eqref{eq0411_04}, \eqref{eq0904_06}, and \eqref{eq0905_04}. See also Remark \ref{rem2.2}.

Let us first consider the case $\lambda > 0$.
Since $u \in \cH_{2,0}^{\alpha,1}(\bR^d_T)$, there exists a sequence $\{u_n\}$ such that $u_n \in C_0^\infty\left([0,T]\times \bR^d\right)$, $u_n(0,x) = 0$, and
$$
\|u_n - u\|_{\cH_2^{\alpha,1}(\bR^d_T)} \to 0
$$
as $n \to \infty$.
In particular, since
$$
\|\partial_t I_0^{1-\alpha} u_n - \partial_t I_0^{1-\alpha}u\|_{\bH_2^{-1}(\bR^d_T)} \to 0,
$$
there exist $G_n = \left({G_n}_1,\ldots,{G_n}_d\right) \in L_2(\bR^d_T)$ and $F_n \in L_2(\bR^d_T)$ such that
$$
\partial_t I_0^{1-\alpha} u_n - \partial_t I_0^{1-\alpha} u = \operatorname{div} G_n + F_n
$$
and
$$
\|G_n\|_{L_2(\bR^d_T)} + \|F_n\|_{L_2(\bR^d_T)} \to 0
\quad\text{as}\,\,n \to \infty.
$$
Then
$$
- \partial_t I_0^{1-\alpha} u_n = -\partial_t I_0^{1-\alpha} u - \operatorname{div} G_n - F_n
$$
$$
= D_i \left(g_i - a^{ij} D_j u - {G_n}_i \right) + \lambda u + f - F_n.
$$
Hence,
$$
-\partial_t I_0^{1-\alpha} u_n + D_i\left(a^{ij} D_j u_n\right) - \lambda u_n=\operatorname{div} g_n + f_n
$$
in $\bR^d_T$, where
$$
{g_n}_i = a^{ij} D_j(u_n - u) + g_i - {G_n}_i \to g_i, \quad f_n = \lambda(u-u_n) + f - F_n \to f
$$
in $L_2(\bR^d_T)$.
Multiplying both sides of the above equation by $u_n$ and integrating by parts, we have
$$
\int_{\bR^d_T} \left(\partial_t I_0^{1-\alpha}u_n\right) u_n \, dx \, dt + \int_{\bR^d_T} a^{ij} D_j u_n D_i u_n \, dx \, dt + \lambda \int_{\bR^d_T} u_n^2 \, dx \, dt
$$
$$
= \int_{\bR^d_T} g_n \cdot \nabla u_n \, dx \, dt - \int_{\bR^d_T} f_n u_n \, dx \, dt.
$$
By \eqref{eq0904_01} with $p=2$ we have
$$
\int_{\bR^d_T} \left(\partial_t I_0^{1-\alpha}u_n\right) u_n \, dx \, dt \ge 0.
$$
It then follows from the ellipticity condition and Young's inequality that
$$
\|D u_n\|_{L_2(\bR^d_T)} + \sqrt{\lambda} \|u_n\|_{L_2(\bR^d_T)} \leq N \|g_n\|_{L_2(\bR^d_T)} + \frac{N}{\sqrt{\lambda}}\|f_n\|_{L_2(\bR^d_T)},
$$
where $N = N(d,\delta)$.
By letting $n \to \infty$, we obtain \eqref{eq0411_04}.

To prove \eqref{eq0904_06} and \eqref{eq0905_04}, we consider
$$
-\partial_t^\alpha u + D_i \left(a^{ij} D_ju\right) - \varepsilon u = D_i g_i + f - \varepsilon u
$$
in $\bR^d_T$, where $\varepsilon > 0$.
Then by the estimate \eqref{eq0411_04}, we have
$$
\|Du\|_{L_p(\bR^d_T)} \leq N \|g_i\|_{L_p(\bR^d_T)} + \frac{N}{\sqrt{\varepsilon}}\|f\|_{L_p(\bR^d_T)} + N \sqrt{\varepsilon}\|u\|_{L_p(\bR^d_T)}.
$$
If $f = 0$, then by letting $\varepsilon \to 0$, we obtain \eqref{eq0905_04}.
Otherwise, set $\varepsilon = 1$ and combine the above estimate with \eqref{eq0903_01} for $p=2$.
\end{proof}

\begin{lemma}[Local estimate for divergence form equations]
							\label{lem0731_1}
Let $\lambda \geq 0$, $p_0\in (1,\infty)$, $\alpha \in (0,1)$, $T \in (0,\infty)$, and $0 < r < R < \infty$.
If Theorem \ref{thm0412_1} holds with this $p_0$ and $u \in \cH_{p_0,0}^{\alpha,1}\left((0, T) \times B_R\right)$ satisfies
$$
-\partial_t^\alpha u + D_i \left(a^{ij}(x_1) D_j u \right) - \lambda u
= D_i g_i + f,
$$
in $(0,T) \times B_R$, where $g_i, f \in L_{p_0}\left((0,T)\times B_R\right)$,
then
\begin{multline*}
\| D u\|_{L_{p_0}\left((0,T) \times B_r\right)} + \sqrt{\lambda}\|u\|_{L_{p_0}\left((0,T) \times B_r\right)}
\le \frac{N}{R-r} \|u\|_{L_{p_0}\left((0,T) \times B_R\right)}
\\
+ N \|g_i\|_{L_{p_0}\left((0,T) \times B_R\right)} + \frac{N(R-r)}{\sqrt{\lambda(R-r)^2+1}} \|f\|_{L_{p_0}\left((0,T) \times B_R\right)},
\end{multline*}
where $N = N(d,\delta,\alpha,p_0)$.
\end{lemma}

\begin{proof}
Set
$$
r_0 = r, \quad r_k = r+(R-r)\sum_{j=1}^k \frac{1}{2^j}, \quad k = 1, 2, \ldots.
$$
Let $\zeta_k = \zeta_k(x)$ be an infinitely differentiable function defined on $\bR^d$ such that
$$
\zeta_k = 1 \quad \text{on} \quad B_{r_k}, \quad \zeta_k = 0 \quad \text{outside} \quad \bR^d \setminus B_{r_{k+1}},
$$
and
$$
|D_x \zeta_k(x)| \le N(d)\frac{2^k}{R-r}.
$$
Then $u \zeta_k$ satisfies
$$
-\partial_t^\alpha(u \zeta_k) + D_i\left(a^{ij} D_j(u \zeta_k) \right) - \lambda (u \zeta_k)
$$
$$
= D_i \left( a^{ij} u D_j \zeta_k + g_i \zeta_k \right) + a^{ij} D_i \zeta_k D_j u - g_i D_i \zeta_k + f \zeta_k
$$
in $\bR^d_T$.
For each non-negative integer $k$, let $\{\lambda_k\}$ be an increasing sequence of positive numbers to be specified below.
We then write
$$
-\partial_t^\alpha(u \zeta_k) + D_i\left(a^{ij} D_j(u \zeta_k) \right) - (\lambda + \lambda_k) (u \zeta_k)
$$
$$
= D_i \left( a^{ij} u D_j \zeta_k + g_i \zeta_k \right) + a^{ij} D_i \zeta_k D_j u - g_i D_i \zeta_k + f \zeta_k - \lambda_k (u \zeta_k)
$$
$$
=: \operatorname{div} g_k + f_k,
$$
where
$$
{g_k}_i = a^{ij} u D_j \zeta_k + g_i \zeta_k, \quad f_k = a^{ij} D_i \zeta_k D_j u - g_i D_i \zeta_k + f \zeta_k - \lambda_k u \zeta_k.
$$
By Theorem \ref{thm0412_1} we have
$$
\|D(u\zeta_k)\|_{L_{p_0}(\bR^d_T)} + \sqrt{\lambda + \lambda_k} \|u\zeta_k\|_{L_{p_0}(\bR^d_T)}
$$
$$
\leq N \|g_k\|_{L_{p_0}(\bR^d_T)} + \frac{N}{\sqrt{\lambda + \lambda_k}} \|f_k\|_{L_{p_0}(\bR^d_T)},
$$
where $N$ is the constant in \eqref{eq0411_04}.
Note that
$$
\|g_k\|_{L_{p_0}(\bR^d_T)} \leq N \frac{2^k}{R-r} \|u\|_{L_{p_0}\left((0,T) \times B_R\right)} + N \|g\|_{L_{p_0}\left((0,T) \times B_R\right)},
$$
$$
\|f_k\|_{L_{p_0}(\bR^d_T)} \leq N \frac{2^k}{R-r} \|D(u \zeta_{k+1})\|_{L_{p_0}(\bR^d_T)} + N \frac{2^k}{R-r} \|g\|_{L_{p_0}\left((0,T) \times B_R\right)}
$$
$$
+ N\|f\|_{L_{p_0}\left((0,T) \times B_R\right)} + N \lambda_k \|u\|_{L_{p_0}\left((0,T) \times B_R\right)}.
$$
Hence,
$$
\|D(u\zeta_k)\|_{L_{p_0}(\bR^d_T)} + \sqrt{\lambda + \lambda_k} \|u\zeta_k\|_{L_{p_0}(\bR^d_T)}
$$
$$
\leq N \frac{2^k}{R-r} \|u\|_{L_{p_0}\left((0,T) \times B_R\right)} + N \|g\|_{L_{p_0}\left((0,T) \times B_R\right)}
$$
$$
+ \frac{N}{\sqrt{\lambda + \lambda_k}} \frac{2^k}{R-r} \|D(u \zeta_{k+1})\|_{L_{p_0}(\bR^d_T)} + \frac{N}{\sqrt{\lambda + \lambda_k}} \frac{2^k}{R-r} \|g\|_{L_{p_0}\left((0,T) \times B_R\right)}
$$
$$
+ \frac{N}{\sqrt{\lambda + \lambda_k}}\|f\|_{L_{p_0}\left((0,T) \times B_R\right)} + \frac{N}{\sqrt{\lambda + \lambda_k}} \lambda_k \|u\|_{L_{p_0}\left((0,T) \times B_R\right)}.
$$
Furthermore, by taking the same $N_0 = N_0(d, \delta,\alpha,p_0) \geq 1$, we have
$$
\|D(u\zeta_k)\|_{L_{p_0}(\bR^d_T)} + \sqrt{\lambda} \|u\zeta_k\|_{L_{p_0}(\bR^d_T)}
$$
$$
\leq N_0 \left(\frac{2^k}{R-r} + \sqrt{\lambda_k}\right) \|u\|_{L_{p_0}\left((0,T) \times B_R\right)} + N_0 \left(1+ \frac{1}{\sqrt{\lambda_k}} \frac{2^k}{R-r} \right)\|g\|_{L_{p_0}\left((0,T) \times B_R\right)}
$$
$$
+ \frac{N_0}{\sqrt{\lambda + \lambda_0}}\|f\|_{L_{p_0}\left((0,T) \times B_R\right)} + \frac{N_0}{\sqrt{\lambda_k}} \frac{2^k}{R-r} \|D(u \zeta_{k+1})\|_{L_{p_0}(\bR^d_T)}.
$$
Multiply both sides by $\varepsilon^k$ and make summations with respect to $k = 0,1,2, \ldots$ to get
$$
\sum_{k=0}^\infty \varepsilon^k \|D(u\zeta_k)\|_{L_{p_0}(\bR^d_T)} + \sqrt{\lambda} \sum_{k=0}^\infty \varepsilon^k \|u\zeta_k\|_{L_{p_0}(\bR^d_T)}
$$
$$
\leq N_0 \|u\|_{L_{p_0}\left((0,T) \times B_R\right)} \sum_{k=0}^\infty \varepsilon^k \left(\frac{2^k}{R-r} + \sqrt{\lambda_k}\right)
$$
$$
+ N_0 \|g\|_{L_{p_0}\left((0,T) \times B_R\right)} \sum_{k=0}^\infty \varepsilon^k \left(1+ \frac{1}{\sqrt{\lambda_k}} \frac{2^k}{R-r} \right) + \frac{N_0}{\sqrt{\lambda + \lambda_0}} \|f\|_{L_{p_0}\left((0,T) \times B_R\right)} \sum_{k=0}^\infty \varepsilon^k
$$
$$
+ N_0 \sum_{k=0}^\infty \frac{\varepsilon^k 2^k}{\sqrt{\lambda_k}(R-r)}\|D(u \zeta_{k+1})\|_{L_{p_0}(\bR^d_T)}.
$$
Now we set
$$
\varepsilon = 2^{-3}, \quad \sqrt{\lambda_k} = \frac{N_0 2^k}{\varepsilon(R-r)}, \quad k = 0, 1, \ldots,
$$
so that
$$
N_0 \frac{\varepsilon^k 2^k}{\sqrt{\lambda_k}(R-r)} = \varepsilon^{k+1}.
$$
Then
$$
\|D(u\zeta_0)\|_{L_{p_0}(\bR^d_T)} + \sum_{k=1}^\infty \varepsilon^k \|D(u\zeta_k)\|_{L_{p_0}(\bR^d_T)} + \sqrt{\lambda} \|u \zeta_0\|_{L_{p_0}(\bR^d_T)}
$$
$$
\leq \frac{4N_0(8N_0+1)}{3(R-r)} \|u\|_{L_{p_0}\left((0,T) \times B_R\right)} + \frac{8 N_0 + 1}{7} \|g\|_{L_{p_0}\left((0,T)\times B_R\right)}
$$
$$
+ \frac{8 N_0}{7\sqrt{\lambda + \lambda_0}}\|f\|_{L_{p_0}\left((0,T) \times B_R\right)} + \sum_{k=1}^\infty \varepsilon^k \|D(u \zeta_k)\|_{L_{p_0}(\bR^d_T)}.
$$
We now remove the same summation terms from both sides of the above inequality and use the fact that (recall that $N_0 \geq 1$)
$$
\frac{1}{\sqrt{\lambda + \lambda_0}} \leq \frac{R-r}{\sqrt{\lambda(R-r)^2 + 1}}.
$$
Finally, by observing that, for instance,
$$
\|Du\|_{L_{p_0}\left((0,T)\times B_r\right)} \leq \|D(u\zeta_0)\|_{L_{p_0}(\bR^d_T)},
$$
we obtain the desired inequality in the lemma.
\end{proof}

\begin{lemma}[Local estimate for non-divergence form equations with variable coefficients in $x$]
							\label{lem0824_1}
Let $p \in (1,\infty)$, $T \in (0,\infty)$, and $0 < r < R < \infty$.
Let $a^{ij}$ be uniformly continuous and the coefficients $b^i(t,x)$ and $c(t,x)$ be bounded by $K$.
Suppose that $u \in \bH_{p,0}^{\alpha,2}\left((0,T) \times B_R\right)$ satisfies
$$
- \partial_t I_0^{1-\alpha} u + a^{ij}(x) D_{ij} u + b^i(x) D_i u + c(x) u = f
$$
in $(0,T) \times B_R$.
Then
\begin{align*}
&\| \partial_t I_0^{1-\alpha} u \|_{L_p\left((0,T) \times B_r\right)} + \| D^2 u\|_{L_p\left((0,T) \times B_r\right)} \\
&\le N \left( (R-r)^{-2} + 1 \right) \|u\|_{L_p\left((0,T) \times B_R\right)} + N \|f\|_{L_p\left((0,T) \times B_R\right)},
\end{align*}
where $N$ depends only on $d$, $\delta$, $\alpha$, $p$, $K$, $T$, and the modulus of continuity of $a^{ij}$.
\end{lemma}

\begin{proof}
Since $a^{ij}$ are uniformly continuous, we can use Theorem 2.4 in \cite{arXiv:1806.02635}.
Then the lemma is proved in the same way as the proof of Lemma 4.2 in \cite{arXiv:1806.02635} is done.
\end{proof}

\begin{lemma}
							\label{lem0814_1}
Let $p \in (1,\infty)$, $0 < T < \infty$, $0 < r < R < \infty$, and $m \in \{1,2,\ldots\}$.
Assume that $a^{ij}(x)$, $b^i(x)$, and $c(x)$ are infinitely differentiable with bounded derivatives.
Suppose that $v \in \bH_{p,0}^{\alpha,2}\left((0,T) \times B_R\right)$ satisfies
$$
- \partial_t I_0^{1-\alpha} v + a^{ij}(x) D_{ij} v + b^i(x) D_i v + c(x) v = f
$$
in $(0,T) \times B_R$,
where
$$
f, D_x f, \ldots, D_x^m f \in L_p \left((0,T) \times B_R\right).
$$
Then, for $k = 1, \ldots, m$,
$$
D_x^k u \in \bH_{p,0}^{\alpha,2}\left((0,T) \times B_r\right).
$$
\end{lemma}

\begin{proof}
By induction, we prove only the case $k=1$.
By moving $b^i B_i v + c v$ to the right-hand side of the equation and noting that $D_x\left(b^i D_i v + c v\right) \in L_p\left((0,T) \times B_R\right)$, we may assume that $b^i = c = 0$.
Let $r_0, r_1 \in (r,R)$ such that $r_0 < r_1$.
By Lemma 4.3 in \cite{arXiv:1806.02635}, for $\varepsilon \in (0,R-r_1)$,
$$
v^{(\varepsilon)}, D_x v^{(\varepsilon)} \in \bH_{p,0}^{\alpha,2}\left((0,T) \times B_{r_1}\right),
$$
where $v^{(\varepsilon)}$ is a mollification of $v$ with respect to the spatial variables, that is,
$$
v^{(\varepsilon)}(t,x) = \int_{B_R} \phi_\varepsilon(x-y) v(t,y) \, dy, \quad \phi_\varepsilon(x) = \varepsilon^{-d} \phi(x/\varepsilon),
$$
and $\phi \in C_0^\infty(B_1)$ is a smooth non-negative function with unit integral.
One can check that
\begin{equation}
							\label{eq0824_05}
(\partial_t I_0^{1-\alpha} v)^{(\varepsilon)} = \partial_t I_0^{1-\alpha} (v^{(\varepsilon)})
\end{equation}
in $(0,T) \times B_{r_1}$, which follows from
$$
\left(I_0^{1-\alpha} v\right)^{(\varepsilon)} = I_0^{1-\alpha} (v^{(\varepsilon)})
$$
in $(0,T) \times B_{r_1}$.
Indeed,
\begin{align*}
&\Gamma(1-\alpha) I^{1-\alpha} (v^{(\varepsilon)}) = \int_0^t (t-s)^{-\alpha} \int_{B_R} \phi_\varepsilon(x-y) v(s,y) \, dy \, ds\\
&= \int_{B_R} \phi_\varepsilon(x-y) \int_0^t (t-s)^{-\alpha} v(s,y) \, ds \, dy = \Gamma(1-\alpha) \left(I_0^{1-\alpha} v\right)^{(\varepsilon)}.
\end{align*}
Using \eqref{eq0824_05} we see that $v^{(\varepsilon)}$ satisfies
$$
-\partial_t I_0^{1-\alpha} v^{(\varepsilon)} +a^{ij} D_{ij} v^{(\varepsilon)} = f^{(\varepsilon)} + a^{ij} D_{ij} v^{(\varepsilon)} - \left(a^{ij}D_{ij}v\right)^{(\varepsilon)}
$$
in $(0,T) \times B_{r_1}$.
By differentiating both sides of the above equation in $x$, we also see that $D_x v^{(\varepsilon)}$ satisfies
\begin{equation}
							\label{eq0824_04}
- \partial_t I_0^{1-\alpha} (D_x v^{(\varepsilon)}) + a^{ij} D_{ij} (D_x v^{(\varepsilon)}) = (Df)^{(\varepsilon)} + g^\varepsilon
\end{equation}
in $(0,T) \times B_{r_1}$,
where
\begin{align*}
g^\varepsilon &:= a^{ij} D D_{ij} v^{(\varepsilon)} - D \left(a^{ij}D_{ij}v\right)^{(\varepsilon)}\\
&= \int_{B_R} D \left[\phi_\varepsilon(x-y)\right]\left( a^{ij}(x) - a^{ij}(y) \right) D_{ij} v(t,y) \, dy,
\end{align*}
so that
$$
|g^\varepsilon(t,x)|\le N \int_{B_R} |D\phi((x-y)/\varepsilon)|\varepsilon^{-d-1}|x-y| |D_{ij} v(t,y)| \, dy
$$
and by the Minkowski inequality,
$$
\|g^\varepsilon\|_{L_p((0,T) \times  B_{r_1})} \leq N \|D^2 v\|_{L_p((0,T) \times B_R)}.
$$
By Lemma \ref{lem0824_1} and the last inequality above, we have
$$
\|D^2 D v^{(\varepsilon)}\|_{L_p((0,T) \times B_{r_0})} \leq N \|D v^{(\varepsilon)}\|_{L_p((0,T) \times B_{r_1})} + N \|(Df)^{(\varepsilon)}\|_{L_p((0,T) \times B_{r_1})}
$$
$$
+ N \|D^2 v\|_{L_p((0,T) \times B_R)},
$$
where the right-hand side is bounded independent of $\varepsilon \in (0,R - r_1)$ because $Dv, Df \in L_p((0,T) \times B_R)$.
This implies that
$$
D^2 D v \in L_p((0,T) \times B_{r_0}).
$$
Using this fact, we rewrite $g^\varepsilon$ as
\begin{align*}
g^\varepsilon &= a^{ij}(x) \int_{B_R} \phi_\varepsilon(x-y) DD_{ij} v(t,y) \, dy - \int_{B_R} \phi_\varepsilon(x-y) D a^{ij}(y) D_{ij} v(t,y) \, dy\\
&\quad - \int_{B_R} \phi_\varepsilon(x-y) a^{ij}(y) DD_{ij} v(t,y) \, dy,
\end{align*}
which converges to $- Da^{ij} D_{ij}v$ in $L_p((0,T) \times B_{r'})$, $r' = (r+r_0)/2$.
In particular, $g^\varepsilon$ is Cauchy in $L_p((0,T) \times B_{r'})$.
Since $Dv, Df \in L_p((0,T) \times B_R)$, we also know that $D v^{(\varepsilon)}$ and $(Df)^{(\varepsilon)}$ are Cauchy in $L_p((0,T) \times B_{r'})$ for $\varepsilon \in (0,R - r_1)$.
Then applying Lemma \ref{lem0824_1} to the equation \eqref{eq0824_04} in $(0,T) \times B_{r'}$, we see that
$D_x v^{(\varepsilon)}$ is Cauchy in $\bH_{p,0}^{\alpha,2}((0,T) \times B_r)$.
Since $D_x v^{(\varepsilon)} \to D_x v$ in $L_p((0,T) \times B_r)$, we conclude that $D_x v \in \bH_{p,0}^{\alpha,2}((0,T) \times B_r)$.
\end{proof}

\begin{lemma}
							\label{lem0820_1}
Let $p\in (1,\infty)$, $\alpha \in (0,1)$, $-\infty< S \leq t_0 < T < \infty$, and $0 < r < R < \infty$.
Also let $\eta(t)$ be an infinitely differentiable function defined on $\bR$ such that $\eta(t)=0$ for $t \leq t_0$.
Assume that $a^{ij}(x)$, $b^i(x)$, and $c(x)$ are infinitely differentiable with bounded derivatives.
If $v \in \cH_{p,0}^{\alpha,1}\left((S, T) \times B_R\right)$, $\eta v \in \bH_{p,0}^{\alpha,2}\left((t_0, T) \times B_R\right)$,
and $\eta v$ satisfies the non-divergence form equation
\begin{equation}
							\label{eq0815_01}
-\partial_t I_{t_0}^{1-\alpha}(\eta v) + a^{ij}(x)D_{ij}(\eta v) + b^i(x) D_i (\eta v)+ c(x) (\eta v)
= F
\end{equation}
in $(t_0,T) \times B_R$,
where $F$ is defined as in \eqref{eq0207_04},
then the following hold.

\begin{enumerate}
\item $\eta(t) \partial_t I_S^{1-\alpha} v$ belongs to $L_p\left((t_0,T) \times B_R\right)$ and satisfies
\begin{equation}
							\label{eq0820_01}
\eta(t) \partial_t I_S^{1-\alpha} v = a^{ij} D_{ij}(\eta v) + b^i D_i(\eta v) + c (\eta v)
\end{equation}
a.e. in $(0,T) \times B_R$.

\item $D_x(\eta v) \in \bH_{p,0}^{\alpha,2}\left( (t_0,T) \times B_r \right)$,
$\eta(t)\partial_t I_S^{1-\alpha} v \in \cH_{p,0}^{\alpha,1}\left((t_0,T) \times B_r\right)$, and
$$
\partial_t I_{t_0}^{1-\alpha} \left(\eta(t) \partial_t I_S^{1-\alpha} v\right) \in \bH_p^{-1}\left((t_0,T) \times B_r\right)
$$
satisfies
\begin{equation}
							\label{eq0820_06}
\partial_t I_{t_0}^{1-\alpha} \left(\eta(t) \partial_t I_S^{1-\alpha} v\right) = D_i g_i + f
\end{equation}
in $(0,T) \times B_r$, where
$$
g_i = a^{ij} \partial_t I_{t_0}^{1-\alpha} \left(D_j(\eta v) \right),
$$
$$
f = - D_i a^{ij} \partial_t I_{t_0}^{1-\alpha} D_j (\eta v) + b^i \partial_t I_{t_0}^{1-\alpha} D_i (\eta v) + c \partial_t I_{t_0}^{1-\alpha} (\eta v).
$$
\end{enumerate}
\end{lemma}

\begin{proof}
Without loss of generality we assume   $t_0 = 0$ so that
$$
- \infty < S \leq 0 < T < \infty.
$$
Because $\eta v \in \bH_{p,0}^{\alpha,2}\left((0,T) \times B_R\right)$, by Lemma \ref{lem0207_1} (also see Remark \ref{rem0910_1}), we have
$$
\eta(t) \partial_t I_S^{1-\alpha} v = \partial_t I_0^{1-\alpha} (\eta v) + F \in L_p\left((0,T)\times B_R\right).
$$
Then, \eqref{eq0815_01} is equivalent to \eqref{eq0820_01}.

Since $v \in \cH_{p,0}^{\alpha,1}\left((S,T)\times B_R\right)$, in particular, $Dv \in L_p\left((S,T) \times B_R\right)$, we have $D_x F \in L_p\left((0,T) \times B_R\right)$.
Indeed,
$$
D_x F(t,x) = \frac{\alpha}{\Gamma(1-\alpha)} \int_S^t (t-s)^{-\alpha-1} \left( \eta(s) - \eta(t) \right) D_x v(s,x) \, ds.
$$
From this, Lemma \ref{lem0814_1}, and the fact that $\eta v$ satisfies the non-divergence equation \eqref{eq0815_01},
we have
\begin{equation*}
D_x(\eta v) \in \bH_{p,0}^{\alpha,2}\left((0,T) \times B_{r_1}\right),
\end{equation*}
where $r_1 = (R+r)/2$.
Clearly, we also have $D_x(\eta v) \in \bH_{p,0}^{\alpha,2}\left((0,T) \times B_r\right)$.

Now we show that $\eta(t) \partial_t I_S^{1-\alpha} v \in \cH_{p,0}^{\alpha,1}\left(0,T) \times B_r\right)$ and \eqref{eq0820_06}.
Set $w := \eta v$ and
$u := \eta(t) \partial_t I_S^{1-\alpha}v$.
Since $w \in \bH_{p,0}^{\alpha,2}\left((0,T) \times B_R\right)$, there exists a sequence $\{w_n\} \subset C^\infty\left([0,T] \times B_R\right)$ such that $w_n(0,x) = 0$ and
\begin{equation}
							\label{eq1217_01}
\|w_n - w\|_{\bH_p^{\alpha,2}\left((0,T) \times B_R\right)} \to 0
\quad\text{as}\,\,n \to \infty.
\end{equation}
Set $w^{(\varepsilon)}$ to be the mollification of $w$ with respect to the spatial variable, that is,
$$
w^{(\varepsilon)}(t,x) = \int_{B_R} \phi_\varepsilon(x-y) w(t,y) \, dy, \quad \phi_\varepsilon(x) = \varepsilon^{-d} \phi(x/\varepsilon),
$$
where $\phi \in C_0^\infty(B_1)$ is a smooth non-negative function with unit integral.
By Lemma 4.3 in \cite{arXiv:1806.02635} and its proof, it follows that
$Dw^{(\varepsilon)} \in \bH_{p,0}^{\alpha,2}\left((0,T) \times B_r\right)$
for $\varepsilon \in (0,R-r)$ and, for each fixed $\varepsilon \in (0, R-r)$,
\begin{equation}
							\label{eq0820_07}
\|D w_n^{(\varepsilon)} - D w^{(\varepsilon)}\|_{\bH_p^{\alpha,2}\left((0,T) \times B_r\right)} \to 0
\quad\text{as}\,\,n \to \infty.
\end{equation}
Moreover, since $Dw \in \bH_{p,0}^{\alpha,2}\left((0,T) \times B_{r_1}\right)$, we have
\begin{equation}
							\label{eq0820_05}
\|D w^{(\varepsilon)} - D w\|_{\bH_p^{\alpha,2}\left((0,T) \times B_r\right)} \to 0
\quad\text{as}\,\,\varepsilon \to 0.
\end{equation}
Now, for $\varepsilon \in (0,R-r_1)$, we set
$$
u^\varepsilon := a^{ij}D_{ij} w^{(\varepsilon)} + b^i D_i w^{(\varepsilon)} + c w^{(\varepsilon)},
$$
$$
u_n^\varepsilon := a^{ij}D_{ij} w_n^{(\varepsilon)} + b^i D_i w_n^{(\varepsilon)} + c w_n^{(\varepsilon)}
$$
in $(0,T) \times B_r$.
Then, $u_n^\varepsilon \in C^\infty\left([0,T] \times B_r\right)$ and $u_n^\varepsilon(0,T) = 0$.
If we have
\begin{equation}
							\label{eq0820_03}
\|u_n^\varepsilon - u^\varepsilon\|_{\cH_p^{\alpha,1}\left((0,T) \times B_r\right)} \to 0
\quad\text{as}\,\,n \to \infty,
\end{equation}
we obtain that $u^\varepsilon \in \cH_{p,0}^{\alpha,1}\left((0,T) \times B_r\right)$.
To check \eqref{eq0820_03},
we write
$$
u_n^\varepsilon = D_i \left(a^{ij} D_j w_n^{(\varepsilon)}\right) - D_i a^{ij} D_j w_n^{(\varepsilon)} + b^i D_i w_n^{(\varepsilon)} + c w_n^{(\varepsilon)}.
$$
Then,
$$
\partial_t^\alpha u_n^\varepsilon = D_i \left( a^{ij} \partial_t^\alpha (D_j w_n^{(\varepsilon)}) \right) - D_i a^{ij} \partial_t^\alpha (D_j w_n^{(\varepsilon)}) + b^i \partial_t^\alpha (D_iw_n^{(\varepsilon)})
+c \partial_t^\alpha w_n^{(\varepsilon)},
$$
where $\partial_t^\alpha = \partial_t I_0^{1-\alpha}$, and by \eqref{eq0820_07} and \eqref{eq1217_01}
$$
a^{ij} \partial_t^\alpha (D_j w_n^{(\varepsilon)}) \to a^{ij} \partial_t^\alpha (D_j w^{(\varepsilon)}),
$$
\begin{multline*}
- D_i a^{ij} \partial_t^\alpha (D_j w_n^{(\varepsilon)}) + b^i \partial_t^\alpha (D_i w_n^{(\varepsilon)}) + c \partial_t^\alpha w_n^{(\varepsilon)}
\\
\to
- D_i a^{ij} \partial_t^\alpha (D_j w^{(\varepsilon)}) + b^i \partial_t^\alpha (D_i w^{(\varepsilon)}) + c \partial_t^\alpha w^{(\varepsilon)}	
\end{multline*}
in $L_p\left((0,T) \times B_r\right)$ as $n \to \infty$.
This shows  that
$$
\|\partial_t^\alpha u_n^\varepsilon - \partial_t^\alpha u^\varepsilon\|_{\bH_p^{-1}\left((0,T) \times B_r\right)} \to 0
\quad\text{as}\,\,n \to \infty
$$
and
\begin{equation}
							\label{eq0820_04}
\partial_t^\alpha u^\varepsilon = D_i \left( a^{ij} \partial_t^\alpha (D_j w^{(\varepsilon)}) \right) - D_i a^{ij} \partial_t^\alpha (D_j w^{(\varepsilon)}) + b^i \partial^\alpha_t (D_i w^{(\varepsilon)})+c\partial_t^\alpha w^{(\varepsilon)},
\end{equation}
where $\partial_t^\alpha = \partial_t I_0^{1-\alpha}$.
Using \eqref{eq1217_01} and \eqref{eq0820_07}, we also see that
$$
\|u_n^\varepsilon - u^\varepsilon\|_{L_p\left((0,T) \times B_r\right)} + \|Du_n^\varepsilon - Du^\varepsilon\|_{L_p\left((0,T) \times B_r\right)} \to 0
\quad\text{as}\,\,n \to \infty.
$$
Hence, \eqref{eq0820_03} is proved and $u^\varepsilon \in \cH_{p,0}^{\alpha,1}\left((0,T)\times B_r\right)$.

Now by \eqref{eq0820_01} we note that
$$
u = \eta(t) \partial_t I_S^{1-\alpha} v = a^{ij} D_{ij}w + b^i D_i w + c w
$$
in $(0,T) \times B_r$.
Then from \eqref{eq0820_05} we see that
$$
\|u^\varepsilon - u\|_{L_p\left((0,T) \times B_r\right)} + \|Du^\varepsilon - Du\|_{L_p\left((0,T) \times B_r\right)}  \to 0
\quad\text{as}\,\,\varepsilon \to 0.
$$
From \eqref{eq0820_04} and \eqref{eq0820_05}, we also see that \eqref{eq0820_06} holds and
$$
\|\partial_t^\alpha u^\varepsilon - \partial_t^\alpha u\|_{\bH_p^{-1}\left((0,T) \times B_r\right)}  \to 0
\quad\text{as}\,\,\varepsilon \to 0.
$$
Hence,
$$
\|u^\varepsilon - u\|_{\cH_p^{\alpha,1}\left((0,T) \times B_r\right)} \to 0
\quad\text{as}\,\,\varepsilon \to 0.
$$
Since $u^\varepsilon \in \cH_{p,0}^{\alpha,1}\left((0,T) \times B_r\right)$ for each $\varepsilon \in (0,(R-r)/2)$, we conclude that
$$
(\eta \partial_t I_S^{1-\alpha} v =) u \in \cH_{p,0}^{\alpha,1}\left((0,T) \times B_r\right).
$$
The lemma is proved.
\end{proof}

If $v \in \cH_{p,0}^{\alpha,1}\left((S,T) \times \bR^d\right)$ is a solution to a homogenous equation, one can improve its regularity as follows.

\begin{lemma}
							\label{lem0820_2}
Let $p_0\in (1,\infty)$, $\alpha \in (0,1)$, $-\infty< S \leq t_0 < T < \infty$, and $0 < R < \infty$.
Also let $a^{ij}(x_1)$ be infinitely differentiable functions of $x_1 \in \bR$ with bounded derivatives.
Suppose that Theorem \ref{thm0412_1} holds with this $p_0$ and $v \in \cH_{p_0,0}^{\alpha,1}\left((S, T) \times B_R\right)$ satisfies
$$
-\partial_t I_S^{1-\alpha} v + D_i \left( a^{ij}(x_1) D_j v \right) = D_i g_i + f
$$
in $(S,T) \times B_R$, where $g_i, f \in L_{p_0}\left((S,T) \times B_R\right)$ and $g_i(t,x) = f(t,x) = 0$ on $(t_0,T) \times B_R$.
Then, for any $r_0, r_1 \in (0,R)$ such that $r_0 < r_1$ and any infinitely differentiable function $\eta(t)$ defined on $\bR$ such that $\eta(t) = 0$ for $t \leq t_0$, we have the following.

\begin{enumerate}
\item $\eta v$ belongs to $\bH_{p_0,0}^{\alpha,2}\left((t_0,T) \times B_{r_1}\right)$ and satisfies the non-divergence form equation
$$
-\partial_t I_{t_0}^{1-\alpha} (\eta v) + a^{ij}(x_1) D_{ij}(\eta v) + D_ia^{ij}(x_1) D_j(\eta v) = F
$$
in $(t_0,T) \times B_{r_1}$,
where $F$ is defined as in \eqref{eq0207_04}.

\item $\eta(t) \partial_t I_S^{1-\alpha} v$ belongs to $L_{p_0}\left((t_0,T) \times B_{r_1}\right)$ and satisfies
\begin{equation}
							\label{eq0820_08}
\begin{aligned}
\eta(t) \partial_t I_S^{1-\alpha} v &= a^{ij}(x_1) D_{ij}(\eta v) + D_i a^{ij}(x_1)D_j(\eta v)
\\
&= D_i \left( a^{ij}(x_1) D_j(\eta v) \right)
\end{aligned}
\end{equation}
a.e. in $(t_0,T) \times B_{r_1}$.

\item $
D_x (\eta v)\in \bH_{p_0,0}^{\alpha,2} \left((t_0,T) \times B_{r_0}\right)$ and
$\eta(t) \partial_t I_S^{1-\alpha} v \in \cH_{p_0,0}^{\alpha,1}\left((t_0,T) \times B_{r_0}\right)$.

\item $\eta(t) \partial_t I_S^{1-\alpha}v$ satisfies the divergence form equation
\begin{equation}
							\label{eq0820_09}
- \partial_t I_{t_0}^{1-\alpha} \left(\eta(t) \partial_t I_S^{1-\alpha}v\right) + D_i \left(a^{ij}(x_1) D_j \left(\eta(t)\partial_t I_S^{1-\alpha} v\right) \right) = D_i G_i
\end{equation}
in $(t_0,T) \times B_{r_0}$,
where
\begin{equation}
							\label{eq0824_01}
G_i(t,x) = \sum_{j=1}^d a^{ij}(x_1) \frac{\alpha}{\Gamma(1-\alpha)} \int_S^t (t-s)^{-\alpha-1}\left( \eta(t)-\eta(s) \right) D_j v(s,x) \, ds.
\end{equation}
\end{enumerate}
\end{lemma}

\begin{proof}
As in the proof of Lemma \ref{lem0820_1}, we assume that $t_0 = 0$.
Let $\psi(x) \in C_0^\infty(\bR^d)$ be such that $0 \leq \psi(x) \leq 1$,
$$
\psi(x) = 1 \quad \text{on} \quad B_{r_1},
\quad \psi(x) = 0 \quad \text{outside} \quad B_R.
$$
Set $\zeta(t,x) = \eta(t) \psi(x)$.
By using Theorem 2.4 in \cite{arXiv:1806.02635}, find
$u \in \bH_{p_0,0}^{\alpha,2}(\bR^d_T)$ satisfying
$$
-\partial_t I_0^{1-\alpha} u + a^{ij}(x_1) D_{ij} u + D_i a^{ij}(x_1) D_j u = F_1 + F_2
$$
in $\bR^d_T$,
where $F_1$ and $F_2$ are defined by
$$
F_1(t,x) = \frac{\alpha}{\Gamma(1-\alpha)} \int_S^t (t-s)^{-\alpha-1} \left( \eta(s) - \eta(t) \right) \psi(x) v(s,x) \, ds
$$
and
$$
F_2(t,x) = D_i a^{ij}(x_1) v D_j \zeta + a^{ij}(x_1) D_i v D_j \zeta + a^{ij}(x_1) v D_{ij} \zeta + a^{ij}(x_1) D_i \zeta D_j v.
$$
Note that $F_1, F_2  \in L_{p_0}(\bR^d_T)$ and $u \in \cH_{p_0,0}^{\alpha,1}(\bR^d_T) \subset \bH_{p_0,0}^{\alpha,2}(\bR^d_T)$ also satisfies the divergence form equation
$$
-\partial_t I_0^{1-\alpha}u + D_i\left(a^{ij}(x_1)D_j u \right) = D_i \left(a^{ij}(x_1) v D_j \zeta\right) + a^{ij}(x_1) D_i \zeta D_j v + F_1
$$
in $\bR^d_T$.
In particular, since $u \in \bH_{p_0,0}^{\alpha,2}(\bR^d_T)$, we have
$$
- \int_{\bR^d_T} \partial_t I_0^{1-\alpha} u \, \varphi \, dx \, dt = \int_{\bR^d_T} I_0^{1-\alpha} u \, \partial_t \varphi \, dx \, dt
$$
for $\varphi \in C_0^\infty\left([0,T) \times \bR^d \right)$.

On the other hand, we see that $v \zeta$ belongs to $\cH_{p_0,0}^{\alpha,1}(\bR^d_T)$ and satisfies
\begin{equation}
							\label{eq0814_01}
-\partial_t I_0^{1-\alpha}(v \zeta) + D_i\left(a^{ij}(x_1)D_j (v \zeta) \right) = D_i \left(a^{ij}(x_1) v D_j \zeta\right) + a^{ij}(x_1) D_i \zeta D_j v + F_1
\end{equation}
in $\bR^d_T$.
Indeed, using Lemma \ref{lem0207_1}, that is,
$$
\partial_t I_0^{1-\alpha} (v \zeta) = \zeta \partial_t I_S^{1-\alpha} v - F_1,
$$
we check \eqref{eq0814_01} as follows.
For $\varphi \in C_0^\infty\left([0,T) \times \bR^d\right)$,
$$
\int_{\bR_T} I_0^{1-\alpha} (v \zeta) \, \varphi_t \, dx \, dt - \int_{\bR_T} a^{ij} D_j(v \zeta) D_i \varphi \, dx \, dt
$$
$$
= \int_0^T\int_{B_R} I_0^{1-\alpha} (v \zeta) \, \varphi_t \, dx \, dt - \int_0^T\int_{B_R} a^{ij} D_j(v \zeta) D_i \varphi \, dx \, dt
$$
$$
= \int_S^T \int_{B_R} I_S^{1-\alpha} v \, \partial_t(\zeta \varphi) \, dx \, dt + \int_0^T \int_{B_R} F_1 \varphi \, dx \, dt - \int_S^T \int_{B_R} a^{ij} D_j(v \zeta) D_i \varphi \, dx \, dt
$$
$$
=\int_S^T \int_{B_R} I_S^{1-\alpha} v \, \partial_t(\zeta \varphi) \, dx \, dt -  \int_S^T \int_{B_R} a^{ij}D_j v D_i (\zeta \varphi) \, dx \, dt
$$
$$
- \int_S^T\int_{B_R} a^{ij} v D_j \zeta D_i \varphi \, dx \, dt + \int_S^T \int_{B_R} a^{ij} D_j v (D_i \zeta) \varphi \, dx \, dt +  \int_0^T \int_{B_R} F_1 \varphi \, dx \, dt
$$
$$
= \int_S^T \int_{B_R} \left(f (\zeta\varphi) - g \cdot \nabla (\zeta \varphi) \right) \, dx \, dt - \int_S^T\int_{B_R} a^{ij} v D_j \zeta D_i \varphi \, dx \, dt
$$
$$
+ \int_S^T \int_{B_R} a^{ij} D_j v (D_i \zeta) \varphi \, dx \, dt + \int_0^T \int_{B_R} F_1 \varphi \, dx \, dt
$$

$$
= - \int_{\bR^d_T} a^{ij} v D_j \zeta D_i \varphi \, dx \, dt
+ \int_{\bR^d_T} a^{ij} D_j v (D_i \zeta) \varphi \, dx \, dt + \int_{\bR^d_T} F_1 \varphi \, dx \, dt,
$$
where we used the fact that $g_i = f = 0$ on $(0,T) \times B_R$.
Then from the uniqueness in Theorem \ref{thm0412_1}, we obtain that
$$
v(t,x) \zeta(t,x) = v(t,x) \eta(t) \psi(x) = u(t,x) \in \bH_{p_0,0}^{\alpha,2}(\bR^d_T)
$$
and $v \zeta$ satisfies the non-divergence form equation
$$
- \partial_t I_0^{1-\alpha} (v \zeta) + a^{ij}(x_1) D_{ij}(v \zeta) + D_ia^{ij}(x_1) D_j (v \zeta) = F_1 + F_2
$$
in $\bR^d_T$.
In particular, $F_2 \equiv 0$ on $(0,T) \times B_{r_1}$, $v \in \cH_{p_0,0}^{\alpha,1} \left( (S,T) \times B_{r_1} \right)$, $\eta v = v \zeta \in \bH_{p_0,0}^{\alpha,2}\left((0,T) \times B_{r_1} \right)$, and
$$
-\partial_t I_0^{1-\alpha} (\eta v) + a^{ij}(x_1) D_{ij}(\eta v) + D_i a^{ij}(x_1) D_j (\eta v) = F_1 = F
$$
in $(0,T) \times B_{r_1}$.
Hence, by Lemma \ref{lem0820_1} with $b^j = D_i a^{ij}$, $c = 0$, $R=r_1$, and $r = r_0$, we have $\eta(t) \partial_t I_S^{1-\alpha} v \in   L_{p_0}\left((0,T) \times B_{r_1}\right)$ and \eqref{eq0820_08} is satisfied.
Moreover,
$$
D_x(\eta v) \in \bH_{p_0,0}^{\alpha,2}\left((0,T) \times B_{r_0}\right), \quad \eta(t) \partial_t I_S^{1-\alpha} v \in \cH_{p_0, 0}^{\alpha,1}\left((0,T)\times B_{r_0}\right).
$$
To check \eqref{eq0820_09}, using \eqref{eq0820_06} we obtain that
$$
- \partial_t I_0^{1-\alpha}\left(\eta(t) \partial_t I_S^{1-\alpha} v\right) + D_i\left(a^{ij}(x_1)D_j \left(\eta(t) \partial_t I_S^{1-\alpha} v\right) \right)
$$
$$
= D_i \left[ a^{ij}(x_1)D_j \left(\eta(t) \partial_t I_S^{1-\alpha} v\right) - a^{ij}(x_1) \partial_t I_0^{1-\alpha} D_j(\eta v)
\right]
$$
in $(0,T) \times B_{r_0}$.
It only remains to notice that
$$
a^{ij}(x_1)D_j \left(\eta(t) \partial_t I_S^{1-\alpha} v\right) - a^{ij}(x_1) \partial_t I_0^{1-\alpha} D_j(\eta v) = G_i.
$$
Indeed, using the fact that
$$
\eta v \in \bH_{p_0,0}^{\alpha,2}\left((0,T) \times B_{r_1}\right)\quad
\text{and}\quad
D(\eta v) \in \bH_{p_0,0}^{\alpha,2}\left((0,T) \times B_{r_0}\right),
$$
we obtain that
$$
\partial_t I_0^{1-\alpha} \left(D_j(\eta v)\right) = D_j \left( \partial_t I_0^{1-\alpha} (\eta v) \right)
$$
in $(0,T) \times B_{r_0}$.
Then using Lemma \ref{lem0207_1} and the fact that $v \in \cH_{p_0,0}^{\alpha,1}\left((S,T) \times B_R\right)$,
\begin{multline}
							\label{eq0823_08}
D_j\left(\eta(t) \partial_t I_S^{1-\alpha} v\right) - \partial_t I_0^{1-\alpha} D_j(\eta v) = D_j \left[ \eta(t) \partial_t I_S^{1-\alpha} v - \partial_t I_0^{1-\alpha} (\eta v) \right]
\\
= \frac{\alpha}{\Gamma(1-\alpha)} D_j \int_S^t (t-s)^{-\alpha-1} \left(\eta(s) - \eta(t)\right) v(s,x) \, ds
\\
= \frac{\alpha}{\Gamma(1-\alpha)} \int_S^t (t-s)^{-\alpha-1} \left(\eta(s) - \eta(t)\right) D_j v(s,x) \, ds.
\end{multline}
The lemma is proved.
\end{proof}

\begin{lemma}
							\label{lem0823_1}
Let $p_0\in (1,\infty)$, $\alpha \in (0,1)$, $-\infty < S \leq t_0 < T < \infty$, $0 < R < \infty$,
and $a^{ij} = a^{ij}(x_1)$ be infinitely differentiable with bounded derivatives.
Suppose that Theorem \ref{thm0412_1} holds with this $p_0$ and $v \in \cH_{p_0,0}^{\alpha,1}\left((S,T) \times B_R\right)$ satisfies
$$
-\partial_t I_S^{1-\alpha} v + D_i \left(a^{ij}(x_1) D_j v \right) = D_i g_i + f
$$
in $(S,T) \times B_R$, where $g_i, f \in L_{p_0}\left((S,T) \times B_R\right)$ and $g_i = f = 0$ on $(t_0,T) \times B_R$.
Then, for any $r \in (0,R)$ and any infinitely differentiable function $\eta(t)$ defined on $\bR$ such that $\eta(t) = 0$ for $t \leq t_0$,
we have
\begin{multline}
							\label{eq0823_04}
\left\|\eta(t) \partial_t I_S^{1-\alpha} v\right\|_{L_{p_0}((t_0,T) \times B_r)} \leq \frac{N}{R-r} \|D(\eta v)\|_{L_{p_0}((t_0,T) \times B_R)}
\\
+ N(R-r) \|G\|_{L_{p_0}((t_0,T) \times B_R)},
\end{multline}
\begin{multline}
							\label{eq0823_05}
\left\|D\left(\eta(t) \partial_t I_S^{1-\alpha} v\right) \right\|_{L_{p_0}((t_0,T) \times B_r)} \leq \frac{N}{(R-r)^2} \|D(\eta v)\|_{L_{p_0}((t_0,T) \times B_R)}
\\
+ N\|G\|_{L_{p_0}((t_0,T) \times B_R)},
\end{multline}
where $N = N(d,\delta,\alpha,p_0)$ and
$G = (G_1, \ldots, G_d)$ is defined as in \eqref{eq0824_01}.
\end{lemma}

\begin{proof}
We again assume that $t_0 = 0$.
Let $\tau_0, \tau_1 \in (0,R)$ such that $\tau_0 < \tau_1$.
By Lemma \ref{lem0820_2} the function $\eta(t) \partial_t I_S^{1-\alpha}v$ belongs to $\cH_{p_0,0}^{\alpha,1}\left((0,T) \times B_{\tau
_1}\right)$ and satisfies the divergence form equation
$$
- \partial_t I_0^{1-\alpha} \left(\eta(t) \partial_t I_S^{1-\alpha}v\right) + D_i \left(a^{ij}(x_1) D_j \left(\eta(t)\partial_t I_S^{1-\alpha} v\right) \right) = D_i G_i
$$
in $(0,T) \times B_{\tau_1}$, where $G_i$ is given in \eqref{eq0824_01}.
Since Theorem \ref{thm0412_1} holds with $p_0$, by Lemma \ref{lem0731_1} with $\lambda = 0$, we obtain that
\begin{multline}
							\label{eq0823_01}
\|D( \eta(t) \partial_t I_S^{1-\alpha} v )\|_{L_{p_0}((0,T) \times B_{\tau_0})}
\\
\leq \frac{N}{\tau_1-\tau_0} \|\eta(t) \partial_t I_S^{1-\alpha} v\|_{L_{p_0}((0,T) \times B_{\tau_1})}
+ N \|G_i\|_{L_{p_0}((0,T) \times B_{\tau_1})},
\end{multline}
where $N = N(d,\delta,\alpha,p_0)$.

We now estimate $\eta(t) \partial_t I_S^{1-\alpha} v$.
Let $r, R_0, R_1, R_2 \in (0,\infty)$ such that $r < R_0 < R_1 < R_2 < R$.
By Lemma \ref{lem0820_2} we have
\begin{equation*}
D (\eta v)\in \bH_{p_0,0}^{\alpha,2} \left((0,T) \times B_{R_1}\right),
\quad
\eta(t) \partial_t I_S^{1-\alpha} v \in \cH_{p_0,0}^{\alpha,1}\left((0,T) \times B_{R_1}\right),
\end{equation*}
and
\begin{equation}
							\label{eq0820_10}
\eta(t)\partial_t I_S^{1-\alpha} v = a^{ij}(x_1) D_{ij}(\eta v) + D_ia^{ij}(x_1) D_j(\eta v) = D_i\left(a^{ij} D_j(\eta v)\right)
\end{equation}
a.e. in $(0,T) \times B_{R_2}$.
Let $\varphi_k(x) \in C_0^\infty(\bR^d)$ such that $0 \leq \varphi_k(x) \leq 1$,
$$
\varphi_k(x) = 1 \quad \text{on} \quad B_{r_k},\quad
\varphi_k(x) = 0 \quad
\text{outside} \quad B_{r_{k+1}},
\quad |D \varphi_k(x)| \leq N(d) \frac{2^k}{R_0 - r},
$$
where
$$
r_0 = r, \quad
r_k = r+(R_0-r) (1-2^{-k}), \quad k = 1, 2, \ldots.
$$
Then the function $\zeta_k(t,x) := \eta(t) \varphi_k(x)$ satisfies $0 \leq \zeta_k(t,x) \leq 1$ and
$$
\zeta_k(t,x) = 0 \quad \text{on} \quad (S,T) \times \bR^d \setminus (0,T) \times B_{r_{k+1}}.
$$
Denote
$$
B_T^k = (0,T) \times B_{r_k}, \quad B_T^\infty = (0,T) \times B_{R_0}.
$$
Since
$$
w:=\eta(t) \partial_t I_S^{1-\alpha} v \in \cH_{p_0,0}^{\alpha,1}\left((0,T) \times B_{R_1}\right),
$$
it follows that
$$
w |w|^{p_0-2} \varphi_k^{p_0} \in L_{q_0}\left((0,T) \times \bR^d\right)
$$
and
$$
D_i\left( w |w|^{p_0-2} \varphi_k^{p_0} \right) \in L_{q_0}\left((0,T) \times \bR^d\right),
$$
where $q_0=p_0/(p_0-1)$.
In particular,
$$
D_i\left( w |w|^{p_0-2} \varphi_k^{p_0} \right)
= (p_0-1) |w|^{p_0-2} (D_i w) \varphi_k^{p_0}
+ p_0 w |w|^{p_0-2} \varphi_k^{p_0-1} D_i \varphi_k.
$$
By multiplying both sides of the equation \eqref{eq0820_10} by $w |w|^{p_0-2} \varphi_k^{p_0}$
and integrating by parts over $(0,T) \times \bR^d$, we have
\begin{multline}
							\label{eq0821_01}
\int_{\bR^d_T} |w|^{p_0} \varphi_k^{p_0} \, dx \, dt = - \int_{\bR^d_T} a^{ij} D_j (\eta v) D_i \left(w|w|^{p_0-2} \varphi_k^{p_0}\right) \, dx \, dt
\\
= - (p_0-1) \int_{\bR^d_T} a^{ij}D_j(\eta v) |w|^{p_0-2} D_i w \, \varphi_k^{p_0} \, dx \, dt
\\
- p_0 \int_{\bR^d_T} a^{ij} D_j(\eta v) w|w|^{p_0-2} \varphi_k^{p_0-1} D_i \varphi_k \, dx \, dt =: J_1 + J_2.
\end{multline}
Note that, for any $\varepsilon_0, \varepsilon_k \in (0,\infty)$, $k=1,2,\ldots$,
$$
|J_1| \leq \varepsilon_0 \int_{\bR^d_T} |w|^{p_0} \varphi_k^{p_0} \, dx \, dt + N \int_{\bR^d_T} |D(\eta v)|^{p_0/2} |Dw|^{p_0/2} \varphi_k^{p_0} \, dx \, dt
$$
$$
\leq \varepsilon_0 \int_{\bR^d_T} |w|^{p_0} \varphi_k^{p_0} \, dx \, dt + \varepsilon_k^{p_0} \int_{\bR^d_T} |Dw|^{p_0} \varphi_k^{p_0} \, dx \, dt
$$
$$
+ N \varepsilon_k^{-p_0} \int_{\bR^d_T} |D(\eta v)|^{p_0} \varphi_k^{p_0} \, dx \, dt,
$$
where $N = N(d,\delta,\varepsilon_0, p_0)$, and
$$
|J_2| \leq \varepsilon_0 \int_{\bR^d_T} |w|^{p_0} \varphi_k^{p_0} \, dx \, dt + N \int_{\bR^d_T} |D(\eta v) D \varphi_k|^{p_0} \, dx \, dt,
$$
where $N = N(d,\delta,\varepsilon_0, p_0)$.
From these inequalities and \eqref{eq0821_01}, and choosing an appropriate $\varepsilon_0 > 0$, we obtain that
$$
\int_{\bR^d_T} |w|^{p_0} \varphi_k^{p_0} \, dx \, dt
$$
$$
\leq N \varepsilon_k^{-p_0} \int_{\bR^d_T} |D(\eta v)|^{p_0}\varphi_k^{p_0} \, dx \, dt + N \int_{\bR^d_T} |D(\eta v)|^{p_0}|D\varphi_k|^{p_0}  \, dx \, dt
$$
$$
+ \varepsilon_k^{p_0} \int_{\bR^d_T} |Dw|^{p_0} \varphi_k^{p_0} \, dx \, dt,
$$
where $N = N(d,\delta,p_0)$.
This shows that
\begin{multline}
							\label{eq0821_02}
\left\|\eta(t) \partial_t I_S^{1-\alpha} v\right\|_{L_{p_0}(B^k_T)} \leq N \left(\frac{2^k}{R_0-r} + \varepsilon_k^{-1}\right) \left\|D(\eta v)\right\|_{L_{p_0}(B^{k+1}_T)}
\\
+ \varepsilon_k \left\| D(\eta(t)\partial_t I_S^{1-\alpha} v)\right\|_{L_{p_0}(B^{k+1}_T)}.
\end{multline}
From the inequality \eqref{eq0823_01} with $\tau_0 = r_{k+1}$ and $\tau_1 = r_{k+2}$, we obtain that
\begin{align*}
&\left\|D( \eta(t) \partial_t I_S^{1-\alpha} v )\right\|_{L_{p_0}(B^{k+1}_T)}\\
&\leq N \frac{2^k}{R_0-r} \left\|\eta(t) \partial_t I_S^{1-\alpha} v\right\|_{L_{p_0}(B^{k+2}_T)} + N \|G_i\|_{L_{p_0}(B^{k+2}_T)},
\end{align*}
where $N = N(d,\delta,\alpha,p_0)$.
Combining this with \eqref{eq0821_02}, we have
\begin{multline}
							\label{eq0823_02}
\left\|\eta(t) \partial_t I_S^{1-\alpha} v\right\|_{L_{p_0}(B^k_T)} \leq N_0 \left(\frac{2^k}{R_0-r} + \varepsilon_k^{-1}\right) \left\|D(\eta v)\right\|_{L_{p_0}(B^{k+1}_T)}
\\
+ N_0 \frac{2^k \varepsilon_k}{R_0-r} \left\| \eta(t) \partial_t I_S^{1-\alpha} v\right\|_{L_{p_0}(B^{k+2}_T)} + N_0 \varepsilon_k \|G_i\|_{L_{p_0}(B^{k+2}_T)},
\end{multline}
where $N_0 = N_0(d,\delta,\alpha,p_0)$.
Choose
$$
\varepsilon_k = \frac{R_0-r}{N_0} 2^{-k - 4}
$$
and multiply both side of the inequality \eqref{eq0823_02} by $2^{-2k}$ so that we have
$$
2^{-2k}\left\|\eta(t) \partial_t I_S^{1-\alpha} v\right\|_{L_{p_0}(B^k_T)} \leq (N_0 + 2^4 N_0^2) \frac{2^{-k}}{R_0-r} \left\|D(\eta v)\right\|_{L_{p_0}(B^{k+1}_T)}
$$
$$
+ 2^{-2(k+2)} \left\| \eta(t) \partial_t I_S^{1-\alpha} v\right\|_{L_{p_0}(B^{k+2}_T)} + (R_0-r) 2^{-k-4} \|G_i\|_{L_{p_0}(B^{k+2}_T)}.
$$
By making summations with respect to $k=0,1,2,\ldots$, it follows that
$$
\left\|\eta(t) \partial_t I_S^{1-\alpha} v\right\|_{L_{p_0}(B^0_T)} + 2^{-2}\left\|\eta(t) \partial_t I_S^{1-\alpha} v\right\|_{L_{p_0}(B^1_T)} + \sum_{k=2}^\infty 2^{-2k}\left\|\eta(t) \partial_t I_S^{1-\alpha} v\right\|_{L_{p_0}(B^k_T)}
$$
$$
\leq \frac{N}{R_0-r} \|D(\eta v)\|_{L_{p_0}(B^\infty_T)} + \sum_{k=2}^\infty 2^{-2k}\left\|\eta(t) \partial_t I_S^{1-\alpha} v\right\|_{L_{p_0}(B^k_T)} + N(R_0-r) \|G_i\|_{L_{p_0}(B^\infty_T)},
$$
where we recall that $B_T^0 = (0,T) \times B_r$ and $B^\infty_T = (0,T) \times B_{R_0} \subset (0,T) \times B_R$.
Removing the same terms from both sides of the above inequality, we get
$$
\left\|\eta(t) \partial_t I_S^{1-\alpha} v\right\|_{L_{p_0}(B_T^0)} \leq \frac{N}{R_0-r} \|D(\eta v)\|_{L_{p_0}(B_T^\infty)} + N(R_0-r) \|G_i\|_{L_{p_0}(B_T^\infty)}.
$$
This implies the inequality \eqref{eq0823_04}.
To obtain \eqref{eq0823_05} we combine the inequality \eqref{eq0823_01} with $\tau_0 = r$ and $\tau_1 = (R+r)/2$ and the above inequality with $r = (R+r)/2$ and $R_0 = R$.
\end{proof}

\begin{lemma}
							\label{lem0823_2}
Let $p_0\in (1,\infty)$, $\alpha \in (0,1)$, $-\infty < S \leq t_0 < T < \infty$, $0 < R < \infty$,
and $a^{ij} = a^{ij}(x_1)$ be infinitely differentiable with bounded derivatives.
Suppose that Theorem \ref{thm0412_1} holds with this $p_0$ and $v \in \cH_{p_0,0}^{\alpha,1}\left((S,T) \times B_R\right)$ satisfies
$$
-\partial_t I_S^{1-\alpha} v + D_i \left(a^{ij}(x_1) D_j v \right)
= D_i g_i + f
$$
in $(S,T) \times B_R$, where $g_i, f \in L_{p_0}\left((S,T) \times B_R\right)$ and $g_i = f = 0$ on $(t_0,T) \times B_R$.
Then, for any $r \in (0,R)$ and any infinitely differentiable function $\eta(t)$ defined on $\bR$ such that $\eta(t) = 0$ for $t \leq t_0$, we have
\begin{equation}
							\label{eq0824_03}
\|D(\eta v)\|_{L_{p_1}((t_0,T) \times B_r)} \leq N \|D (\eta v)\|_{L_{p_0}((t_0,T) \times B_R)} + N\|G\|_{L_{p_0}((t_0,T) \times B_R)},
\end{equation}
where $p_1 \in (p_0,\infty]$ satisfies
\begin{equation}
							\label{eq1029_01}
p_1 \geq p_0 + \frac{\alpha}{\alpha d + 1 - \alpha}, \quad p_1 = \infty \quad \text{if} \quad p_0 > d + 1/\alpha,
\end{equation}
$G = (G_1,\ldots,G_d)$ is defined by
$$
G_\ell = \frac{\alpha}{\Gamma(1-\alpha)} \int_S^t (t-s)^{-\alpha-1}\left( \eta(t)-\eta(s) \right) D_\ell v(s,x) \, ds, \quad \ell = 1, \ldots,d,
$$
and $N=N(d,\delta,\alpha,p_0, r,R)$.
\end{lemma}

\begin{proof}
As in the proofs above, set $t_0 = 0$.
Denote
$$
V = \sum_{j=1}^d a^{1j}(x_1) D_j (\eta v).
$$
We claim that
\begin{multline}
							\label{eq0823_06}
\|D_{x'}(\eta v)\|_{\bH_{p_0}^{\alpha,1}\left((0,T) \times B_r\right)} + \|V\|_{\bH_{p_0}^{\alpha,1}\left((0,T) \times B_r\right)}
\\
\leq N \|D(\eta v)\|_{L_{p_0}((0,T) \times B_R)} + N\|G\|_{L_{p_0}((0,T) \times B_R)},
\end{multline}
where $N=N(d,\delta,\alpha,p_0,r,R)$ and  $D_{x'}$ means one of $D_{x_k}$, $k = 2, \ldots,d$, or the whole set of $\{ D_{x_2}, \ldots, D_{x_d}\}$ depending on the context.
If \eqref{eq0823_06} holds, we take $p_1 = p_1(d,\alpha,p_0) \in (p_0, \infty]$ as follows.
If $p_0 \leq 1/\alpha$, take $p_1$ satisfying
$$
p_1 \in \left(p_0, \frac{1/\alpha +d}{1/(\alpha p_0) + d/p_0-1}\right) \quad \text{if} \quad p_0 \leq d,
$$
$$
p_1 \in (p_0, \alpha p_0^2 + p_0) \quad \text{if} \quad p_0 > d.
$$
If $p_0 > 1/\alpha$, take $p_1$ satisfying
$$
p_1 \in \left(p_0, p_0 + p_0^2/d\right) \quad \text{if} \quad p_0 \leq d,
$$
$$
p_1 \in (p_0, 2p_0) \quad \text{if} \quad p_0 > d \,\,\text{and}\,\, p_0 \leq d + 1/\alpha,
$$
$$
p_1 = \infty \quad \text{if} \quad p_0 > d + 1/\alpha.
$$
Then we see that \eqref{eq1029_01} is satisfied, and by the Sobolev inequalities (see Corollary \ref{cor1211_1}, Theorem \ref{thm1207_2}, Corollary \ref{cor0225_1}, Theorem \ref{thm0214_1}, and Corollary \ref{cor1029_1}), we have
$$
\|D_{x'}(\eta v)\|_{L_{p_1}\left((0,T) \times B_r\right)} \leq N \|D(\eta v)\|_{L_{p_0}((0,T) \times B_R)} + N\|G\|_{L_{p_0}((0,T) \times B_R)},
$$
$$
\|V\|_{L_{p_1}\left((0,T) \times B_r\right)} \leq N \|D(\eta v)\|_{L_{p_0}((0,T) \times B_R)} + N\|G\|_{L_{p_0}((0,T) \times B_R)},
$$
where $N = N(d,\delta,\alpha,p_0,r,R)$.
From these two inequalities and the relation
$$
D_1(\eta v) = \frac{1}{a^{11}(x_1)}\bigg( V - \sum_{j=2}^d a^{1j}(x_1) D_j(\eta v) \bigg),
$$
we obtain \eqref{eq0824_03} including the estimate for $D_1(\eta v)$.

To prove \eqref{eq0823_06}, it suffices to show that
\begin{multline}
							\label{eq0823_11}
\|D \left(D_{x'}(\eta v)\right)\|_{L_{p_0}\left((0,T) \times B_r\right)} + \|\partial_t I_0^{1-\alpha}\left(D(\eta v)\right)\|_{L_{p_0}\left((0,T) \times B_r\right)} + \|DV\|_{L_{p_0}\left((0,T) \times B_r\right)}
\\
\leq N \|D(\eta v)\|_{L_{p_0}((0,T) \times B_R)} + N\|G\|_{L_{p_0}((0,T) \times B_R)},
\end{multline}
where $N = N(d,\delta,\alpha,p_0,r,R)$.
Indeed, this inequality also includes the estimate for $\partial_t I_0^{1-\alpha}V$ because
$$
\partial_t I_0^{1-\alpha} V = \partial_t I_0^{1-\alpha} \sum_{j=1}^d a^{1j}(x_1) D_j(\eta v) = \sum_{j=1}^d a^{1j}(x_1) \partial_t I_0^{1-\alpha} \left(D_j(\eta v)\right).
$$

We now prove \eqref{eq0823_11}.
Let $R_0, R_1 \in (r,R)$ such that $R_0 < R_1$.
By Lemma \ref{lem0820_2} we have
\eqref{eq0820_10} a.e. in $(0,T) \times B_{R_1}$.
Differentiate both sides of \eqref{eq0820_10} with respect to $x_\ell$, $\ell=2,\ldots,d$, to get
\begin{equation}
							\label{eq0824_02}
D_\ell \left( \eta(t) \partial_t I_S^{1-\alpha} v \right) = D_i \left( a^{ij} D_j D_\ell(\eta v) \right)
\end{equation}
in $(0,T) \times B_{R_0}$.
This is possible because
$$
D (\eta v)\in \bH_{p_0,0}^{\alpha,2} \left((0,T) \times B_{R_0}\right),
\quad
\eta(t) \partial_t I_S^{1-\alpha} v \in \cH_{p,0}^{\alpha,1}\left((0,T) \times B_{R_0}\right).
$$
Then from \eqref{eq0824_02} and \eqref{eq0823_08}
\begin{equation}
							\label{eq0821_03}
- \partial_t I_0^{1-\alpha} \left( D_\ell(\eta v) \right) + D_i \left(a^{ij} D_j D_\ell(\eta v) \right) = G_\ell
\end{equation}
a.e. in $(0,T) \times B_{R_0}$.
Since the equation \eqref{eq0821_03} can be viewed as a divergence form equation so that $w:= D_\ell (\eta v)$ satisfies
$$
- \partial_t I_0^{1-\alpha} w + D_i \left(a^{ij} D_j w \right) = G_\ell
$$
in $(0,T) \times B_{R_0}$,
by Lemma \ref{lem0731_1} with $\lambda = 0$
\begin{multline}
							\label{eq0823_09}
\|D D_\ell(\eta v)\|_{L_{p_0}\left((0,T) \times B_r\right)} \leq \frac{N}{R_0-r}\|D_\ell(\eta v)\|_{L_{p_0}\left((0,T) \times B_{R_0}\right)}
\\
+ N (R_0-r) \|G\|_{L_{p_0}\left((0,T) \times B_{R_0}\right)}.
\end{multline}

We then estimate $\partial_t I_0^{1-\alpha} \left( D(\eta v) \right)$.
Note that as in \eqref{eq0823_08}
$$
\partial_t I_0^{1-\alpha} \left( D(\eta v) \right) = D \left( \eta(t) \partial_t I_S^{1-\alpha} v \right) + \partial_t I_0^{1-\alpha} \left(D(\eta v)\right) - D\left(\eta(t) \partial_t I_S^{1-\alpha} v \right)
$$
$$
= D \left( \eta(t) \partial_t I_S^{1-\alpha} v \right) + \frac{\alpha}{\Gamma(1-\alpha)} \int_S^t (t-s)^{-\alpha-1}\left( \eta(t)-\eta(s) \right) D v(s,x) \, ds
$$
in $(0,T) \times B_{R_0}$.
Then by \eqref{eq0823_05} in Lemma \ref{lem0823_1}
\begin{multline*}
\|\partial_t I_0^{1-\alpha} \left( D(\eta v) \right)\|_{L_{p_0}\left((0,T) \times B_r\right)} \leq \frac{N}{(R-r)^2} \|D(\eta v)\|_{L_{p_0}\left((0,T) \times B_R\right)}
\\
+ \|G\|_{L_{p_0}\left((0,T) \times B_R\right)}.
\end{multline*}

Finally, to estimate $DV$, from the relation \eqref{eq0820_10} and the fact that $a^{ij}(x_1)$ are independent of $x_j$, $j=2,\ldots, d$, we see that
$$
D_1 V = \eta(t) \partial_t I_S^{1-\alpha} v - \sum_{i=2}^d \sum_{j=1}^d a^{ij}(x_1) D_{ij}(\eta v)
$$
a.e. in $(0,T) \times B_{R_1}$.
That is,
$$
D_1 V = \eta(t) \partial_t I_S^{1-\alpha} v - a^{ij}(x_1) D \left(D_{x'}(\eta v)\right).
$$
Also see that
$$
D_{x'}V = a^{1j} D_{x'} \left(D(\eta v)\right).
$$
Then the estimate for $DV$ follows from \eqref{eq0823_09} and \eqref{eq0823_04} in Lemma \ref{lem0823_1}.
Therefore, \eqref{eq0823_11} is proved, and so is the lemma.
\end{proof}

\section{Level set arguments}
                        \label{sec5}

Recall that $Q_{R_1,R_2}(t,x) = (t-R_1^{2/\alpha}, t) \times B_{R_2}(x)$ and  $Q_R(t,x)=Q_{R,R}(t,x)$.
For $(t_0,x_0) \in \bR \times \bR^d$ and a function $g$ defined on $(-\infty,T) \times \bR^d$, we set
\begin{equation}
							\label{eq0406_03b}
\cM g(t_0,x_0) = \sup_{Q_{R}(t,x) \ni (t_0,x_0)} \dashint_{Q_{R}(t,x)}|g(s,y)| I_{(-\infty,T) \times \bR^d} \, dy \, ds
\end{equation}
and
\begin{equation}
							\label{eq0406_03}
\cS\cM g(t_0,x_0) = \sup_{Q_{R_1,R_2}(t,x) \ni (t_0,x_0)} \dashint_{Q_{R_1,R_2}(t,x)}|g(s,y)| I_{(-\infty,T) \times \bR^d} \, dy \, ds.
\end{equation}
The first one is called the (parabolic) maximal function of $g$, and second one the strong (parabolic) maximal function of $g$.
 Below we use the notation $(u)_{\cD}$ to denote the average of $u$ over $\cD$, where $\cD$ is a subset of $\bR^{d+1}$.

\begin{proposition}
							\label{prop0406_1}
Let $p_0\in (1,\infty)$, $\alpha \in (0,1)$, $T \in (0,\infty)$, and $a^{ij} = a^{ij}(x_1)$.
Assume that Theorem \ref{thm0412_1} holds with this $p_0$ and $u \in \cH_{p_0,0}^{\alpha,1}(\bR^d_T)$ satisfies
$$
-\partial_t I_0^{1-\alpha} u + D_i \left(a^{ij}(x_1) D_j u \right)
= D_i g_i
$$
in $(0,T) \times \bR^d$, where $g = (g_1,\ldots,g_d) \in L_p(\bR^d_T)$.
Then, for $(t_0,x_0) \in [0,T] \times \bR^d$ and $R \in (0,\infty)$,
there exist
$$
w \in \cH_{p_0,0}^{\alpha,1}((t_0-R^{2/\alpha}, t_0)\times \bR^d), \quad v \in \cH_{p_0,0}^{\alpha,1}((S, t_0)\times \bR^d),
$$
where $S := \min\{0, t_0-R^{2/\alpha}\}$,
such that $u = w + v$ in $Q_R(t_0,x_0)$,
\begin{equation}
                                    \label{eq8.13}
\left( |Dw|^{p_0} \right)_{Q_R(t_0,x_0)}^{1/p_0} \le N \left( |g_i|^{p_0} \right)_{Q_{2R}(t_0,x_0)}^{1/p_0},
\end{equation}
and
\begin{multline}
							\label{eq0411_01}
\left( |Dv|^{p_1} \right)_{Q_{R/2}(t_0,x_0)}^{1/p_1} \leq N \left( |g_i|^{p_0} \right)_{Q_{2R}(t_0,x_0)}^{1/p_0}
\\
+ N \sum_{k=0}^\infty 2^{-k\alpha} \left( \dashint_{\!t_0 - (2^{k+1}+1)R^{2/\alpha}}^{\,\,\,t_0} \dashint_{B_R(x_0)} |Du(s,y)|^{p_0} \, dy \, ds \right)^{1/p_0},
\end{multline}
where $p_1 = p_1(d, \alpha,p_0)\in (p_0,\infty]$ satisfies \eqref{eq1029_01} and $N=N(d,\delta, \alpha,p_0)$.
Here we understand that $u$ and $f$ are extended to be zero whenever $t < 0$
and
$$
\left( |Dv|^{p_1} \right)_{Q_{R/2}(t_0,x_0)}^{1/p_1} = \|Dv\|_{L_\infty(Q_{R/2}(t_0,x_0))},
$$
provided that $p_1 = \infty$.
\end{proposition}

\begin{proof}
We extend $u$ and $g_i$ to be zero, again denoted by $u$ and $g_i$, on $(-\infty,0) \times \bR^d$.
Thanks to translation, it suffices to prove the desired inequalities when $x_0 = 0$.
Moreover, by scaling we assume that $R = 1$.
By considering
$$
-\partial_t I_0^{1-\alpha} u + D_i \left(a^{ij}_{\varepsilon}(x_1) D_j u \right)
= D_i \left( g_i + \left(a^{ij}_{\varepsilon}(x_1) - a^{ij}(x_1)\right) D_j u \right),
$$
where $a^{ij}_{\varepsilon}(x_1)$ are mollifications of $a^{ij}(x_1)$, and using the dominated convergence theorem, we may assume that the coefficients $a^{ij}(x_1)$ are infinitely differential with bounded derivatives.

For $R=1$ and $t_0 \in (0,\infty)$, set $\zeta = \zeta(t,x)$ to be an infinitely differentiable function defined on $\bR^{d+1}$ such that
$$
\zeta = 1 \quad \text{on} \quad (t_0-1, t_0) \times B_1,
$$
and
$$
\zeta = 0 \quad \text{on} \quad \bR^{d+1} \setminus (t_0-2^{2/\alpha}, t_0+2^{2/\alpha}) \times B_2.
$$
Using Theorem \ref{thm0412_1} with $p_0$, find $w \in \cH_{p_0,0}^{\alpha,1}(\bR^d_T)$ to be the solution of the problem

$$
\left\{
\begin{aligned}
-\partial_t I_{t_0-1}^{1-\alpha} w + D_i(a^{ij}(x_1) D_j w ) &= D_i (\zeta g_i)
\quad  \text{in} \,\, (t_0-1, t_0) \times \bR^d,
\\
w(t_0 - 1,x) &= 0 \quad  \text{on} \quad \bR^d.
\end{aligned}
\right.
$$
We extend $w$ to be zero on $(-\infty,t_0-1) \times \bR^d$.
From Theorem \ref{thm0412_1} we have
\begin{equation}
							\label{eq1214_01}
\|Dw\|_{L_{p_0}\left(Q_r(t_0,0)\right)} \leq N \|g\|_{L_{p_0}(Q_2(t_0,0))}
\end{equation}
for any $r > 0$, where $N = N(d,\delta,\alpha,p_0)$.
Set $v = u - w$ so that
$$
v =
\left\{
\begin{aligned}
u-w, &\quad t \in (t_0 - 1,t_0),
\\
u, &\quad t \in (-\infty, t_0 - 1],
\end{aligned}
\right.
$$
where we note that it is possible to have $t_0 - 1 < 0$.
Then by Lemma \ref{lem0206_1}, $v$ belongs to $\cH_{p_0,0}^{\alpha,1}\left((S,t_0) \times \bR^d\right)$ for $S = \min \{0, t_0 -1\}$, and $w$, $u$, and $v$ satisfy
$$
\partial_t^\alpha w = \partial_t I_{t_0-1}^{1-\alpha} w = \partial_t I_S^{1-\alpha} w,
\quad
I_S^{1-\alpha} w = \left\{
\begin{aligned}
I_{t_0-1}^{1-\alpha} w, &\quad t \geq t_0-1,
\\
0, &\quad  S \leq t < t_0-1,
\end{aligned}
\right.
$$
$$
\partial_t^\alpha u = \partial_t I_0^{1-\alpha} u = \partial_t I_S^{1-\alpha} u, \quad
I_S^{1-\alpha}u = \left\{
\begin{aligned}
I_0^{1-\alpha}u, &\quad t \geq 0,
\\
0 &\quad S \leq t < 0,
\end{aligned}
\right.
$$
and
\begin{equation}
							\label{eq0725_01}
-\partial_t I_S^{1-\alpha} v + D_i\left(a^{ij}(x_1) D_j v \right) = D_i h_i
\end{equation}
in $(S,t_0)\times \bR^d$, where
$$
h_i(t,x) = \left\{
\begin{aligned}
\left( 1 -\zeta(t,x)\right) g_i(t,x) \quad &\text{in} \,\, (t_0 - 1, t_0) \times \bR^d,
\\
g_i(t,x) \quad &\text{in} \,\, (S, t_0 - 1) \times \bR^d,
\end{aligned}
\right.
\quad i = 1, \ldots, d.
$$
Recall that $g_i$ are extended as zero for $t \leq 0$.

Let $\eta \in C_0^\infty(\bR)$ such that $0 \leq \eta(t) \leq 1$ and
$$
\eta(t) =
\left\{
\begin{aligned}
1 \quad &\text{if} \quad t \in (t_0-(1/2)^{2/\alpha},t_0),
\\
0 \quad &\text{if} \quad t \in \bR \setminus (t_0-1,t_0+1),
\end{aligned}
\right.
$$
and
$$
\left|\frac{\eta(t)-\eta(s)}{t-s}\right| \le N(\alpha).
$$

Since $v \in \cH_{p_0,0}^{\alpha,1}\left((S,t_0) \times B_1\right)$ satisfies \eqref{eq0725_01}
in $(S,t_0) \times B_1$ and $h_i = 0$ on $(t_0-1,t_0) \times B_1$,
by Lemma \ref{lem0823_2}
we have
\begin{multline}
							\label{eq0715_01}
\|D(\eta v)\|_{L_{p_1}((t_0-1,t_0) \times B_{1/2})} \leq N \|D(\eta v)\|_{L_{p_0}((t_0-1,t_0) \times B_1)} + N \|G\|_{L_{p_0}((t_0-1,t_0) \times B_1)}
\\
\leq N \| |D u| + |D w| + |G| \|_{L_{p_0}\left((t_0-1,t_0)\times B_1\right)},
\end{multline}
where $p_1 = p_1(d,\alpha,p_0)$ satisfying \eqref{eq1029_01}, $N = N(d, \delta,
\alpha,p_0)$, and $G$ is defined as in Lemma \ref{lem0823_2}.

Since $D v = 0$ for $t \leq S$, we write
\begin{align*}
&\frac{\Gamma(1-\alpha)}{\alpha} G(t,x) = \int_{-\infty}^t (t-s)^{-\alpha-1} \left( \eta(s) - \eta(t) \right) D v(s,x) \, ds\\
&= \int_{t-1}^t (t-s)^{-\alpha-1}\left( \eta(s) - \eta(t) \right) D v(s,x) \, ds\\
&\quad + \int_{-\infty}^{t-1} (t-s)^{-\alpha-1}\left( \eta(s) - \eta(t) \right) D v(s,x) \, ds := I_1(t,x)+I_2(t,x),
\end{align*}
where
\begin{align*}
|I_1(t,x)| \le N \int_{t-1}^t |t-s|^{-\alpha} |D v(s,x)|\,ds= N \int_0^1 |s|^{-\alpha} |D v(t-s,x)|\, ds.
\end{align*}
From this and the Minkowski inequality we have
\begin{multline}
							\label{eq0715_02}
\|I_1\|_{L_{p_0}\left((t_0-1,t_0) \times B_1\right)} \le N \|D v\|_{L_{p_0}\left((t_0-2,t_0)\times B_1\right)}
\\
\leq N \|D v\|_{L_{p_0}\left((t_0-1,t_0)\times B_1\right)} + N \|D u\|_{L_{p_0}\left((t_0-2,t_0-1)\times B_1\right)}.
\end{multline}
To estimate $I_2$, we see that
$\eta(s) = 0$ for any $s \in (-\infty, t-1)$ with $t \in (t_0-1,t_0)$.
Thus we have
$$
I_2(t,x) = -\eta(t) \int_{-\infty}^{t-1} (t-s)^{-\alpha-1} D v(s,x) \, ds.
$$
Then,
\begin{align*}
|I_2(t,x)| &\le \int_{-\infty}^{t-1} |t-s|^{-\alpha-1} |D v(s,x)| \, ds\\
&= \sum_{k=0}^\infty \int_{t-2^{k+1}}^{t-2^k} |t-s|^{-\alpha-1} |D v(s,x)|\,ds\\
&\le \sum_{k=0}^\infty \int_{t-2^{k+1}}^{t-2^k } 2^{-k(\alpha+1)} |D v(s,x)|\,ds.
\end{align*}
From this we have
\begin{equation*}
\|I_2\|_{L_{p_0}\left((t_0-1,t_0) \times B_1\right)}\le \sum_{k=0}^\infty 2^{-k(\alpha+1)} \left\| \int_{t-2^{k+1}}^{t-2^k} |D v(s,x)| \, ds \right\|_{L_{p_0}\left((t_0-1,t_0) \times B_1\right)}.
\end{equation*}
Since $t_0 - 1 < t < t_0$,
$$
\int_{t-2^{k+1}}^{t-2^k} |D v(s,x)|\, ds \leq \int_{t_0-(2^{k+1}+1)}^{t_0-2^k} |D v(s,x)|\, ds.
$$
Hence, by the Minkowski inequality,
\begin{align*}
&\left\| \int_{t-2^{k+1}}^{t-2^k } |D v(s,x)| \, ds \right\|_{L_{p_0}\left(Q_1(t_0,0)\right)}\\
&\le \int_{t_0 - (2^{k+1}+1)}^{t_0 - 2^k} \left( \int_{B_1} |D v(s,x)|^{p_0} \, dx \right)^{1/p_0} \, ds\\
&\le 2^{k+2}\left( \dashint_{\!t_0 - (2^{k+1}+1)}^{\,\,\,t_0} \dashint_{B_1} |D v(s,x)|^{p_0} \, dx \, ds \right)^{1/p_0}.
\end{align*}
It then follows that
\begin{align*}
&\|I_2\|_{L_{p_0}\left(Q_1(t_0,0)\right)}\\
&\le \sum_{k=0}^\infty 2^{-k\alpha+2}\left( \dashint_{\!t_0 - (2^{k+1}+1)}^{\,\,\,t_0} \dashint_{B_1} |D v(s,x)|^{p_0} \, dx \, ds \right)^{1/p_0}\\
&\le \sum_{k=0}^\infty 2^{-k\alpha+2} \left( \dashint_{\!t_0 - (2^{k+1}+1)}^{\,\,\,t_0} \dashint_{B_1} |D u(s,x)|^{p_0} \, dx \, ds \right)^{1/p_0}\\
&\quad + \sum_{k=0}^\infty 2^{-k\alpha+2}\left( \dashint_{\!t_0 - (2^{k+1}+1)}^{\,\,\,t_0} \dashint_{B_1} |D w(s,x)|^{p_0} \, dx \, ds \right)^{1/p_0},
\end{align*}
where
$$
\sum_{k=0}^\infty 2^{-k\alpha+2}\left( \dashint_{\!t_0 - (2^{k+1}+1)}^{\,\,\,t_0} \dashint_{B_1} |D w(s,x)|^{p_0} \, dx \, ds \right)^{1/p_0} \leq N(\alpha) \left( |D w|^{p_0} \right)^{1/p_0}_{Q_1(t_0,0)}.
$$
Combining the above inequalities, \eqref{eq0715_01}, and \eqref{eq0715_02}, we get
\begin{align*}
&\|D v\|_{L_{p_1}\left(Q_{1/2}(t_0,0)\right)} \leq N \left( |D w|^{p_0} \right)_{Q_1(t_0,0)}^{1/p_0}\\
&\quad + N \sum_{k=0}^\infty 2^{-k\alpha} \left( \dashint_{\!t_0 - (2^{k+1}+1)}^{\,\,\,t_0} \dashint_{B_1(x_0)} |D u(s,y)|^{p_0} \, dy \, ds \right)^{1/p_0}.
\end{align*}
We then use \eqref{eq1214_01} with $r=1$ to obtain \eqref{eq0411_01} with $R = 1$.
The proposition is proved.
\end{proof}

Let $\gamma\in (0,1)$, and let $p_0\in (1,\infty)$ and $p_1=p_1(d,\alpha,p_0)$ be from Proposition \ref{prop0406_1}.
Denote
\begin{equation}
							\label{eq0406_04}
\cA(\textsf{s}) = \left\{ (t,x) \in (-\infty,T) \times \bR^d: |D u(t,x)| > \textsf{s} \right\}
\end{equation}
and
\begin{multline}
							\label{eq0406_05}
\cB(\textsf{s}) = \big\{ (t,x) \in (-\infty,T) \times \bR^d:
\\
\gamma^{-1/p_0}\left( \cM |g|^{p_0} (t,x) \right)^{1/p_0} + \gamma^{-1/p_1}\left( \cS\cM |D u|^{p_0}(t,x)\right)^{1/p_0} > \textsf{s}  \big\},
\end{multline}
where, to well define $\cM$ and $\cS\cM$ (recall the definitions in \eqref{eq0406_03b} and \eqref{eq0406_03}), we extend a given function to be zero for $t \leq S$ if the function is defined on $(S,T) \times \bR^d$.

Set
\begin{equation*}
\cC_R(t,x) = (t-R^{2/\alpha},t+R^{2/\alpha}) \times B_R(x),\quad
\hat \cC_R(t,x)=\cC_R(t,x)\cap \{t\le T\}.
\end{equation*}

\begin{lemma}
							\label{lem0409_1}
Let $p_0\in (1,\infty)$, $\alpha \in (0,1)$, $T \in (0,\infty)$, $a^{ij} = a^{ij}(x_1)$, $R \in (0,\infty)$, and $\gamma \in (0,1)$.
Assume that Theorem \ref{thm0412_1} holds with this $p_0$ and $u \in \cH_{p_0,0}^{\alpha,1}(\bR^d_T)$ satisfies
$$
-\partial_t^\alpha u + D_i \left( a^{ij}(x_1) D_j u \right)
= D_i g_i
$$
in $(0,T) \times \bR^d$, where $g = (g_1,\ldots,g_d) \in L_p(\bR^d_T)$.
Then, there exists a constant $\kappa = \kappa(d,\delta,\alpha,p_0) > 1$ such that the following hold:
for $(t_0,x_0) \in (-\infty,T] \times \bR^d$ and $\textsf{s}>0$, if
\begin{equation}
							\label{eq0406_02}
|\cC_{R/4}(t_0,x_0) \cap \cA(\kappa s)|  \geq \gamma |\cC_{R/4}(t_0,x_0)|,
\end{equation}
then
$$
\hat\cC_{R/4}(t_0,x_0) \subset \cB(\textsf{s}).
$$
\end{lemma}

\begin{proof}
By dividing the equation by $\textsf{s}$, we may assume that $\textsf{s} = 1$.
We only consider $(t_0,x_0) \in (-\infty,T] \times \bR^d$ such that $t_0 + (R/4)^{2/\alpha} \geq 0$, because otherwise,
$$
\cC_{R/4}(t_0,x_0) \cap \cA(\kappa) \subset \left\{ (t,x) \in (-\infty,0] \times \bR^d: |D^2 u(t,x)| > s \right\} = \emptyset
$$
as $u(t,x)$ is extended to be zero for $t < 0$.
Suppose that there is a point $(s,y) \in \hat\cC_{R/4}(t_0,x_0)$
such that
\begin{equation}
							\label{eq0406_01}
\gamma^{-1/p_0}\left( \cM |g|^{p_0} (s,y) \right)^{1/p_0} + \gamma^{-1/p_1} \left( \cS\cM |D u|^{p_0}(s,y)\right)^{1/p_0} \leq 1.
\end{equation}
Set
$$
t_1 := \min \{ t_0 + (R/4)^{2/\alpha}, T\} \quad \text{and} \quad x_1 := x_0.
$$
Then $(t_1,x_1) \in [0,T] \times \bR^d$ and by Proposition \ref{prop0406_1} there exist $p_1 = p_1(d,\alpha,p_0) \in (p_0,\infty]$ and $w \in \cH_{p_0,0}^{\alpha,1}\left((t_1-R^{2/\alpha}, t_1) \times \bR^d\right)$, $v \in \cH_{p_0,0}^{\alpha,1}\left((S,t_1) \times \bR^d\right)$ with $S=\min\{0,t_1-R^{2/\alpha}\}$, such that $u = w + v$ in $Q_R(t_1,x_1)$,
\begin{equation}
							\label{eq0409_01}
\left(|D w|^{p_0}\right)^{1/p_0}_{Q_R(t_1,x_1)} \leq N \left( |g|^{p_0} \right)_{Q_{2R}(t_1,x_1)}^{1/p_0},
\end{equation}
and
\begin{multline}
							\label{eq0409_02}
\left( |Dv|^{p_1} \right)_{Q_{R/2}(t_1,x_1)}^{1/p_1} \le  N \left( |g|^{p_0} \right)_{Q_{2R}(t_1,x_1)}^{1/p_0}
\\
+ N \sum_{k=0}^\infty 2^{-k \alpha} \left(\dashint_{\! t_1 - (2^{k+1}+1)R^{2/\alpha}}^{\,\,\,t_1} \dashint_{B_R(x_1)} |D u(\ell,z)|^{p_0} \, dz \, d \ell \right)^{1/p_0},
\end{multline}
where $N=N(d,\delta, \alpha,p_0)$.
Since $t_0 \leq T$, we have
$$
(s,y) \in \hat\cC_{R/4}(t_0,x_0)
\subset Q_{R/2}(t_1,x_1) \subset Q_{2R}(t_1,x_1),
$$
$$
(s,y) \in \hat\cC_{R/4}(t_0,x_0) \subset (t_1- (2^{k+1}+1)R^{2/\alpha}, t_1) \times B_R(x_1)
$$
for all $k = 0,1,\ldots$.
From these set inclusions, in particular, we observe that
$$
\dashint_{\! t_1 - (2^{k+1}+1)R^{2/\alpha}}^{\,\,\,t_1} \dashint_{B_R(x_1)} |Du(\ell,z)|^{p_0} \, dz \, d \ell \leq \cS\cM |Du|^{p_0}(s,y)
$$
for all $k=0,1,2,\ldots$.
Thus the inequality \eqref{eq0406_01} along with \eqref{eq0409_01} and \eqref{eq0409_02} implies that
$$
\left( |Dv|^{p_1} \right)_{Q_{R/2}(t_1,x_1)}^{1/p_1} \leq N \gamma^{1/p_0} + N \gamma^{1/p_1} \leq N_0 \gamma^{1/p_1},
$$
$$
\left(|D w|^{p_0}\right)^{1/p_0}_{Q_R(t_1,x_1)} \leq N_1 \gamma^{1/p_0},
$$
where $N_0$ and $N_1$ depend only on $d$, $\delta$, $\alpha$, and $p_0$.
Note that, for a sufficiently large $K_1$,
\begin{align*}
&|\cC_{R/4}(t_0,x_0) \cap \cA(\kappa)|
= |\{(t,x) \in \cC_{R/4}(t_0,x_0), t \in (-\infty,T): |Du(t,x)| > \kappa\}|\\
&\leq \left|\{ (t,x) \in Q_{R/2}(t_1,x_1): |D u(t,x)| > \kappa\}\right|\\
&\leq \left|\{(t,x) \in Q_{R/2}(t_1,x_1): |D w(t,x)| > \kappa - K_1 \}\right|\\
&\quad + \left|\{(t,x) \in Q_{R/2}(t_1,x_1): |D v(t,x)| > K_1 \}\right|\\
&\leq (\kappa-K_1)^{-p_0} \int_{Q_{R/2}(t_1,x_1)} |D w|^{p_0} \, dx \, dt + K_1^{-p_1} \int_{Q_{R/2}(t_1,x_1)} |D v|^{p_1} \, dx \, dt\\
&\leq \frac{N_1^{p_0} \gamma |Q_R|}{(\kappa - K_1)^{p_0}} + \frac{N_0^{p_1}\gamma|Q_{R/2}|}{K_1^{p_1}}I_{p_1 \neq \infty}\\
&\leq N(d,\alpha) |\cC_{R/4}| \left(\frac{N_1^{p_0} \gamma }{(\kappa - K_1)^{p_0}} + \gamma \left(\frac{N_0}{K_1}\right)^{p_0} I_{p_1 \neq \infty} \right) < \gamma |\cC_{R/4}(t_0,x_0)|,
\end{align*}
provided that we choose a sufficiently large $K_1(\ge N_0)$ depending only on $d$, $\delta$, $\alpha$, and $p_0$, so that
$$
N(d,\alpha)  (N_0/K_1)^{p_0} < 1/2,
$$
and then choose a $\kappa$ depending only on $d$, $\delta$, $\alpha$, and $p_0$, so that
$$
N(d,\alpha) N_1^{p_0}/(\kappa-K_1)^{p_0} < 1/2.
$$
Considering \eqref{eq0406_02}, we get a contradiction.
The lemma is proved.
\end{proof}

\section{$L_p$-estimates}
                    \label{sec6}

\begin{proof}[Proof of Theorem \ref{thm0412_1}]
When $p=2$, the theorem follows from Proposition \ref{prop0720_1}.
Suppose that the theorem holds for some $p_0\in [2,\infty)$, which is indeed true for $p_0 = 2$.
Fix $p_1 \in (p_0, \infty]$ determined by $d$, $\alpha$, and $p_0$ as in Proposition \ref{prop0406_1}.
Now we prove Theorem \ref{thm0412_1} for $p\in (p_0,p_1)$.
To do this, we only prove the a priori estimates \eqref{eq0411_04}, \eqref{eq0904_06}, and \eqref{eq0905_04}.
Once these a prior estimates are available, the existence results follow from the method of continuity and the solvability of a simple equation as in the proof of Proposition \ref{prop0720_1}.
In particular,  the solvability of a simple equation is guaranteed by the results in \cite{arXiv:1806.02635} for non-divergence form equations and by the a priori estimates.

Let us first prove \eqref{eq0905_04}.
As in the proof of Proposition \ref{prop0720_1}, we assume $u \in C_0^\infty\left([0,T] \times \bR^d\right)$ with $u(0,x) = 0$.
Note that
\begin{equation}
                                    \label{eq8.47}
\|Du\|_{L_p(\bR^d_T)}^p = p \int_0^\infty |\cA(\textsf{s})| \textsf{s}^{p-1} \, d \textsf{s} = p
\kappa^p \int_0^\infty |\cA(\kappa \textsf{s})| \textsf{s}^{p-1} \, d \textsf{s}.
\end{equation}
By Lemmas \ref{lem0409_1} and \cite[Lemma A.20]{arXiv:1806.02635} it follows that
\begin{equation}
                                    \label{eq8.46}
|\cA(\kappa \textsf{s})| \leq N(d,\alpha) \gamma|\cB(\textsf{s})|
\end{equation}
for all $\textsf{s} \in (0,\infty)$.
Hence, by the Hardy-Littlewood maximal function theorem,
\begin{align*}
&\|D u\|_{L_p(\bR^d_T)}^p
\leq N p  \kappa^p \gamma \int_0^\infty |\cB(\textsf{s})| \textsf{s}^{p-1} \, d \textsf{s}\\
&\le N\gamma \int_0^\infty\left|\left\{ (t,x) \in (-\infty,T) \times \bR^d:\gamma^{-\frac 1{ p_1}}\left( \cS\cM |Du|^{p_0}(t,x)\right)^{\frac 1 {p_0}} > \textsf{s}/2 \right\}\right| \textsf{s}^{p-1} \, d \textsf{s}\\
&\quad +
N\gamma \int_0^\infty\left|\left\{ (t,x) \in (-\infty,T) \times \bR^d:\gamma^{-\frac 1 {p_0}}\left( \cM |g|^{p_0} (t,x) \right)^{\frac 1 {p_0}} > \textsf{s}/2 \right\}\right| \textsf{s}^{p-1} \, d \textsf{s}\\
&\leq N \gamma^{1-p/p_1} \|Du\|^p_{L_p(\bR^d_T)} + N \gamma^{1-p/p_0} \|g\|^p_{L_p(\bR^d_T)},
\end{align*}
where $N = N(d,\delta,\alpha,p)$.
By choosing $\gamma \in (0,1)$ so that
$$
N \gamma^{1-p/p_1} < 1/2,
$$
which is possible because $p\in (p_0,p_1)$, we obtain \eqref{eq0905_04}.

Next, we prove \eqref{eq0411_04}, which follows easily from \eqref{eq0905_04} just proved above and S. Agmon's idea.
Indeed, by following the proof of \cite[Lemma 5.5]{MR2304157}, we obtain \eqref{eq0411_04} for $\lambda \in [\lambda_0, \infty)$, where $\lambda_0$ is sufficiently large number.
For $\lambda \in (0,\lambda_0)$, we use a dilation argument.

Finally, as in the proof of Proposition \ref{prop0720_1}, we obtain the estimate \eqref{eq0904_06} by using \eqref{eq0411_04} and Lemma \ref{lem0904_1}.

Now that we have proved Theorem \ref{thm0412_1} for $p \in [p_0,p_1)$.
We repeat this procedure until we have $p_1 = \infty$.
Indeed, this can be accomplished by finitely many iterations because of \eqref{eq1029_01}, which shows that each time the increment from $p_0$ to $p_1$ can be made bigger than a positive number depending only on $d$ and $\alpha$.
Thus, in finite steps we get a $p_0$ which is larger than $d + 1/\alpha$, so that $p_1=p_1(d,\alpha,p_0)=\infty$. Therefore, the theorem is proved for any $p\in [2,\infty)$.

For $p \in (1,2)$, we use a duality argument.
As above, we only prove the a priori estimates \eqref{eq0411_04}, \eqref{eq0904_06}, and \eqref{eq0905_04}.
We again assume $u \in C_0^\infty\left([0,T] \times \bR^d\right)$ with $u(0,x) = 0$.
Let $\phi_\ell \in L_q(\bR^d_T)$, $\ell = 1, \ldots, d$, and $\psi \in L_q(\bR^d_T)$, where $1/p+1/q=1$.
Then
$$
\phi_\ell(-t,x), \,\, \psi(-t,x) \in L_q\left((-T,0) \times \bR^d\right).
$$
Using Theorem \ref{thm0412_1} for $q \in [2,\infty)$, find $w \in \bH_{q,0}^{\alpha,1}\left((-T,0) \times \bR^d\right)$ satisfying
$$
- \partial_t^\alpha w + D_i\left(a^{ji}(x_1) D_j w \right) - \lambda w = D_\ell \left(-\phi_\ell (-t,x)\right) + \psi(-t,x)
$$
in $(-T,0) \times \bR^d$
with the estimates
\begin{multline}
							\label{eq1205_01}
\|Dw\|_{L_q\left((-T,0) \times \bR^d\right)} + \sqrt{\lambda} \|w\|_{L_q\left((-T,0) \times \bR^d\right)}
\\
\leq N \|\phi_\ell(-t,x)\|_{L_q\left((-T,0) \times \bR^d\right)} + \frac{N}{\sqrt{\lambda}} \|\psi(-t,x)\|_{L_q\left((-T,0) \times \bR^d\right)}
\\
= N \|\phi_\ell\|_{L_q(\bR^d_T)} + \frac{N}{\sqrt{\lambda}} \|\psi\|_{L_q(\bR^d_T)},
\end{multline}
provided that $\lambda > 0$ or $\lambda = 0$ with $\psi \equiv 0$,
where $N = N(d,\delta,\alpha,q)$, and
\begin{equation}
							\label{eq1205_02}
\|Dw\|_{L_q\left((-T,0) \times \bR^d\right)} + \|w\|_{L_q\left((-T,0) \times \bR^d\right)} \leq N \|\phi_\ell\|_{L_q(\bR^d_T)} + N \|\psi\|_{L_q(\bR^d_T)},
\end{equation}
provided that $\lambda = 0$ with $\psi \not\equiv  0$, where $N = N(d,\delta,\alpha,q, T)$.
We here note that
$$
\partial_t^\alpha w = \partial_t I_{-T}^{1-\alpha} w.
$$
Set $\varphi(t,x) = u(-t,x)$ and $\tilde{\varphi}(t,x) = w(-t,x)$.
Considering $w_k \in C_0^\infty\left([-T,0] \times \bR^d\right)$ with $w_k(-T,0) = 0$ such that $w_k \to w$ in $\bH_{q,0}^{\alpha,1}\left((-T,0) \times \bR^d \right)$, we observe that
$$
\int_0^T \int_{\bR^d} \left(\phi_\ell D_\ell u + \psi u\right)  \, dx \, dt
$$
$$
= \int_{-T}^0 \int_{\bR^d} \left( \phi_\ell(-t,x) D_\ell \varphi(t,x) + \psi(-t,x) \varphi(t,x) \right) \, dx \, dt
$$
$$
= \int_{-T}^0 \int_{\bR^d} \left( I_{-T}^{1-\alpha} w \, \varphi_t - a^{ji}(x_1) D_j w D_i \varphi - \lambda w \varphi \right) \, dx \, dt
$$
$$
= \int_0^T \int_{\bR^d} \left( I_0^{1-\alpha} u \, \tilde{\varphi}_t - a^{ij}(x_1)D_j u D_i \tilde{\varphi} - \lambda u \tilde{\varphi} \right) \, dx \, dt = \int_0^T \int_{\bR^d} \left(f \tilde{\varphi} - g_i D_i \tilde{\varphi} \right) \, dx \, dt
$$
$$
\leq \|f\|_{L_p(\bR^d_T)} \|w(-t,x)\|_{L_q(\bR^d_T)} + \|g_i\|_{L_p(\bR^d_T)} \|D_iw(-t,x)\|_{L_q(\bR^d_T)}.
$$
This together with the estimates \eqref{eq1205_01} and \eqref{eq1205_02} for $w$ implies \eqref{eq0411_04}, \eqref{eq0904_06}, and \eqref{eq0905_04}.
The theorem is proved.
\end{proof}

To prove Theorem \ref{main_thm}, we extend Proposition \ref{prop0406_1} to the case when $a^{ij}=a^{ij}(t,x)$ satisfy Assumption \ref{assump2.2}.

\begin{proposition}
							\label{prop0515_1}
Let $p_0 \in (1,\infty)$, $\alpha,\gamma_0 \in (0,1)$, $T \in (0,\infty)$, $\mu\in (1,\infty)$, $\nu=\mu/(\mu-1)$, and $a^{ij} = a^{ij}(t,x)$ satisfy Assumption \ref{assump2.2} ($\gamma_0$).
Assume that $u \in \cH_{p_0,0}^{\alpha,1}(\bR^d_T)$ vanishes for $(t,x) \notin Q_{R_0}(t_1,x_1)$ for some $(t_1,x_1)\in \bR^{d+1}$, and satisfies
\begin{equation}
							\label{eq1207_01}
-\partial_t I_0^{1-\alpha}u + D_i\left(a^{ij}(t,x) D_j u\right) = D_i g_i
\end{equation}
in $\bR^d_T$, where $g = (g_1,\ldots,g_d) \in L_{p_0}(\bR^d_T)$.
Then there exists
$$
p_1 = p_1(d,\alpha,p_0)\in (p_0,\infty]
$$
satisfying \eqref{eq1029_01} and the following.
For any $(t_0,x_0) \in [0,T] \times \bR^d$ and $R \in (0,\infty)$,
there exist
$$
w \in \cH_{p_0,0}^{\alpha,1}((t_0-R^{2/\alpha}, t_0)\times \bR^d), \quad v \in \cH_{p_0,0}^{\alpha,1}((S, t_0)\times \bR^d),
$$
where $S := \min\{0, t_0-R^{2/\alpha}\}$,
such that $u = w + v$ in $Q_R(t_0,x_0)$,
\begin{equation*}
\left( |Dw|^{p_0} \right)_{Q_R(t_0,x_0)}^{\frac{1}{p_0}} \le N \left( |g_i|^{p_0} \right)_{Q_{2R}(t_0,x_0)}^{\frac{1}{p_0}} + N \gamma_0^{\frac{1}{p_0\nu}} \left( |Du|^{p_0\mu} \right)_{Q_{2R}(t_0,x_0)}^{\frac{1}{p_0\mu}},
\end{equation*}
and
\begin{multline*}
\left( |Dv|^{p_1} \right)_{Q_{R/2}(t_0,x_0)}^{\frac{1}{p_1}} \leq N \left( |g_i|^{p_0}\right)_{Q_{2R}(t_0,x_0)}^{\frac{1}{p_0}} + N \gamma_0^{\frac{1}{p_0\nu}} \left( |Du|^{p_0\mu} \right)_{Q_{2R}(t_0,x_0)}^{\frac{1}{p_0\mu}}
\\
+ N \sum_{k=0}^\infty 2^{-k\alpha} \left( \dashint_{\!t_0 - (2^{k+1}+1)R^{2/\alpha}}^{\,\,\,t_0} \dashint_{B_R(x_0)} |Du(s,y)|^{p_0} \, dy \, ds \right)^{\frac{1}{p_0}},
\end{multline*}
where $N=N(d,\delta, \alpha,p_0,\mu)$.
\end{proposition}

\begin{proof}
Denote
$$
Q:=(t_0-(2R)^{2/\alpha}, t_0) \times ({x_0}_1-2R,{x_0}_1+2R) \times B_{2R}'(x_0')
$$
if $2R\le R_0$, and
$$
Q:= (t_1-R_0^{2/\alpha}, t_1) \times ({x_1}_1-R_0,{x_1}_1+R_0) \times B_{R_0}'(x_1')
$$
if $2R > R_0$,
where ${x_0}_1$ and ${x_1}_1$ are the first coordinates of $x_0$ and $x_1 \in \bR^d$, respectively.
Note that in both cases
$$
\frac{|Q|}{\left|(t_0-(2R)^{2/\alpha}, t_0) \times ({x_0}_1-2R,{x_0}_1+2R) \times B_{2R}'(x_0')\right|} \leq 1.
$$
Then, by Assumption \ref{assump2.2}, there exist $\bar a^{ij}=\bar a^{ij}(x_1)$ such that
\begin{multline}
							\label{eq1207_02}
\sup_{i,j} \dashint_{\!{x_0}_1 - 2R}^{\,\,\,{x_0}_1+2R} \dashint_{Q'_{2R}(t_0,x_0')} \left|a^{ij}(s,y_1,y') - \bar{a}^{ij}(y_1)\right| 1_Q \, dy' ds \, d y_1
\\
= \sup_{i,j} \dashint_{\!{x_0}_1 - 2R}^{\,\,\,{x_0}_1+2R} \dashint_{Q'_{2R}(t_0,x_0')} \left|a^{ij}(s,y_1,y') - \bar{a}^{ij}(y_1)\right| \, dy' ds \, d y_1 \leq \gamma_0
\end{multline}
if $2R \leq R_0$, and
\begin{multline}
							\label{eq1207_03}
\sup_{i,j} \dashint_{\!{x_0}_1 - 2R}^{\,\,\,{x_0}_1+2R} \dashint_{Q'_{2R}(t_0,x_0')} \left|a^{ij}(s,y_1,y') - \bar{a}^{ij}(y_1)\right| 1_Q \, dy' ds \, d y_1
\\
\leq \sup_{i,j} \dashint_{\!{x_1}_1 - R_0}^{\,\,\,{x_1}_1 + R_0} \dashint_{Q'_{R_0}(t_1,x_1')} \left|a^{ij}(s,y_1,y') - \bar{a}^{ij}(y_1)\right| \, dy' ds \, d y_1 \leq \gamma_0
\end{multline}
if $2R > R_0$, where $1_Q$ is the indicator function of $Q$.
We then rewrite \eqref{eq1207_01} into
$$
-\partial_t^\alpha u +  D_i\left(\bar a^{ij}(x_1)D_j u \right) = D_i( \bar{g}_i ),
$$
where
$$
\bar{g}_i = g_i + \left(\bar a^{ij}(x_1)-a^{ij}(t,x)\right)D_j u.
$$
Now that Theorem \ref{thm0412_1} holds for this equation with $p_0$,  it follows from Proposition \ref{prop0406_1} that there exist
$$
w \in \cH_{p_0,0}^{\alpha,1}((t_0-R^{2/\alpha}, t_0)\times \bR^d), \quad v \in \cH_{p_0,0}^{\alpha,1}((S, t_0)\times \bR^d),
$$
where $S := \min\{0, t_0-R^{2/\alpha}\}$,
such that $u = w + v$ in $Q_R(t_0,x_0)$, and \eqref{eq8.13} and \eqref{eq0411_01} hold with $\bar{g}_i$ in place of $g_i$. To conclude the proof, it remains to notice that
$$
\left( |\bar{g}_i|^{p_0} \right)_{Q_{2R}(t_0,x_0)}^{\frac{1}{p_0}}
\le \left( |g_i|^{p_0} \right)_{Q_{2R}(t_0,x_0)}^{\frac{1}{p_0}} + \left( |(\bar a^{ij}(x_1)-a^{ij})D_j u|^{p_0} \right)_{Q_{2R}(t_0,x_0)}^{\frac{1}{p_0}},
$$
where by H\"older's inequality, \eqref{eq1207_02}, \eqref{eq1207_03}, the fact that $u$ has compact support in $Q_{R_0}(t_1,x_1)$, and the inclusions $Q_{2R}(t_0,x_0) \subset Q$ and $Q_{R_0}(t_1,x_1) \subset Q$ when $2R \leq R_0$ and $2R > R_0$, respectively, we have
$$
\left( |(\bar a^{ij}(x_1)-a^{ij})D_j u|^{p_0} \right)_{Q_{2R}(t_0,x_0)}^{\frac{1}{p_0}} = \left( |(\bar a^{ij}(x_1)-a^{ij}) 1_Q D_j u|^{p_0} \right)_{Q_{2R}(t_0,x_0)}^{\frac{1}{p_0}}
$$
$$
\leq \left( N_0 \dashint_{\!{x_0}_1 - 2R}^{\,\,\,{x_0}_1+2R} \dashint_{Q'_{2R}(t_0,x_0')} |\bar{a}^{ij}(x_1) - a^{ij}|^{p_0 \nu} 1_Q \, dx \, dt \right)^{\frac{1}{p_0 \nu}} \left( |Du|^{p_0\mu} \right)_{Q_{2R}(t_0,x_0)}^{\frac{1}{p_0 \mu}}
$$
$$
\leq N \gamma_0^{\frac{1}{p_0 \nu}} \left( |Du|^{p_0\mu} \right)_{Q_{2R}(t_0,x_0)}^{\frac{1}{p_0 \mu}}.
$$
In particular,
$$
N_0 = \frac{\left|(t_0-(2R)^{2/\alpha}, t_0) \times ({x_0}_1-2R,{x_0}_1+2R) \times B_{2R}'(x_0')\right|}{|Q_{2R}(t_0,x_0)|}
$$
$$
= \frac{2\Gamma(d/2+1)}{\sqrt{\pi}\Gamma(d/2+1/2)}.
$$
The proposition is proved.
\end{proof}

Now we define $\cA(\textsf{s})$ as in \eqref{eq0406_04}, but instead of \eqref{eq0406_05} we define
\begin{multline*}
\cB(\textsf{s}) = \big\{ (t,x) \in (-\infty,T) \times \bR^d:
\gamma^{-\frac{1}{p_0}}\left( \cM |g|^{p_0} (t,x) \right)^{\frac{1}{p_0}}
\\
+\gamma^{-\frac{1}{p_0}}\gamma_0^{\frac{1}{p_0\nu}}\left( \cM |D u|^{p_0\mu} (t,x) \right)^{\frac{1}{p_0\mu}}
+ \gamma^{-\frac{1}{p_1}}\left( \cS\cM |D u|^{p_0}(t,x)\right)^{\frac{1}{p_0}} > \textsf{s}  \big\}.
\end{multline*}

By following the proof of Lemma \ref{lem0409_1} with minor modifications, from Proposition \ref{prop0515_1}, we get the following lemma.

\begin{lemma}
							\label{lem6.3}
Let $p\in (1,\infty)$, $\alpha,\gamma_0,\gamma \in (0,1)$, $T \in (0,\infty)$, $R \in (0,\infty)$, $\mu\in (1,\infty)$, $\nu=\mu/(\mu-1)$, and $a^{ij} = a^{ij}(t,x)$ satisfy Assumption \ref{assump2.2} ($\gamma_0$).
Assume that $u \in \cH_{p,0}^{\alpha,1}(\bR^d_T)$ vanishes for $(t,x) \notin Q_{R_0}(t_1,x_1)$ for some $(t_1,x_1)\in \bR^{d+1}$, and satisfies \eqref{eq1207_01} in $(0,T) \times \bR^d$.
Then, there exists a constant $\kappa = \kappa(d,\delta,\alpha,p_0,\mu) > 1$ such that the following hold:
for $(t_0,x_0) \in (-\infty,T] \times \bR^d$ and $\textsf{s}>0$, if
\begin{equation*}
|\cC_{R/4}(t_0,x_0) \cap \cA(\kappa \textsf{s})|  \geq \gamma |\cC_{R/4}(t_0,x_0)|,
\end{equation*}
then
$$
\hat\cC_{R/4}(t_0,x_0) \subset \cB(\textsf{s}).
$$
\end{lemma}

\begin{proposition}
							\label{prop1218_1}
Let $\alpha \in (0,1)$, $T \in (0,\infty)$, and $p \in (1,\infty)$.
There exists $\gamma_0\in (0,1)$ depending only on $d$, $\delta$, $\alpha$, and $p$, such that, under Assumption \ref{assump2.2} ($\gamma_0$), the following hold.
Suppose that $u \in \cH_{p,0}^{\alpha,1}(\bR^d_T)$ vanishes for $(t,x) \notin Q_{R_0}(t_1,x_1)$ for some $(t_1,x_1)\in \bR^{d+1}$, and satisfies
\begin{equation*}
-\partial_t^\alpha u + D_i\left( a^{ij} D_j u \right) = D_i g_i
\end{equation*}
in $\bR^d_T$, where $g = (g_1,\ldots,g_d)$, $g_i \in L_p(\bR^d_T)$, $i=1,\ldots,d$.
Then
\begin{equation}
							\label{eq1207_04}
\|Du\|_{L_p(\bR_T)} \leq N \|g_i\|_{L_p(\bR_T)},
\end{equation}
where $N = N(d,\delta,\alpha,p)$.
\end{proposition}

\begin{proof}
We take $p_0\in (1,p)$ and $\mu\in (1,\infty)$ depending only on $p$ such that $p_0<p_0\mu<p<p_1$, where $p_1=p_1(d,\alpha,p_0)$ is taken from Proposition \ref{prop0515_1}. By Lemma \ref{lem6.3} and \cite[Lemma A.20]{arXiv:1806.02635}, we have \eqref{eq8.46}, which together with \eqref{eq8.47} and the Hardy-Littlewood maximal function theorem implies that
{\small \begin{align*}
&\|Du\|_{L_p(\bR^d_T)}^p
\leq N p  \kappa^p \gamma \int_0^\infty |\cB(\textsf{s})| \textsf{s}^{p-1} \, d\textsf{s}\\
&\le N\gamma \int_0^\infty\left|\left\{ (t,x) \in (-\infty,T) \times \bR^d:\gamma^{-\frac 1 {p_1}}\left( \cS\cM |D u|^{p_0}(t,x)\right)^{\frac 1 {p_0}} > \frac{\textsf{s}}{3} \right\}\right| \textsf{s}^{p-1} \, d\textsf{s}\\
&\ +
N\gamma \int_0^\infty\left|\left\{ (t,x) \in (-\infty,T) \times \bR^d:\gamma^{-\frac 1 {p_0}}\left( \cM |g|^{p_0} (t,x) \right)^{\frac 1 {p_0}} > \frac{\textsf{s}}{3} \right\}\right| \textsf{s}^{p-1} \, d\textsf{s}\\
&\ +
N\gamma \int_0^\infty\left|\left\{ (t,x) \in (-\infty,T) \times \bR^d:\gamma^{-\frac 1 {p_0}}\gamma_0^{\frac 1 {p_0\nu}}
\left( \cM |D u|^{p_0\mu} (t,x) \right)^{\frac 1 {p_0\mu}} > \frac{\textsf{s}}{3} \right\}\right| \textsf{s}^{p-1} \, d\textsf{s}\\
&\leq N (\gamma^{1-p/p_1}+\gamma^{1-p/p_0}\gamma_0^{p/(p_0\nu)}) \|Du\|^p_{L_p(\bR^d_T)} + N \gamma^{1-p/p_0} \|g\|^p_{L_p(\bR^d_T)},
\end{align*}}
where $N = N(d,\delta,\alpha,p)$.
Now choose $\gamma \in (0,1)$ sufficiently small and then $\gamma_0$ sufficiently small, depending only on $d$, $\delta$, $\alpha$, and $p$, so that
$$
N (\gamma^{1-p/p_1}+\gamma^{1-p/p_0}\gamma_0^{p/(p_0\nu)}) < 1/2.
$$
Then we obtain \eqref{eq1207_04}.
\end{proof}

Finally, we give the proof of Theorem \ref{main_thm} after the following two lemmas.

\begin{lemma}
							\label{lem1218_01}
Let $\lambda > 0$, $\alpha \in (0,1)$, $T \in (0,\infty)$, and $p \in (1,\infty)$. There exists $\gamma_0\in (0,1)$ depending only on $d$, $\delta$, $\alpha$, and $p$, such that, under Assumption \ref{assump2.2} ($\gamma_0$), the following hold.
Suppose that $u \in \cH_{p,0}^{\alpha,1}(\bR^d_T)$ vanishes for $(t,x) \notin Q_{R_0/\sqrt{2}}(t_1,x_1)$ for some $(t_1,x_1)\in \bR^{d+1}$ and satisfies
$$
-\partial_t^\alpha u + D_i \left( a^{ij} D_j u \right) - \lambda u = D_i g_i + f
$$
in $\bR^d_T$, where $g_i, f \in L_p(\bR^d_T)$.
Then
$$
\sqrt{\lambda}\|u\|_{L_p(\bR^d_T)} + \|D u\|_{L_p(\bR^d_T)} \leq N \|g\|_{L_p(\bR^d_T)} + \frac{N}{\sqrt{\lambda}}\|f\|_{L_p(\bR^d_T)},
$$
provided that $\lambda \geq \lambda_0 > 0$, where $N = N(d,\delta,\alpha,p)$ and $\lambda_0 = \lambda_0(d,\delta,\alpha,p,R_0)$.
\end{lemma}

\begin{proof}
We repeat the proof of Lemma 5.5 in \cite{MR2304157} by using Proposition \ref{prop1218_1} and S. Agmon's idea.
In particular, using  a scaling argument, one can check that $N$ can be chosen independent of $R_0$.
\end{proof}

\begin{lemma}
							\label{lem1221_1}
Let $\lambda > 0$, $\alpha \in (0,1)$, $T \in (0,\infty)$, $p \in (1,\infty)$, and the lower-order coefficients $a^i, b^k, c$ satisfy the assumption \eqref{eq1221_01}.
There exists $\gamma_0\in (0,1)$ depending only on $d$, $\delta$, $\alpha$, and $p$, such that, under Assumption \ref{assump2.2} ($\gamma_0$), the following hold.
If $u \in \cH_{p,0}^{\alpha,1}(\bR^d_T)$ satisfies
$$
-\partial_t^\alpha u + D_i \left( a^{ij} D_j u + a^i u \right) + b^i D_i u + c u - \lambda u = D_i g_i + f
$$
in $\bR^d_T$, where $g_i, f \in L_p(\bR^d_T)$, then
\begin{equation}
							\label{eq1222_01}
\sqrt{\lambda}\|u\|_{L_p(\bR^d_T)} + \|D u\|_{L_p(\bR^d_T)} \leq N \|g\|_{L_p(\bR^d_T)} + \frac{N}{\sqrt{\lambda}}\|f\|_{L_p(\bR^d_T)},
\end{equation}
provided that $\lambda \geq \max\{T^{1-\alpha}\lambda_0, \lambda_0\} > 0$, where $N = N(d,\delta,\alpha,p)$ and $\lambda_0$ depends only on $d$, $\delta$, $\alpha$, $p$, $K$, and $R_0$.
\end{lemma}

\begin{proof}
We first assume that $a^i = b^i = c = 0$.
To use a partition of unity argument, we find sequences of $\{t_k,x_k\} \subset (0,T] \times \bR^d$ and $\{\eta_k(t)\zeta_k(x)\}$ such that,
$$
\eta_k(t)\zeta_k(x) \geq 0, \quad
\eta_k(t)\zeta_k(x) \in C_0^\infty(\bR^{d+1}), \quad \operatorname{supp} (\eta_k(t)\zeta_k(x)) \subset Q_{R_0/\sqrt{2}}(t_k,x_k),
$$
and, for $(t,x) \in \bR^d_T$,
\begin{equation}
							\label{eq1221_03}
\begin{aligned}
1 \leq & \sum_{k=1}^\infty |\eta_k(t)\zeta_k(x)|^p \leq \chi_0, \quad \sum_{k=1}^\infty |\zeta_k(x)|^p \leq \chi_0,
\\
&\sum_{k=1}^\infty |\eta_k(t) D_x\zeta_k(x)|^p \leq \chi_1,	
\end{aligned}
\end{equation}
where $\chi_0$ depends only on $p$, and $\chi_1$ depends only on $d$, $\alpha$, $p$, and $R_0$.
Note that $u_k(t,x) : = u(t,x) \eta_k(t)\zeta_k(x)$ satisfies
\begin{multline}
							\label{eq1221_02}
-\partial_t^\alpha u_k + D_i \left(a^{ij} D_j u_k \right) - \lambda u_k = D_i \left(g_i \eta_k\zeta_k + a^{ij} u \eta_k D_j \zeta_k \right)
\\
+ f \eta_k\zeta_k + F_k \zeta_k  + a^{ij} (D_j u) \eta_k D_i\zeta_k - g_i \eta_k D_i \zeta_k
\end{multline}
in $\bR^d_T$, where
$$
F_k(t,x) = \frac{\alpha}{\Gamma(1-\alpha)} \int_0^t (t-s)^{-\alpha-1} \left( \eta_k (s) - \eta_k (t) \right) u(s,x) \, ds.
$$
As in Lemma \ref{lem0207_1}, we see that
$$
\|F_k\|_{L_p(\bR^d_T)} \leq N T^{1-\alpha}\|u\|_{L_p(\bR^d_T)},
$$
where $N$ depends on $\alpha$, $p$, and $R_0$. See also Lemma A.2 in \cite{arXiv:1806.02635}.
By applying Lemma \ref{lem1218_01} to \eqref{eq1221_02}, summing in $k$, and using the inequalities in \eqref{eq1221_03}, we obtain that
$$
\sqrt{\lambda}^p \|u\|_{L_p(\bR^d_T)}^p + \|Du\|_{L_p(\bR^d_T)}^p \leq N_0 \|g\|_{L_p(\bR^d_T)}^p + N_1 \|u\|_{L_p(\bR^d_T)}^p + \frac{N_0}{\sqrt{\lambda}^p} \|f\|_{L_p(\bR^d_T)}^p
$$
$$
+ \frac{N_1 T^{(1-\alpha)p}}{\sqrt{\lambda}^p} \|u\|_{L_p(\bR^d_T)}^p + \frac{N_1}{\sqrt{\lambda}^p}\|Du\|_{L_p(\bR^d_T)}^p + \frac{N_1}{\sqrt{\lambda}^p}\|g\|_{L_p(\bR^d_T)}^p,
$$
where $N_0 = N_0(d, \delta,\alpha,p)$ and $N_1=N_1(d,\delta,\alpha,p,R_0)$.
Then we choose a sufficiently large $\lambda_0$ depending only on $d$, $\delta$, $\alpha$, $p$, and $R_0$ to obtain the estimate \eqref{eq1222_01}.
Precisely, we choose $\lambda_0$ so that, for $\lambda \geq \max\{T^{1-\alpha}\lambda_0, \lambda_0\} > 0$, we have
$$
\frac{\sqrt{\lambda}^p}{2} \leq \sqrt{\lambda}^p - N_1 - \frac{N_1 T^{(1-\alpha)p}}{\sqrt{\lambda}^p}.
$$

In the general case, we move the lower-order terms to the right-hand side as
$$
-\partial_t^\alpha u + D_i \left( a^{ij} D_j u \right) - \lambda u = D_i (g_i - a^iu) + f - b^i D_i u - cu.
$$
Then we apply the estimate just obtained above.
In this case, the choice of $\lambda_0$ depends also on $K$. \end{proof}

\begin{proof}[Proof of Theorem \ref{main_thm}]
As above, we only prove the estimate \eqref{eq0411_04c}.
We first deal with the case $\Omega = \bR^d$.

For $u \in \cH_{p,0}^{\alpha,1}(\bR^d_T)$ satisfying \eqref{eq0411_03c}, we consider the equation in $(0,\tau) \times \bR^d$, where $\tau \in (0,T]$.
By adding $-\lambda u$ to both sides of the equation, we have
$$
-\partial_t^\alpha u + D_i \left( a^{ij} D_j u + a^i u \right) + b^i D_i u + c u -\lambda u = D_i g_i + f - \lambda u
$$
in $\bR^d_\tau$.
By applying Lemma \ref{lem1221_1} to the above equation, we have $\lambda_0$ depending only on $d$, $\delta$, $\alpha$, $p$, $K$, and $R_0$ such that the following estimate holds provided that $\lambda \geq \max\{\tau^{1-\alpha}\lambda_0,\lambda_0\}$.
\begin{equation}
							\label{eq1222_02}
\sqrt{\lambda}\|u\|_{L_p(\bR^d_\tau)} + \|D u\|_{L_p(\bR^d_\tau)} \leq N \|g\|_{L_p(\bR^d_\tau)} + \frac{N}{\sqrt{\lambda}}\|f\|_{L_p(\bR^d_\tau)} + N \sqrt{\lambda}\|u\|_{L_p(\bR^d_\tau)},
\end{equation}
where $N = N(d,\delta,\alpha,p)$.
Now we fix $\lambda_1 = \max\{T^{1-\alpha}\lambda_0, \lambda_0, 1\}$.
Then the estimate \eqref{eq1222_02} holds with this $\lambda_1$ in place of $\lambda$ for all $\tau \in (0,T]$.
This implies that
\begin{equation}
							\label{eq1209_06}
\|u\|_{L_p(\bR^d_\tau)} + \|Du\|_{L_p(\bR^d_\tau)} \leq N \|g\|_{L_p(\bR^d_\tau)} + N \|f\|_{L_p(\bR^d_\tau)} + N_1 \|u\|_{L_p(\bR^d_\tau)}
\end{equation}
for any $\tau \in (0,T]$,
where, in particular, $N$ and $N_1$ are independent of $\tau$. ($N_1$ may depend on $d$, $\delta$, $\alpha$, $p$, $K$, $R_0$, and $T$.)
Thus, to complete the proof of \eqref{eq0411_04c}, it only remains to prove
\begin{equation}
							\label{eq1209_04}
\|u\|_{L_p(\bR^d_T)}\le N\|g\|_{L_p(\bR^d_T)} + N\|f\|_{L_p(\bR^d_T)}.
\end{equation}
We first prove \eqref{eq1209_04} for $p \geq 2$.
In this case, by moving the lower-order terms to the right-hand side, from Lemma \ref{lem0904_1} and the  estimate \eqref{eq1209_06}, we obtain that, for any $\tau \in (0,T]$,
\begin{multline}
							\label{eq1209_05}
\sup_{0 < t < \tau}\|I_0^{1-\alpha}|u|^p(t,\cdot)\|_{L_1(\bR^d)} \leq N (\tau^{\alpha(p-2)/2} + \tau^{\alpha(p-1)})
\|g\|^p_{L_p(\bR^d_\tau)}
\\
+ N \tau^{\alpha(p-1)} \|f\|^p_{L_p(\bR^d_\tau)} + N_1 (\tau^{\alpha(p-2)/2} + \tau^{\alpha(p-1)}) \|u\|^p_{L_p(\bR^d_\tau)}.
\end{multline}
Note that, for $0 \leq \tau_0 < \tau_1 \leq T$,
\begin{multline}
							\label{eq1209_03}
 \int_{\bR^d} \int_{\tau_0}^{\tau_1} |u(s,x)|^p \, ds \,dx \leq \int_{\bR^d}   (\tau_1-\tau_0)^\alpha \int_0^{\tau_1} (\tau_1 - s)^{-\alpha} |u(s,x)|^p \, ds \, dx
\\
= (\tau_1 - \tau_0)^\alpha \Gamma(1-\alpha) \int_{\bR^d} I_0^{1-\alpha} |u|^p(\tau_1,x) \, dx.
\end{multline}
Take a sufficiently large integer $m=m(d,\delta,\alpha,p, K, T,R_0)$ such that $$
N_1 (\tau^{\alpha(p-2)/2} + \tau^{\alpha(p-1)})(T/m)^{\alpha} \leq N_1 (T^{\alpha(p-2)/2} + T^{\alpha(p-1)})(T/m)^{\alpha} \le \frac{1}{2\Gamma(1-\alpha)}
$$
for any $\tau \in (0,T]$.
Then for any $j=0,2,\ldots,m-1$, by using \eqref{eq1209_03} and \eqref{eq1209_05} we have
$$
\int_{jT/m}^{(j+1)T/m} \int_{\bR^d} |u(s,x)|^p \, dx \, ds  \leq \left(\frac{T}{m}\right)^\alpha \Gamma(1-\alpha) \| I_0^{1-\alpha}|u|^p((j+1)T/m,\cdot)\|_{L_1(\bR^d)}
$$
$$
\leq \left(\frac{T}{m}\right)^\alpha \Gamma(1-\alpha) \sup_{t \in (0,(j+1)T/m)} \|I_0^{1-\alpha}|u|^p(t,\cdot)\|_{L_1(\bR^d)}
$$
$$
\leq N\|g\|^p_{L_p(\bR^d_T)} + N \|f\|^p_{L_p(\bR^d_T)} + \frac{1}{2} \|u\|_{L_p((0,(j+1)T/m);L_p(\bR^d))}^p.
$$
This implies that
$$
\|u\|_{L_p((jT/m,(j+1)T/m);L_p(\bR^d))}
\le N\|g\|_{L_p(\bR^d_T)} + N\|f\|_{L_p(\bR^d_T)}+\|u\|_{L_p((0,jT/m);L_p(\bR^d))}.
$$
By an induction on $j$, we obtain \eqref{eq1209_04} for $p \geq 2$.
To prove \eqref{eq1209_04} for $p \in (1,2)$, we use the duality argument as in the proof of Theorem \ref{thm0412_1}.
Note that when considering the dual equation in $(-T,0) \times \bR^d$, we take the even extension of $a^{ij}$ with respect to $t=0$.

To prove the case $\Omega = \bR^d_+$, we use the method of odd/even extensions; see, for instance, the proof of Theorem 2.4 in \cite{MR2764911}.
\end{proof}

\appendix

\section{Sobolev embeddings for $\bH_{p,0}^{\alpha,1}$}

\begin{theorem}[Embedding with $\alpha$-time derivative and $1$-spatial derivatives for $p < \min\{1/\alpha,d\}$]
							\label{thm1204_1}
Let $\alpha \in (0, 1)$ and $p, q \in (1,\infty)$ satisfy
$$
p < \min\{1/\alpha, d\}, \quad p < q < q^* := \frac{1/\alpha + d}{1/(\alpha p) + d/p - 1}.
$$
Then
\begin{equation}
							\label{eq0213_01}
\|\psi\|_{L_q(\bR^d_T)} \le N  \|D_x \psi\|_{L_p(\bR^d_T)}^{\theta} \|\partial_t^\alpha \psi \|_{L_p(\bR^d_T)}^{\tau(1-\theta)} \|\psi\|_{L_p(\bR^d_T)}^{(1-\tau)(1-\theta)}
\end{equation}
for $\psi \in \bH_{p,0}^{\alpha,1}(\bR^d_T)$, where
$$
\theta = d \left(\frac{1}{p} - \frac{1}{q}\right) \in (0,1), \quad \tau = \frac{1}{\alpha d} \frac{\theta}{1-\theta} \in (0,1),
$$
and
$N = N(d,\alpha,p,q)$, but independent of $T$.
If $q = q^*$, then
\begin{equation}
							\label{eq0213_02}
\|\psi\|_{L_q(\bR^d_T)} \le N  \|D_x \psi\|_{L_p(\bR^d_T)}^{\alpha d/(1+\alpha d)} \|\partial_t^\alpha \psi \|_{L_p(\bR^d_T)}^{1/(1+\alpha d)}.
\end{equation}
\end{theorem}

\begin{proof}
By the definition of $\bH_{p,0}^{\alpha,1}(\bR^d_T)$, we may assume that $\psi \in C_0^\infty\left([0,T] \times \bR^d \right)$ and $\psi(0,x) = 0$.
Using the Sobolev embedding in $x$, we have
\begin{equation}
                                    \label{eq11.55}
\|\psi\|_{L_p((0,T);L_{pd/(d-p)}(\bR^d))} \le N\|D_x \psi\|_{L_p(\bR^d_T)}.
\end{equation}
By \cite[Lemma A.7]{arXiv:1806.02635} with $\theta = 1$, we have
$$
\|\psi\|_{L_p(\bR^d;L_{p/(1-\alpha p)}((0,T)))} \le N\|\partial_t^\alpha \psi\|_{L_p(\bR^d_T)},
$$
which together with the Minkowski inequality implies that
\begin{equation}
                               \label{eq11.56}
\|\psi\|_{L_{p/(1-\alpha p)}((0,T);L_p(\bR^d))}
= \left\| \int_{\bR^d} |\psi(\cdot,x)|^p \, dx \right\|_{L_{\frac{1}{1-\alpha p}}(0,T)}^{\frac{1}{p}}
\le N\|\partial_t^\alpha \psi\|_{L_p(\bR^d_T)}.
\end{equation}
When $q = q^*$, by H\"older's inequality it follows that
$$
\|\psi\|_{L_q(\bR^d_T)} \leq \left( \int_0^T \|\psi(t,\cdot)\|_{L_p(\bR^d)}^{\frac{p}{1+\alpha d - \alpha p}} \|\psi(t,\cdot)\|_{L_{\frac{pd}{d-p}}(\bR^d)}^{\frac{\alpha d p}{1+\alpha d - \alpha p}} \, dt \right)^{1/q}
$$
$$
\leq \left( \int_0^T \|\psi(t,\cdot)\|_{L_p(\bR^d)}^{\frac{p}{1-\alpha p}} \, dt \right)^{\frac{1-\alpha p}{p + \alpha p d}} \left( \int_0^T \|\psi(t,\cdot)\|_{L_{\frac{pd}{d-p}}(\bR^d)}^p \, dt \right)^{\frac{\alpha d}{p+ \alpha p d}}
$$
$$
= \|\psi\|_{L_{p/(1-\alpha p)}\left((0,T); L_p(\bR^d)\right)}^{\frac{1}{1+\alpha d}} \|\psi\|_{L_p\left((0,T); L_{pd/(d-p)}(\bR^d)\right)}^{\frac{\alpha d}{1+\alpha d}}.
$$
This along with \eqref{eq11.55} and \eqref{eq11.56} proves \eqref{eq0213_02}.
The inequality \eqref{eq0213_01} follows from  \eqref{eq0213_02} and H\"older's inequality.
\end{proof}

From Theorem \ref{thm1204_1} the following corollary follows easily.

\begin{corollary}
							\label{cor1211_1}
Let $\alpha \in (0, 1)$ and $p, q \in (1,\infty)$ satisfy
$$
p < \min\{1/\alpha,d\}, \quad p<q \leq q^* := \frac{1/\alpha + d}{1/(\alpha p) + d/p - 1}.
$$
Then we have
$$
\|\psi\|_{L_q\left((0,T) \times B_1\right)} \le N \|\psi\|_{\bH_p^{\alpha,1}\left((0,T) \times B_1\right)}
$$
for any $\psi \in \bH_{p,0}^{\alpha,1}\left((0,T) \times B_1\right)$, where
$N = N(d,\alpha,p,q)$, but independent of $T$.
If $p \leq d$ and $p \leq 1/\alpha$, then the same estimate holds for $q\in [1,q^*)$ with $N$ depending also on $T$.
\end{corollary}

\begin{proof}
If $p < d$ and $p < 1/\alpha$, the result follows easily from Theorem \ref{thm1204_1} with an extension of $\psi$ to a function in $\bH_{p,0}^{\alpha,1}(\bR^d_T)$ with a comparable norm.
If $p = d$ or $p = 1/\alpha$, then find $\varepsilon > 0$ such that
$$
q \leq \frac{1/\alpha + d}{1/(\alpha (p-\varepsilon)) + d/(p-\varepsilon) - 1} < \frac{1/\alpha + d}{1/(\alpha p) + d/p - 1}.
$$
Then
$$
\|\psi\|_{L_q\left((0,T) \times B_1\right)} \leq N \|\psi\|_{\bH_{p-\varepsilon}^{\alpha,1}\left((0,T) \times B_1\right)} \leq N \|\psi\|_{\bH_p^{\alpha,1}\left((0,T) \times B_1\right)}.
$$
The corollary is proved.
\end{proof}

\begin{theorem}[Embedding with $\alpha$-time derivative and $1$-spatial derivatives for $d < p < 1/\alpha$]
							\label{thm1207_2}
Let $\alpha \in (0, 1)$ and $p, q \in (1,\infty)$ satisfy
$$
d < p < \frac{1}{\alpha}, \quad p < q \leq p(\alpha p + 1).
$$
Then, for $\psi \in \bH_{p,0}^{\alpha, 1}\left((0,T) \times B_1\right)$, we have
\begin{equation}
							\label{eq0213_04}
\|\psi\|_{L_q((0,T)\times B_1)} \le N \left( \sum_{0 \leq |\beta| \leq 1} \|D_x^\beta \psi\|_{L_p((0,T)\times B_1)} \right)^{1-\theta} \|\partial_t^\alpha\psi\|_{L_p((0,T)\times B_1)}^\theta,
\end{equation}
where $N = N(d,\alpha,p,q)$, but independent of $T$, and
$$
\theta = \frac{1}{\alpha} \left( \frac{1}{p} - \frac{1}{q} \right) \in (0,1).
$$
If $d < p \leq 1/\alpha$, then the same estimate holds for $q$ satisfying
$$
1 \leq q < p(\alpha p + 1)
$$
with $N$ depending also on $T$.
\end{theorem}

\begin{proof}
As above, we assume that $\psi \in  \bH_{p,0}^{\alpha,1}\left((0,T) \times B_1\right) \cap C^\infty\big([0,T] \times B_1\big)$ and $\psi(0,x) = 0$.
We first assume that $p < 1/\alpha$.
Since $p>d$, by the Sobolev embedding in $x$, we have
\begin{equation}
                                    \label{eq12.01}
\|\psi\|_{L_p((0,T);L_{\infty}(B_1))} \le N\left( \sum_{0 \leq |\beta| \leq 1} \|D_x^\beta \psi\|_{L_p((0,T)\times B_1)} \right).
\end{equation}
By \cite[Lemma A.7]{arXiv:1806.02635} with $\theta = 1$ and the Minkowski inequality, we have
\begin{multline}
                               \label{eq11.56b}
\|\psi\|_{L_{p/(1-\alpha p)}((0,T);L_p(B_1))} = \left\| \int_{B_1} |\psi(\cdot,x)|^p \, dx \right\|_{L_{1/(1-\alpha p)}((0,T))}^{1/p}
\\
\leq \|\psi\|_{L_p(B_1; L_{p/(1-\alpha p)}((0,T)))} \le N\|\partial_t^\alpha \psi\|_{L_p((0,T)\times B_1)}.
\end{multline}
Then for $q = p(\alpha p + 1)$, by H\"older's inequality it follows that
$$
\|\psi\|_{L_q(B_1)} \leq \left( \int_0^T \|\psi(t,\cdot)\|_{L_\infty(B_1)}^{q-p} \|\psi(t,\cdot)\|_{L_p(B_1)}^p \, dt \right)^{1/q}
$$
$$
\leq \left( \int_0^T \|\psi(t,\cdot)\|_{L_\infty(B_1)}^p \, dt \right)^{\frac{\alpha}{\alpha p + 1}} \left( \int_0^T \|\psi(t,\cdot)\|_{L_p(B_1)}^{\frac{p}{1-\alpha p}} \, dt \right)^{\frac{1-\alpha p}{p(\alpha p + 1)}}
$$
$$
= \|\psi\|_{L_p((0,T);L_{\infty}(B_1))}^{\frac{\alpha p}{\alpha p + 1}} \|\psi\|_{L_{p/(1-\alpha p)}\left((0,T); L_p(B_1)\right)}^{\frac{1}{\alpha p + 1}} .
$$
From this, \eqref{eq12.01}, and \eqref{eq11.56b}, we get \eqref{eq0213_04} with $q=p(\alpha p+1)$ and $\theta=1/(\alpha p+1)$. The remaining cases then follow from  H\"older's inequality.
\end{proof}

\begin{theorem}[Embedding with $\alpha$-time derivative and $1$-spatial derivatives for $1/\alpha < p <  d$]
Let $\alpha \in (0,1)$ and $p, q \in (1,\infty)$ such that
\begin{equation*}
\frac{1}{\alpha} < p < d, \quad p < q \leq p + \frac{p^2}{d}.
\end{equation*}
Then, for $\psi \in \bH_{p,0}^{\alpha,1}(\bR^d_T)$,
$$
\|\psi\|_{L_q(\bR^d_T)} \le N T^{\alpha\left(1-\frac{p}{q}\right)-\frac{1}{p}+\frac{1}{q}} \|\partial_t^\alpha \psi\|_{L_p(\bR^d_T)}^{1-p/q} \|D_x \psi\|_{L_p(\bR^d_T)}^{\theta p/q} \|\psi\|_{L_p(\bR^d_T)}^{(1-\theta)p/q},
$$
where $N=N(d,\alpha,p,q)$ and
$
\theta = d(q-p)/p^2 \in (0,1]$.
\end{theorem}
\begin{proof}
As above, we assume that $\psi \in C_0^\infty\left([0,T] \times \bR^d \right)$ and $\psi(0,x) = 0$.
Since $\alpha>1/p$, by \cite[Lemma A.6]{arXiv:1806.02635} and the Minkowski inequality, we have
\begin{equation}
                                    \label{eq12.14}
\|\psi\|_{L_{\infty}((0,T);L_p(\bR^d))}
\le \|\psi\|_{L_p(\bR^d;L_{\infty}((0,T)))}\le NT^{\alpha-1/p}\|\partial_t^\alpha \psi\|_{L_p(\bR^d_T)}.
\end{equation}
By the Sobolev embedding in $x$, we have
\begin{equation}
                                    \label{eq12.19}
\|\psi\|_{L_p((0,T);L_{dp/(d-p)}(\bR^{d}))}\le N\|D_x \psi\|_{L_p(\bR^d_T)}.
\end{equation}
By \eqref{eq12.14}, \eqref{eq12.19}, and H\"older's inequality, we get the desired estimate with $q=p+p^2/d$. The general case then follows from H\"older's inequality.
\end{proof}

By extending $\psi \in \bH_{p,0}^{\alpha,1}\left((0,T) \times B_1 \right)$ to a function in $\bH_{p,0}^{\alpha,1}(\bR^d_T)$ with a comparable norm and using the above theorem, we get

\begin{corollary}
							\label{cor0225_1}
Let $\alpha \in (0,1)$ and $p, q \in (1,\infty)$ such that
\begin{equation*}
\frac{1}{\alpha} < p < d, \quad p < q \leq p + \frac{p^2}{d}.
\end{equation*}
Then, for $\psi \in \bH_{p,0}^{\alpha,1}\left((0,T) \times B_1\right)$,
$$
\|\psi\|_{L_q\left((0,T) \times B_1\right)} \le N T^{\alpha\left(1-\frac{p}{q}\right)-\frac{1}{p}+\frac{1}{q}} \|\psi\|_{\bH_p^{\alpha,1}\left((0,T) \times B_1\right)},
$$
where $N=N(d,\alpha,p,q)$.
If $1/\alpha < p \leq d$, the same estimate holds for $q$ satisfying
$$
1 \leq q < p + p^2/d
$$
with $N$ depending also on $T$.
\end{corollary}

\begin{theorem}[Embedding with $\alpha$-time derivative and $1$-spatial derivatives for $\max\{1/\alpha ,d\} < p \leq d + 1/\alpha$]
							\label{thm0214_1}
Let $\alpha \in (0,1)$ and $p, q \in (1,\infty)$ such that
\begin{equation*}
\max\{1/\alpha ,d\} < p \leq d + 1/\alpha, \quad p < q \le 2p.
\end{equation*}
Then, for $\psi \in \bH_{p,0}^{\alpha,1}\left((0,T) \times B_1\right)$,
\begin{align*}
&\|\psi\|_{L_q\left((0,T) \times B_1\right)} \\
&\le N T^{\frac{\alpha p}{q} - \frac{1}{p} + \frac{1}{q}} \left( \sum_{0 \leq |\beta|\leq 1}\|D_x^\beta \psi\|_{L_p\left((0,T) \times B_1\right)}\right)^{1-\theta} \|\partial_t^\alpha \psi\|_{L_p\left((0,T) \times B_1\right)}^\theta,
\end{align*}
where $N=N(d,\alpha,p,q)$ and
$
\theta = p/q \in (0,1)$.
\end{theorem}

\begin{proof}
Again we assume that
$$
\psi \in \bH_{p,0}^{\alpha,1}\left((0,T)\times B_1\right) \cap C^\infty\left([0,T] \times B_1 \right)\quad
\text{and}\quad\psi(0,x) = 0.
$$
We set $q':=p^2/(2p-q)\in (p,\infty]$. Since $\alpha - 1/p > 0$, from \cite[Lemma A.6]{arXiv:1806.02635} and the Minkowski inequality,
$$
\|\psi\|_{L_{q'}((0,T);L_p(B_1))}
\le \|\psi\|_{L_p(B_1;L_{q'}((0,T)))}
\le N T^{\alpha - 1/p + 1/q'} \|\partial_t^\alpha \psi\|_{L_p((0,T)\times B_1)},
$$
where $N = N(\alpha,p,q')$.
This, \eqref{eq12.01}, and H\"older's inequality yield the desired inequality.
\end{proof}

Recall that
$$
Q_R(t,x) =Q_{R,R}(t,x) = (t-R^{2/\alpha}, t) \times B_R(x).
$$
For the H\"{o}lder semi-norm, we denote
$$
[u]_{C^{\sigma_1, \sigma_2}(\cD)} = \sup_{\substack{(t,x),(s,y) \in \cD \\ (t,x) \neq (s,y)}}\frac{|u(t,x) - u(s,y)|}{|t-s|^{\sigma_1} + |x-y|^{\sigma_2}},
$$
where $\cD \subset \bR \times \bR^d$.

\begin{theorem}[Embedding with $\alpha$-time derivative and $1$-spatial derivatives for $p \in (d+1/\alpha,\infty)$]
							\label{thm1029_1}
Let $\alpha \in (0,1)$ and $p \in (1,\infty)$ such that
$$
\sigma := 1-(d+1/\alpha)/p \in (0,1).
$$
Then, for $\psi \in \bH_{p,0}^{\alpha,1}\left((0,1) \times \bR^d\right)$, we have
$$
[\psi]_{C^{\sigma \alpha, \sigma}\left((0,1) \times \bR^d\right)} \leq N(d,\alpha,p) \|\psi\|_{\bH_p^{\alpha,1}\left((0,1) \times \bR^d\right)}.
$$
\end{theorem}

\begin{proof}
Define
$$
K = \sup_{\substack{(t,x),(s,y) \in (0,1)\times \bR^d \\ (t,x) \neq (s,y)}}\frac{|\psi(t,x) - \psi(s,y)|}{|t-s|^{\sigma \alpha} + |x-y|^\sigma}.
$$
To prove the estimate, we take $(t_1,x), (t_2,y) \in (0,1) \times \bR^d$ and set
$$
\rho = \varepsilon \left( |t_1-t_2|^{\alpha} + |x-y| \right),
$$
where $\varepsilon \in (0,1)$ is to be specified below.
We write
$$
|\psi(t_1,x) - \psi(t_2,y)| \leq |\psi(t_1,x) - \psi(t_2,x)| + |\psi(t_2,x) - \psi(t_2,y)| := J_1 + J_2.
$$
To estimate $J_1$, for $z \in B_\rho(x)$, we have
$$
J_1 \leq |\psi(t_1,x)-\psi(t_1,z)| + |\psi(t_1,z) - \psi(t_2,z)| + |\psi(t_2,z) -\psi(t_2,x)|
$$
$$
\leq 2 K \rho^\sigma + |\psi(t_1,z) - \psi(t_2,z)|,
$$
where by \cite[Lemma A.14]{arXiv:1806.02635} we see that
\begin{equation}
							\label{eq0218_04}
|\psi(t_1,z) - \psi(t_2,z)| \le N(\alpha,p)|t_1-t_2|^{\alpha-1/p} \|\partial_t^\alpha \psi(\cdot,z)\|_{L_p(0,1)}.
\end{equation}
Then by taking the average of $J_1$ over $B_\rho(x)$ with respect to $z$ along with H\"{o}lder's inequality and using \eqref{eq0218_04}, we get
$$
J_1 \leq 2K \rho^\sigma + N |t_1-t_2|^{\alpha-1/p} \rho^{-d/p} \| \partial_t^\alpha \psi\|_{L_p\left((0,1) \times \bR^d\right)}
$$
$$
\leq 2 K \rho^\sigma + N \varepsilon^{-1+1/(\alpha p)} \rho^\sigma \| \partial_t^\alpha \psi\|_{L_p\left((0,1) \times \bR^d\right)},
$$
where $N = N(d,\alpha,p)$.

To estimate $J_2$, we have that, for $s \in (t_2-\rho^{1/\alpha},t_2+\rho^{1/\alpha}) \cap (0,1)$,
$$
J_2 \leq |\psi(t_2,x) - \psi(s,x)| + |\psi(s,x) - \psi(s,y)| + |\psi(s,y) - \psi(t_2, y)|
$$
$$
\leq
2 K \rho^\sigma + N(d,p)|x-y|^{1-d/p} \|\psi(s,\cdot)\|_{W_p^1(\bR^d)},
$$
where we used the usual Sobolev embedding for functions in $x \in \bR^d$ and the condition that $1-d/p \in (0,1)$.
Then by taking the average of $J_2$ over the interval $(t_2 - \rho^{1/\alpha}, t_2+\rho^{1/\alpha}) \cap (0,1)$ with respect to $s$ along with H\"{o}lder's inequality, we get
$$
J_2 \leq 2K \rho^\sigma + N \varepsilon^{-1+d/p} \rho^\sigma \||\psi|+|D_x \psi|\|_{L_p\left((0,1) \times \bR^d \right)}.
$$

Collecting the estimates for $J_1$ and $J_2$ above, we see that
$$
|\psi(t_1,x) - \psi(t_2,y)| \leq 4 K \rho^\sigma + N \left(\varepsilon^{-1+1/(\alpha p)} +  \varepsilon^{-1+d/p}  \right) \rho^\sigma \|\psi\|_{\bH_p^{\alpha,1} \left((0,1) \times \bR^d \right)},
$$
which implies that
$$
K \leq 4 \varepsilon^\sigma K + N(d,\alpha,p) (\varepsilon^{-d/p} +  \varepsilon^{-1/(\alpha p)} )\|\psi\|_{\bH_p^{\alpha,1} \left((0,1) \times \bR^d \right)}.
$$
We finish the proof by choosing $\varepsilon > 0$ small enough so that
$4 \varepsilon^\sigma < 1$.
\end{proof}

\begin{corollary}
                                        \label{cor1029_1}
Let $\alpha \in (0,1)$ and $p \in (1,\infty)$ such that
$$
\sigma := 1-(d+1/\alpha)/p \in (0,1).
$$
Then, for $\bH_{p,0}^{\alpha,1}\left((0,1) \times B_1\right)$, we have
\begin{equation*}
[\psi]_{C^{\sigma \alpha, \sigma}\left((0,1) \times B_1\right)} \leq N(d,\alpha,p) \|\psi\|_{\bH_p^{\alpha,1}\left((0,1) \times B_1\right)}.
\end{equation*}
\end{corollary}

\begin{proof}
The corollary follows from Theorem \ref{thm1029_1} with an extension of $\psi$ to a function in $\bH_{p,0}^{\alpha,1}\left((0,1)\times \bR^d\right)$ with a comparable norm.
\end{proof}

\bibliographystyle{plain}

\def\cprime{$'$}

\end{document}